\documentclass[10pt]{amsart}
\usepackage{geometry} 
\newtheorem{theorem}{Theorem}[section]

\newtheorem{lemma}[theorem]{Lemma} 
\newtheorem{proposition}[theorem]{Proposition}
\newtheorem{remark}[theorem]{Remark}
\makeatletter
\numberwithin{equation}{section}
\makeatother

\usepackage{amsmath,hhline,epsfig,latexsym}
\usepackage{amssymb}
\usepackage{hyperref,color}

\newcommand{\eps}{\varepsilon}

\def\A{\mathcal A}

\def\C{\mathcal C}

\def\A{\mathcal A}

\title[Space-time least-squares methods for Navier-Stokes]{A fully space-time least-squares method for the unsteady Navier-Stokes system }

\date{11-09-2019}
\author{J\'er\^ome Lemoine}
\thanks{
Laboratoire de Math\'{e}matiques Blaise Pascal, Universit\'{e} Clermont Auvergne, UMR CNRS 6620, Campus des C\'ezeaux,  63177~Aubi\`ere, France. e-mail: {\tt jerome.lemoine@uca.fr}. }
\author{Arnaud M\"unch}
\thanks{
Laboratoire de Math\'{e}matiques Blaise Pascal, Universit\'{e} Clermont Auvergne, UMR CNRS 6620, Campus des C\'ezeaux,  63177~Aubi\`ere, France. e-mail: {\tt arnaud.munch@uca.fr} (\textbf{Corresponding author}).}

\begin{document}

\maketitle

\begin{abstract}
We introduce and analyze a space-time least-squares method associated to the unsteady Navier-Stokes system. Weak solution in the two dimensional case and regular solution in the three dimensional case are considered. From any initial guess, we construct a minimizing sequence for the least-squares functional which converges strongly to a solution of the Navier-Stokes system. After a finite number of iterates related to the value of the viscosity constant, the convergence is quadratic. Numerical experiments within the two dimensional case support our analysis. This globally convergent least-squares approach is related to the damped Newton method when used to solve the Navier-Stokes system through a variational formulation. 
\end{abstract}




\section{Introduction}

\label{intro}

Let $\Omega\subset \mathbb{R}^d$, $d=2,3$ be a bounded connected open set whose boundary $\partial\Omega$ is Lipschitz.
   We denote by $\boldsymbol{\mathcal{V}}=\{v\in \mathcal{D}(\Omega)^d, \nabla\cdot v=0\}$, $\boldsymbol{H}$ the closure of $\boldsymbol{\mathcal{V}}$ in $L^2(\Omega)^d$ and $\boldsymbol{V}$ the closure of $\boldsymbol{\mathcal{V}}$ in $H^1(\Omega)^d$.  Endowed with the norm $\|v\|_{\boldsymbol{V}}=\|\nabla v\|_2:=\|\nabla v\|_{(L^2(\Omega))^{d^2}}$, $\boldsymbol{V}$ is an Hilbert space. The dual $\boldsymbol{V}'$ of $\boldsymbol{V}$, endowed with  the dual norm
   $$\|v\|_{\boldsymbol{V}'}=\sup_{w\in \boldsymbol{V},\ \|w\|_{\boldsymbol{V}}=1}\langle v,w\rangle_{\boldsymbol{V}'\times \boldsymbol{V}}$$
   is also an Hilbert space. We denote $\langle\cdot,\cdot\rangle_{\boldsymbol{V}'}$ the scalar product associated to the norm $\| \ \|_{\boldsymbol{V}'}$.
   
   Let $T>0$. We note $Q_T:=\Omega\times (0,T)$ and $\Sigma_T:=\partial\Omega\times [0,T]$.
   
The Navier-Stokes system describes a viscous incompressible fluid flow in the bounded domain $\Omega$ during the time interval $(0,T)$ submitted to the external force $f$. It reads as follows : 
\begin{equation}
\label{NS-intro}
\left\{
\begin{aligned}
 &     y_t- \nu\Delta y + (y\cdot \nabla)y + \nabla p = f, \quad  \nabla\cdot y=0  \quad \text{in}\ \ Q_T, \\
  &   y = 0 \quad \text{on}\ \ \Sigma_T,  \\
  & y(\cdot,0)=u_0, \quad \text{in}\quad  \Omega,
\end{aligned}
\right.
\end{equation}
where $y$ is the velocity of the fluid, $p$ its pressure and $\nu$ is the viscosity constant. We refer to \cite{Lions69, Simon, Temam}. 

In the case $d=2$, we recall (see \cite{Temam}) that for $f\in L^2(0,T,\boldsymbol{V}^\prime)$ and $u_0\in \boldsymbol{H}$, there exists a unique weak solution $y\in L^2(0,T; \boldsymbol{V})$, $\partial_t y\in L^2(0,T; \boldsymbol{V}')$ of the system
\begin{equation}\label{NS}
\left\{
\begin{aligned}
 &  
\frac{d}{dt}\int_\Omega y\cdot w+\nu\int_\Omega \nabla y\cdot\nabla w+\int_\Omega y\cdot\nabla y\cdot w=\langle f,w\rangle_{\boldsymbol{V}'\times \boldsymbol{V}},\quad \forall w\in \boldsymbol{V} \\
& y(\cdot,0)=u_0, \quad \text{in}\quad  \Omega.
\end{aligned}
\right.
\end{equation}

This work is concerned with the approximation of solution for (\ref{NS}), that is, the explicit construction of a sequence $(y_k)_{k\in\mathbb{N}}$ converging to a solution $y$ for a suitable norm. In most of the works devoted to this topic (we refer for instance to \cite{GLOWINSKI2003, pironneau}), the approximation of (\ref{NS}) is addressed through a time marching method. Given $\{t_n\}_{n=0...N}$, $N\in \mathbb{N}$, a uniform discretization of the time interval $(0,T)$ and $\delta t=T/N$ the corresponding time discretization step, we mention for instance the unconditionally stable backward Euler scheme
\begin{equation}\label{NS_euler}
\left\{
\begin{aligned}
 &  
\int_\Omega \frac{y^{n+1}-y^n}{\delta t}\cdot w+\nu\int_\Omega \nabla y^{n+1}\cdot\nabla w+\int_\Omega y^{n+1}\cdot\nabla y^{n+1}\cdot w=\langle f^n,w\rangle_{\boldsymbol{V}'\times \boldsymbol{V}},\, \forall n\geq 0, \, \forall w\in \boldsymbol{V} \\
& y^0(\cdot,0)=u_0, \quad \text{in}\quad  \Omega
\end{aligned}
\right.
\end{equation}
with $f^n:=\frac{1}{\delta t}\int_{t_n}^{t_{n+1}}f(\cdot,s)ds$. The piecewise linear interpolation (in time) of $\{y^n\}_{n\in [0,N]}$ weakly converges in $L^2(0,T,\boldsymbol{V})$ toward a solution $y$ of (\ref{NS}) as $\delta t$ goes to zero (we refer to \cite[chapter 3, section 4]{Temam}). Moreover, it achieves a first order convergence with respect to $\delta t$. 
For each $n\geq 0$, the determination of $y^{n+1}$ from $y^n$ requires the resolution of a steady Navier-Stokes equation, parametrized by $\nu$ and $\delta t$. This can be done using Newton type methods (see for instance \cite[Section 10.3]{quarteroni}) for the weak formulation of (\ref{NS_euler}). Alternatively, this can be done using least-squares method
which consists roughly in minimizing (the square of) a norm of the state equation with respect to $y^{n+1}$. We refer to \cite{bristeau,Hmoinsun-glo} where a so-called $H^{-1}(\Omega)$ least-squares method has been introduced, and recently analyzed and extended in \cite{lemoinemunch}. In \cite{lemoinemunch}, the convergence of the method is proved and leads in practice to the so-called damped Newton method, more robust and faster than the usual one. 

The main reason of this work is to explore if the analysis performed in \cite{lemoinemunch,lemoine2018} for the steady Navier-Stokes system can be extended to a full space-time setting. More precisely, following the terminology of \cite{bristeau}, one may introduced the following $L^2(0,T;\boldsymbol{V}^\prime)$ least-squares functional $\widetilde{E}:(H^1(0,T;\boldsymbol{V^\prime})\cap L^2(0,T;\boldsymbol{V})) \to \mathbb{R}^+$ 
\begin{equation}
\widetilde{E}(y):=\frac{1}{2}\Vert y_t+\nu B_1(y)+ B(y,y) -f\Vert^2_{L^2(0,T; \boldsymbol{V}^\prime)} \label{Etilde}
\end{equation}
where $B_1$ and $B$ are defined in Lemmas  \ref{B-V'}  and  \ref{B_1-V'}.
The real quantity $\widetilde{E}(y)$ measures how the element $y$ is close to the solution of (\ref{NS}). The minimization of this functional leads to a so-called continuous  weak least-squares type method. Least-squares methods to solve nonlinear boundary value problems have been the subject of intensive developments in the last decades, as they present several advantages, notably on computational and stability viewpoints. We refer to the book \cite{bochevgunz} devoted to the analysis of least-squares methods to solve discrete finite dimensional systems. We notably mention \cite{Chehab} where steady fluids flows are approximated in the two dimensional case of the lid-driven cavity. We show in the present work that some minimizing sequences for this so-called error functional $\widetilde{E}$ do actually converge strongly to the solution of (\ref{NS}). 

This approach which consist in minimizing an appropriate norm of the solution is refereed to in the literature as variational approach. We mention notably the work \cite{PedregalNS} where strong solution of \eqref{NS-intro} are characterized in the two dimensional case in term of the critical points of a quadratic functional, close to $\widetilde{E}$. Similarly, the authors in \cite{OrtizNS} show that the following functional 
$$
I^{\epsilon}(y)=\int_0^\infty\!\!\!\int_{\Omega} e^{-t/\epsilon} \biggl\{\vert \partial_t y+y\cdot \nabla y\vert^2+\vert y\cdot \nabla y\vert^2 +\frac{\nu}{\epsilon}\vert \nabla y\vert^2\biggl\}
$$
admits minimizers $u^{\eps}$ for all $\epsilon>0$ and, up to subsequences, such minimizers converge weakly to a Leray-Hopf solution of (\ref{NS-intro}) as $\epsilon\to 0$.
  
  The paper is organized as follows. We start in Section \ref{section-2D} in the two dimensional case with the weak solution of \eqref{NS} associated to initial data $u_0$ in $\boldsymbol{H}$ and source term $f\in L^2(0,T,\boldsymbol{V}^{\prime})$. We introduce our least-squares functional, quoted by $E$, in term of a corrector variable $v$. We show in two steps that any minimizing sequence for $E$ strongly converges to the solution (see Theorem \ref{th-convergence}). This is achieved in two steps: first, we obtain a coercivity type property which show that $E(y)$
is an upper  bound of the distance of $y$ to a solution of (\ref{NS}). Then, we introduce a bounded element $Y_1$ in (\ref{solution_Y1}) along which the differential $E^{\prime}$ of $E$ is parallel to $E$ (see \eqref{E_Eprime}). The use of the element $-Y_1$ as a descent direction allows to define iteratively a minimizing sequence $\{y_{k}\}_{k>0}$ which converges with a quadratic rate (except for the first iterates) to the solution of (\ref{NS}). It turns out that the underlying algorithm (\ref{algo_LS_Y}) coincides with the one derived from the damped Newton method when used to find solution of (\ref{NS}). In Section \ref{section-3D}, in the three dimensional case, we employ the same methodology to approximate regular solution of (\ref{NS}) associated to $u_0\in \boldsymbol{V}$ and $f\in L^2(Q_T)^3$. We obtain similar results of convergence. Numerical experiments in Section \ref{section-numeric} confirm the efficiency of the method based on the element $Y_1$, in particular for small values of the viscosity constant $\nu$.  Section \ref{section-conclusion} concludes with some perspectives.

  
\section{Space-time least squares method: the two dimensional case}\label{section-2D}

Adapting \cite{lemoinemunch}, we introduce and analyze a so-called weak least-squares functional allowing to approximate the solution of the boundary value problem (\ref{NS}).

\subsection{Preliminary technical results}
In the following, we repeatedly use the following classical estimate.

\begin{lemma} Let any $u\in  \boldsymbol{H}$, $v,w\in \boldsymbol{V}$. There exists a constant $c=c(\Omega)$ such that
\begin{equation}
\label{2D-Case}
\int_{\Omega} u\cdot \nabla v\cdot w \leq c \Vert u\Vert_{\boldsymbol{H}}  \Vert  v\Vert_{ \boldsymbol{V}}  \Vert  w\Vert_{ \boldsymbol{V}}.    
\end{equation}
\end{lemma}
\begin{proof} 
If $u\in  \boldsymbol{H}$, $v,w\in \boldsymbol{V}$,  denoting $\tilde u, \tilde v$ and $\tilde w$ their extension to $0$ in $\mathbb{R}^2$, we have, see \cite{coifman} 
and \cite{tartar}
$$
\begin{aligned}
\biggl|\int_\Omega u\cdot\nabla v\cdot w\biggr|=\biggl|\int_\Omega \tilde u\cdot\nabla \tilde v\cdot \tilde w\biggr|
&\le \|\tilde u\cdot\nabla\tilde v\|_{{\mathcal H}^1(\mathbb{R}^2)}\|\tilde w\|_{BMO(\mathbb{R}^2)}\le c\| \tilde u\|_2\|\nabla \tilde v\|_2\|\tilde w\|_{H^1(\mathbb{R}^2)}\\
&\le c\| u\|_2\|\nabla v\|_2\|w\|_{H^1(\Omega)^2} \le c\| u\|_{\boldsymbol{H}}\| v\|_{ \boldsymbol{V}}\| w\|_{ \boldsymbol{V}}.
\end{aligned}
$$
\end{proof} 

\begin{lemma}\label{B-V'}
Let any $u\in L^\infty(0,T;\boldsymbol{H})$ and $v\in L^2(0,T;\boldsymbol{V})$. Then the function $B(u,v)$ defined by 
$$
\langle B(u(t),v(t)),w\rangle =\int_{\Omega} u(t)\cdot \nabla v(t)\cdot w   \qquad\forall w\in \boldsymbol{V},\ \hbox{a.e in}\  t\in[0,T]
$$
belongs to $L^2(0,T;\boldsymbol{V}')$ and
\begin{equation}\label{Estim-B}
\|B(u,v)\|_{L^2(0,T;\boldsymbol{V}')}\le c\Big(\int_0^T\|u\|^2_{\boldsymbol{H}}\|v\|^2_{\boldsymbol{V}}\Big)^{\frac12}\le c\|u\|_{L^\infty(0,T;\boldsymbol{H})}\|v\|_{L^2(0,T;\boldsymbol{V})}.
\end{equation}
Moreover 
\begin{equation}\label{B-nul}
\langle B(u,v),v\rangle_{\boldsymbol{V}'\times \boldsymbol{V}}=0.
\end{equation}
\end{lemma}
\begin{proof} 
Indeed, a.e in $t\in[0,T]$ we have (see (\ref{2D-Case})), $\forall w\in \boldsymbol{V}$
$$
|\langle B(u(t),v(t)),w\rangle|\le c\|u(t)\|_{ \boldsymbol{H}}\|v(t)\|_{ \boldsymbol{V}}\|w\|_{\boldsymbol{V}}
$$
and thus, 
$$\int_0^T\|B(u,v)\|^2_{\boldsymbol{V}'}\le c\int_0^T\|u\|^2_{\boldsymbol{H}}\|v\|^2_{\boldsymbol{V}}\le  c\|u\|^2_{L^\infty(0,T,\boldsymbol{H})}\|v\|^2 _{L^2(0,T,\boldsymbol{V})}<+\infty.$$

We also have a.e in $t\in[0,T]$ (see \cite{Temam})
$$\langle B(u(t),v(t)),v(t)\rangle_{\boldsymbol{V}'\times \boldsymbol{V}}=\int_\Omega u(t)\cdot\nabla v(t)\cdot v(t)=0.$$
\end{proof} 

\begin{lemma}\label{B_1-V'}
Let any $u\in L^2(0,T;\boldsymbol{V})$. Then the function $B_1(u)$ defined by 
$$
\langle B_1(u(t)),w\rangle =\int_{\Omega} \nabla u(t)\cdot \nabla  w   \qquad\forall w\in \boldsymbol{V},\ \hbox{a.e in}\  t\in[0,T]
$$
belong to $L^2(0,T;\boldsymbol{V}')$ and 
\begin{equation}\label{Estim-B_2}
\Vert B_1(u)\Vert_{L^2(0,T;\boldsymbol{V^\prime})}\le \|u\|^2_{L^2(0,T,\boldsymbol{V})}<+\infty.\end{equation}
\end{lemma}
\begin{proof} 
Indeed, a.e in $t\in[0,T]$ we have
$$
|\langle B_1(u(t)),w\rangle|\le \|\nabla u(t)\|_2\|\nabla w\|_2=\|u(t)\|_{\boldsymbol{V}}\|w\|_{\boldsymbol{V}}
$$
and thus,  a.e in $t\in[0,T]$
$$\|B_1(u(t))\|_{\boldsymbol{V}'}\le \|u(t)\|_{\boldsymbol{V}}$$
which gives (\ref{Estim-B_2}).\end{proof} 

We also have (see \cite{Temam, Lions69}) : 
\begin{lemma}\label{y'-V'}
For all $y\in   L^2(0,T,\boldsymbol{V})\cap H^1(0,T;\boldsymbol{V}')$ we have $y\in \C([0,T];\boldsymbol{H})$ and in  $\mathcal{D}'(0,T)$, for all $w\in \boldsymbol{V}$ :
\begin{equation}\label{partial_t y}
\langle \partial_t y,w\rangle_{\boldsymbol{V}'\times \boldsymbol{V}}=\int_{\Omega}\partial_t y\cdot w=\frac{d}{dt}\int_\Omega y\cdot w, \qquad \langle \partial_t y,y\rangle_{\boldsymbol{V}'\times \boldsymbol{V}}=\frac12\frac{d}{dt}\int_\Omega |y|^2
\end{equation}
and
\begin{equation}\label{estim-y-L-infini}
\|y\|^2_{L^\infty(0,T;\boldsymbol{H})}\le c\|y\|_{L^2(0,T;\boldsymbol{V})}\|\partial_t y\|_{L^2(0,T;\boldsymbol{V}')}.
\end{equation}
\end{lemma}

We recall that along this section, we suppose that $u_0\in \boldsymbol{H}$, $f\in L^2(0,T,\boldsymbol{V}')$ and $\Omega$  is a bounded lipschitz domain of $\mathbb{R}^2$.
We also denote 
$$\A=\{y\in L^2(0,T;\boldsymbol{V})\cap H^1(0,T;\boldsymbol{V}'),\ y(0)=u_0\}$$
and 
$$\A_0=\{y\in L^2(0,T;\boldsymbol{V})\cap H^1(0,T;\boldsymbol{V}'),\ y(0)=0\}.$$
Endowed with the scalar product
$$\langle y,z\rangle_{\A_0}=\int_0^T\langle y,z\rangle_{\boldsymbol{V}}+\langle \partial_t y,\partial_t z\rangle_{\boldsymbol{V}'}$$
and the associated norm
$$\|y\|_{\A_0}=\sqrt{\|y\|^2_{ L^2(0,T;\boldsymbol{V})}+\|\partial_t y\|^2_{L^2(0,T;\boldsymbol{V}')}}$$
$\A_0$ is an Hilbert space.

We also recall and introduce several technical results. The first one is well-known  (we refer to \cite{Lions69} and  \cite{Temam}). 
 \begin{proposition}\label{Existence-NS}
There exists a unique $\bar y\in  \A$  solution in $\mathcal{D}^{\prime}(0,T)$ of (\ref{NS}). This solution satisfies the following estimates :
  $$
  \Vert   \bar y\Vert_{L^\infty(0,T;\boldsymbol{H})}^2+\nu\| \bar y\|_{L^2(0,T;\boldsymbol{V})}^2  \leq  \Vert u_0\Vert^2_{\boldsymbol{H}}+ \frac1\nu\Vert f\Vert^2_{L^2(0,T;\boldsymbol{V}')},
  $$
  $$
  \Vert \partial_t \bar y\Vert_{L^2(0,T;\boldsymbol{V}')}\le \sqrt{\nu}\Vert u_0\Vert_{\boldsymbol{H}}+ 2\Vert f\Vert_{L^2(0,T;\boldsymbol{V}')}+\frac c{\nu^{\frac32}}(\nu\Vert u_0\Vert^2_{\boldsymbol{H}}+ \Vert f\Vert^2_{L^2(0,T;\boldsymbol{V}')}).$$
\end{proposition}

We also introduce the following result :
\begin{proposition}\label{prop_existence_v}
For all $y\in   L^2(0,T,\boldsymbol{V})\cap H^1(0,T;\boldsymbol{V}')$,  there exists  a unique $v\in \A_0$  solution in $\mathcal{D}^{\prime}(0,T)$ of
\begin{equation}\label{corrector}
\left\{
\begin{aligned}
& \frac{d}{dt} \int_\Omega v\cdot w +\int_{\Omega} \nabla v\cdot \nabla w +\frac{d}{dt} \int_\Omega y\cdot w+\nu\int_{\Omega} \nabla y\cdot \nabla w  \\
&\hspace*{6cm}+\int_{\Omega} y\cdot \nabla y\cdot w =<f,w>_{\boldsymbol{V}'\times \boldsymbol{V}}, \quad \forall w\in \boldsymbol{V}\\
&v(0)=0.
\end{aligned}
\right.
\end{equation}
Moreover, for all $t\in [0,T]$, 
$$
\| v(t)\|^2_{\boldsymbol{H}} +\| v\|^2_{L^2(0,t;\boldsymbol{V})}\le  \|f-B(y,y)-\nu B_1y-\partial_t y\|^2_{L^2(0,t;\boldsymbol{V}')}
$$
and
$$
\begin{aligned}
 \|\partial_t v\|_{L^2(0,T;\boldsymbol{V}')}
 &\le \|v\|_{L^2(0,T,\boldsymbol{V})}+\|f-B(y,y)-\nu B_1y-\partial_t y\|_{L^2(0,T; \boldsymbol{V}')}\\
 &\le2\|f-B(y,y)-\nu B_1y-\partial_t y\|_{L^2(0,T; \boldsymbol{V}')}.
 \end{aligned}
$$
\end{proposition}
The proof of this proposition is a consequence of the following standard  result (see \cite{tartar, Lions69}).

\begin{proposition}\label{Existence}
For all $z_0\in \boldsymbol{H}$ and all $F\in L^2(0,T;\boldsymbol{V}')$, there exists a unique $z\in   L^2(0,T,\boldsymbol{V})\cap H^1(0,T;\boldsymbol{V}')$  solution in $\mathcal{D}^\prime(0,T)$ of 
\begin{equation}\label{sol_z}
\left\{
\begin{aligned}
&\frac{d}{dt} \int_\Omega z\cdot w +\int_{\Omega} \nabla z\cdot \nabla w =<F,w>_{\boldsymbol{V}'\times \boldsymbol{V}}, \quad \forall w\in \boldsymbol{V} \\
& z(0)=z_0.
\end{aligned}
\right.
\end{equation}
Moreover, for all $t\in [0,T]$, 
\begin{equation}\label{estimate-z-1}
\|z(t)\|_{\boldsymbol{H}}^2+\|z\|_{L^2(0,t;\boldsymbol{V})}^2  \le  \|F\|^2_{L^2(0,t;\boldsymbol{V}')}+\| z_0\|^2_{\boldsymbol{H}}
\end{equation}
and
\begin{equation}\label{estimate-z-2}
 \|\partial_t z\|_{L^2(0,T;\boldsymbol{V}')}\le\|z\|_{L^2(0,T,\boldsymbol{V})}+\|F\|_{L^2(0,T; \boldsymbol{V}')}\le 2\|F\|_{L^2(0,T; \boldsymbol{V}')}+ \|z_0\|_{\boldsymbol{H}}. 
\end{equation}
\end{proposition}

\begin{proof} (of Proposition \ref{prop_existence_v}) 
Let $y\in  L^2(0,T,\boldsymbol{V})\cap H^1(0,T;\boldsymbol{V}')$. 
Then the functions $B(y,y)$ and $B_1(y)$ defined in $\mathcal{D}'(0,T)$ by 
$$
\langle B(y,y),w\rangle =\int_{\Omega} y\cdot \nabla y\cdot w \quad\hbox{ and } \quad\langle B_1(y),w\rangle =\int_{\Omega} \nabla y\cdot \nabla w,  \qquad\forall w\in \boldsymbol{V}
$$
belong to $L^2(0,T;\boldsymbol{V}')$ (see Lemma  \ref{B-V'}  and  \ref{B_1-V'}). 

Moreover, since $y\in   L^2(0,T,\boldsymbol{V})\cap H^1(0,T,\boldsymbol{V}')$ then, in view of (\ref{partial_t y}), in  $\mathcal{D}'(0,T)$, for all $w\in \boldsymbol{V}$  we have :
$$
\frac{d}{dt}\int_\Omega y\cdot w=\langle \partial_t y,w\rangle_{\boldsymbol{V}'\times \boldsymbol{V}}.
$$
Then (\ref{corrector}) may be rewritten as
$$
\left\{
\begin{aligned}
& \frac{d}{dt} \int_\Omega v\cdot w +\int_{\Omega} \nabla v\cdot \nabla w =<F,w>_{\boldsymbol{V}'\times \boldsymbol{V}}, \quad \forall w\in \boldsymbol{V}\\
& v(0)=0,
\end{aligned}
\right.
$$
where $F=f-B(y,y)-\nu B_1y-\partial_t y\in L^2(0,T,\boldsymbol{V}')$;  Proposition \ref{prop_existence_v} is therefore a consequence of Proposition \ref{Existence}.
\end{proof}

\subsection{The least-squares functional}

We now introduce our least-squares functional $E: H^1(0,T,\boldsymbol{V}')\cap L^2(0,T,\boldsymbol{V})\to\mathbb{R}^+$ by putting
\begin{equation}\label{foncE}
E(y)=\frac{1}{2}\int_0^T\| v\|^2_{\boldsymbol{V}}+\frac{1}{2}\int_0^T \| \partial_t v\|^2_{\boldsymbol{V}'}=\frac12\|v\|^2_{\A_0}
\end{equation}
where the corrector $v$ is the unique solution of (\ref{corrector}). The infimum of $E$ is equal to zero and is reached by a solution of (\ref{NS}). In this sense, the functional $E$ is a so-called error functional which measures, through the corrector variable $v$, the deviation of  $y$ from being a solution of the underlying equation (\ref{NS}).  

Beyond this statement, we would like to argue why we believe it is a good idea to use a (minimization) least-squares approach to approximate the solution of \eqref{NS} by minimizing the functional $E$. Our main result of this section is a follows:

\begin{theorem}\label{th-convergence}
Let $\{y_k\}_{k\in \mathbb{N}}$ be a sequence of $\A$ bounded in $L^2(0,T,\boldsymbol{V})\cap H^1(0,T;\boldsymbol{V}')$. If  $E^{\prime}(y_k)\to 0$ as $k\to \infty$, then the whole sequence $\{y_k\}_{k\in\mathbb{N}}$ converges strongly as $k\to \infty$ in $L^2(0,T,\boldsymbol{V})\cap H^1(0,T;\boldsymbol{V}')$ to the solution $\bar y$ of \eqref{NS}.
\end{theorem}

As in \cite{lemoinemunch}, we divide the proof in two main steps. 
\begin{enumerate}
\item First, we use a typical \textit{a priori} bound to show that leading the error functional $E$ down to zero implies strong convergence to the unique solution of \eqref{NS}. 
\item Next, we show that taking the derivative $E'$ to zero actually suffices to take $E$ to zero. 
\end{enumerate}

Before to prove this result, we mention the following equivalence which justifies the least-squares terminology we have used in the following sense: the minimization of the functional $E$ is equivalent to the minimization of the $L^2(0,T,\boldsymbol{V^\prime})$-norm of the main equation of the Navier-Stokes system. 

\begin{lemma}
There exists $c_1>0$ and $c_2>0$ such that 
$$c_1E(y)\le \| y_t+\nu B_1(y)+ B(y,y) -f\|^2_{L^2(0,T;  \boldsymbol{V}')}\le c_2E(y)$$
for all $y\in L^2(0,T,\boldsymbol{V})\cap H^1(0,T;\boldsymbol{V}^\prime)$. 
\end{lemma}
\begin{proof}
From Proposition \ref{prop_existence_v} we deduce that 
$$
2E(y)=\|v\|^2_{\A_0}\le 5\| y_t+\nu B_1(y)+ B(y,y) -f\|^2_{L^2(0,T;  \boldsymbol{V}')}.
$$
On the other hand, from the definition of $v$,
$$
\begin{aligned}
 &\| y_t+\nu B_1(y)+ B(y,y) -f\|_{L^2(0,T;  \boldsymbol{V}')}=\|v_t+B_1(v) \|_{L^2(0,T;  \boldsymbol{V}')}\\
&\hspace*{1.5cm}\le \|v_t \|_{L^2(0,T;  \boldsymbol{V}')}+\| B_1(v) \|_{L^2(0,T;  \boldsymbol{V}')}\le \sqrt{2} \|v\|_{\A_0}=2\sqrt{E(y)}.
\end{aligned}
$$
\end{proof}

We start with the following proposition which establishes that as we take down the error $E$ to zero, we get closer, in the norm $L^2(0,T; \boldsymbol{V})$  and $H^1(0,T; \boldsymbol{V}')$, to the solution $\bar y$ of the problem (\ref{NS}), and so, it justifies why a promising strategy to find good approximations of the solution of problem \eqref{NS} is to look for global minimizers of \eqref{foncE}. 

\begin{proposition}\label{main-estimatee}
Let $\bar y\in\A $ be the solution of (\ref{NS}), $M\in \mathbb{R}$ such that $\Vert \partial_t \bar y\Vert_{L^{2}(0,T,\boldsymbol{V}')}\leq M$ and $\sqrt{\nu}\Vert \nabla \bar y\Vert_{L^2(Q_T)^4}\leq M$ and let $y\in \A$. If $\Vert \partial_t y\Vert_{L^{2}(0,T,\boldsymbol{V}')}\leq M$ and $\sqrt{\nu}\Vert \nabla y\Vert_{L^2(Q_T)^4}\leq M$, then there exists a constant $c(M)$ such that 
\begin{equation}
\| y-\bar y\|_{L^\infty(0,T;\boldsymbol{H})}  +\sqrt{\nu}\| y-\bar y\|_{L^2(0,T;\boldsymbol{V})}+\|\partial_t y-\partial_t \bar y\|_{L^2(0,T,\boldsymbol{V}')}\leq c(M) \sqrt{E(y)}. \label{coercivity}
\end{equation}
\end{proposition}
\par\noindent
\begin{proof} Let $Y=y-\bar y$.
The functions $B(Y,y)$, $B(\bar y,Y)$ and $B_1(v)$ defined in $\mathcal{D}'(0,T)$ by 
$$
\langle B(Y,y),w\rangle =\int_{\Omega} Y\cdot \nabla y\cdot w,\ \langle B(\bar y,Y),w\rangle =\int_{\Omega} \bar y\cdot \nabla Y\cdot w 
\hbox{ and }   \langle B_1(v),w\rangle =\int_{\Omega} \nabla v\cdot \nabla w \qquad\forall w\in \boldsymbol{V}
$$
belong to $L^2(0,T;\boldsymbol{V}')$ (see Lemma  \ref{B-V'} and  \ref{B_1-V'}), and from (\ref{NS}), (\ref{corrector})    and  (\ref{partial_t y}) we deduce that 
$$
\left\{
\begin{aligned}
&\frac{d}{dt} \int_\Omega Y\cdot w+\nu\int_{\Omega} \nabla Y\cdot \nabla w  =-\langle \partial_t v+B_1(v)+ B(Y,y)+B(\bar y,Y),w \rangle_{\boldsymbol{V}'\times \boldsymbol{V}}, \quad \forall w\in \boldsymbol{V}\\
& Y(0)=0,
\end{aligned}
\right.
$$
and from (\ref{estimate-z-1}),  (\ref{estimate-z-2}), (\ref{Estim-B}), (\ref{B-nul}) and (\ref{Estim-B_2}) we deduce that for all $t\in[0,T]$
$$
\begin{aligned}
\int_\Omega \vert Y(t)\vert^2 +\nu\int_{Q_t} \vert \nabla Y\vert^2
&\le \frac{1}{\nu}\int_0^t\| \partial_t v+ B_1 (v)+ B(Y,y)\|_{\boldsymbol{V}'}^2\\
&\le \frac{4}{\nu}(\| \partial_t v\|^2_{L^2(0,T,\boldsymbol{V}')}+\|v\|^2_{L^2(0,T,\boldsymbol{V})}+c\int_0^t\|Y\|_2^2\| y\|_{\boldsymbol{V}}^2)\\
&\le \frac{4}{\nu}(2E(y)+c\int_0^t\|Y\|_2^2\| y\|_{\boldsymbol{V}}^2).
\end{aligned}
$$
Gronwall's lemma then implies that for all $t\in[0,T]$
$$
\int_\Omega \vert Y(t)\vert^2  +\nu\int_{Q_t} \vert \nabla Y\vert^2\leq \frac8\nu E(y)\exp\big(\frac{c}\nu\int_0^t \|y\|_{\boldsymbol{V}}^2\big)
\leq \frac8\nu E(y)\exp(\frac{c}{\nu^2}M^2)$$ 
which gives
$$\|Y\|_{L^\infty(0,T;\boldsymbol{H})}+\sqrt{\nu}\|Y\|_{L^2(0,T;\boldsymbol{V})}\le \frac{4\sqrt{2}}{\sqrt{\nu}}\sqrt{ E(y)}\exp(\frac{c}{\nu^2}M^2)\le C(M)\sqrt{E(y)}.$$
Now
$$
\begin{aligned}
\|\partial_t Y\|_{L^2(0,T,\boldsymbol{V}')}
&\le \| \partial_t v+ B_1 (v)+\nu B_1(Y)+ B(Y,y)+B(\bar y,Y)\|_{L^2(0,T;\boldsymbol{V}')}\\
&\le  \nu\|Y\|_{L^2(0,T,\boldsymbol{V})}+ \| \partial_t v\|_{L^2(0,T,\boldsymbol{V}')}+\| v\|_{L^2(0,T,\boldsymbol{V})}\\
&\quad +c\|Y\|_{L^\infty(0,T;\boldsymbol{H})}\| y\|_{L^2(0,T;\boldsymbol{V})}+c\|\bar y\|_{L^\infty(0,T;\boldsymbol{H})}\|Y\|_{L^2(0,T,\boldsymbol{V})}\\
&\le  \sqrt{E(y)}\biggl(2\sqrt{2}\exp(\frac{c}{\nu^2}M^2)+2\sqrt{2}+cM\frac{4\sqrt{2}}{\nu}\exp(\frac{c}{\nu^2}M^2)\biggr)
\end{aligned} 
$$
and thus
$$\|\partial_t Y\|_{L^2(0,T,\boldsymbol{V}')}\le c(M)\sqrt{E(y)}.$$
\end{proof}

We now proceed with the second part of the proof and would like to show that the only critical points for $E$ correspond to solutions of (\ref{NS}). In such a case, the search for an element $y$ solution of \eqref{NS} is reduced to the minimization of $E$.

For any  $y\in  \A$, we now look for an element $Y_1\in \A_0$ solution of the following formulation 
\begin{equation}
\label{solution_Y1}
\left\{
\begin{aligned}
& \frac{d}{dt}\int_\Omega Y_1\cdot w +\nu\int_{\Omega} \nabla Y_1\cdot \nabla w +\int_{\Omega} y\cdot \nabla Y_1\cdot w \\
&\hspace*{3cm}+\int_{\Omega} Y_1\cdot \nabla y\cdot w=-\frac{d}{dt}\int_\Omega v\cdot w -  \int_\Omega \nabla v\cdot \nabla w, \quad \forall w\in \boldsymbol{V} \\
& Y_1(0)=0,
\end{aligned}
\right.
\end{equation}
where $v\in \A_0$ is the corrector (associated to $y$) solution of \eqref{corrector}. $Y_1$ enjoys the following property:

\begin{proposition}\label{prop_boundY1}
For all $y\in\A$, there exists a unique  $Y_1\in\A_0$ solution of (\ref{solution_Y1}). Moreover if for some $M\in\mathbb{R}$,  $\Vert \partial_t y\Vert_{L^{2}(0,T,\boldsymbol{V}')}\leq M$ and $\sqrt{\nu}\Vert \nabla y\Vert_{L^2(Q_T)^4}\leq M$, then this solution satisfies 
$$
\|\partial_t Y_1\|_{L^2(0,T,\boldsymbol{V}^\prime)}+\sqrt{\nu}\Vert \nabla Y_1\Vert_{L^2(Q_T)^4}\le c(M)\sqrt{E(y)}
$$
for some constant $c(M)>0$.
\end{proposition}

\begin{proof} As in Proposition  \ref{main-estimatee},    (\ref{solution_Y1}) can be written as 
\begin{equation}
\label{solution_Z}
\left\{
\begin{aligned}
&\frac{d}{dt} \int_\Omega Y_1\cdot w +\nu\int_{\Omega} \nabla Y_1\cdot \nabla w +\int_{\Omega} y\cdot \nabla Y_1\cdot w +\int_{\Omega} Y_1\cdot \nabla y\cdot w=-
\langle \partial_t v+B_1(v),w\rangle_{\boldsymbol{V}'\times\boldsymbol{V}}, \quad \forall w\in \boldsymbol{V}\\
& Y_1(0)=0.
\end{aligned}
\right.
\end{equation}
(\ref{solution_Z}) admits a unique solution   $Y_1\in \A_0$.
Indeed, let $y_1\in L^2(0,T,\boldsymbol{V})\cap \C([0,T];\boldsymbol{H})$. Moreover,  there exists (see \cite{Temam}) a unique $z_1\in \A_0$ solution of  
\begin{equation}
\label{solution_Z_1}
\left\{
\begin{aligned}
&\frac{d}{dt} \int_\Omega z_1\cdot w +\nu\int_{\Omega} \nabla z_1\cdot \nabla w +\int_{\Omega} y\cdot \nabla z_1\cdot w +\int_{\Omega} y_1\cdot \nabla y\cdot w=-
\langle \partial_t v+B_1(v),w\rangle_{\boldsymbol{V}'\times\boldsymbol{V}}, \quad \forall w\in \boldsymbol{V} \\
& z_1(0)=0.
\end{aligned}
\right.
\end{equation}
Let $\mathcal{T}:y_1\mapsto z_1$. Then if $z_2=\mathcal{T}(y_2)$, $z_1-z_2$ is solution of 
$$
\left\{
\begin{aligned}
&\frac{d}{dt} \int_\Omega (z_1-z_2)\cdot w +\nu\int_{\Omega} \nabla (z_1-z_2)\cdot \nabla w +\int_{\Omega} y\cdot \nabla (z_1-z_2)\cdot w +\int_{\Omega} (y_1-y_2)\cdot \nabla y\cdot w=0, \quad \forall w\in \boldsymbol{V} \\
&(z_1-z_2)(0)=0,
\end{aligned}
\right.
$$
and thus, for $w=z_1-z_2$
$$\frac12\frac{d}{dt} \int_\Omega|z_1-z_2|^2 +\nu\int_{\Omega} |\nabla (z_1-z_2)|^2 =-\int_{\Omega} (y_1-y_2)\cdot \nabla y\cdot (z_1-z_2).$$
But
$$\Big|\int_{\Omega} (y_1-y_2)\cdot \nabla y\cdot (z_1-z_2)\Big|
\le c\| y_1-y_2\|_2\| y\|_{\boldsymbol{V}}\|\nabla(z_1-z_2)\|_2\le c\| y_1-y_2\|^2_2\| y\|^2_{\boldsymbol{V}}+\frac{\nu}2\|\nabla(z_1-z_2)\|^2_2$$
so that
$$
\frac{d}{dt} \int_\Omega|z_1-z_2|^2 +\nu\int_{\Omega} |\nabla (z_1-z_2)|^2 \le c \| y_1-y_2\|^2_2\| y\|^2_{\boldsymbol{V}},
$$
and for all $t\in[0,T]$ 
$$ \| z_1-z_2\|_{L^\infty(0,t,\boldsymbol{H})}^2+\nu\int_0^{t}\!\!\!\int_{\Omega} |\nabla (z_1-z_2)|^2 \le c \| y_1-y_2\|_{L^\infty(0,t,\boldsymbol{H})}^2\int_0^t\| y\|^2_{\boldsymbol{V}}.$$

Since $y\in L^2(0,T;\boldsymbol{V})$, there exists $t'\in]0,T]$ such that $\int_0^{t'}\| y\|^2_{\boldsymbol{V}}\le \frac1{2c}$. We then have 
$$ \| z_1-z_2\|_{L^\infty(0,t',\boldsymbol{H})}^2+\nu\int_0^{t'}\!\!\!\!\int_{\Omega} |\nabla (z_1-z_2)|^2 \le \frac12 \| y_1-y_2\|_{L^\infty(0,t',\boldsymbol{H})}^2$$
and the map $\mathcal{T}$ is a contraction mapping on $X=\C([0,t'];\boldsymbol{H})\cap L^2(0,t';\boldsymbol{V})$. So $\mathcal{T}$ admits a unique fixed point $Y_1\in X$.
Moreover, from (\ref{solution_Z_1}) we deduce that $\partial_tY_1\in L^2(0,t', \boldsymbol{V}')$. 
Since the map $t\mapsto \int_0^{t}\|\nabla y\|^2_2$ is a uniformly continuous function, we can take $t'=T$.

For this solution we have, for all $t\in[0,T]$, since $\int_{Q_t}y\cdot\nabla Y_1\cdot Y_1=0$
$$
\frac{1}{2}\int_\Omega \vert Y_1(t)\vert^2 +\nu\int_{Q_t} \vert \nabla Y_1\vert^2= - \int_0^t\langle B(Y_1,y)+\partial_t v+  B_1(v),Y_1\rangle_{\boldsymbol{V}'\times\boldsymbol{V}}. 
$$
Moreover, as in the proof of Proposition \ref{main-estimatee}, we have 
\begin{equation}\label{esti1-Y_1}
\int_\Omega \vert Y_1(t)\vert^2 +\nu\int_{Q_t} \vert \nabla Y_1\vert^2\le \frac8\nu E(y)\exp(\frac{c}\nu\int_0^t \|y\|_{\boldsymbol{V}}^2) 
\end{equation}
and thus
$$\sqrt{\nu}\|Y_1\|_{L^2(0,T;\boldsymbol{V})}\le \frac{2\sqrt{2}}{\sqrt{\nu}}\sqrt{ E(y)}\exp(\frac{c}\nu\int_0^T \|y\|_{\boldsymbol{V}}^2) \leq \frac{2\sqrt{2}}{\sqrt{\nu}}\sqrt{ E(y)}\exp(\frac{c}{\nu^2}M^2)\le c(M)\sqrt{E(y)}
$$
and
\begin{equation}\label{esti2-Y_1}
\begin{aligned}
\|\partial_t Y_1\|_{L^2(0,T,\boldsymbol{V}')}  
&\le \sqrt{E(y)}(2\sqrt{2}\exp(\frac{c}\nu\int_0^T \|y\|_{\boldsymbol{V}}^2)+2\sqrt{2}+c\|y\|_{L^2(0,T;\boldsymbol{V})}\frac{2\sqrt{2}}{\sqrt{\nu}}\exp(\frac{c}\nu\int_0^T \|y\|_{\boldsymbol{V}}^2) \\
&\hskip 3cm +c\|y\|_{L^\infty(0,T;\boldsymbol{H})}\frac{2\sqrt{2}}{{\nu}}\exp(\frac{c}\nu\int_0^T \|y\|_{\boldsymbol{V}}^2) )\le c(M)\sqrt{E(y)}\\
&\le \sqrt{E(y)}\biggl(2\sqrt{2}\exp(\frac{c}{\nu^2}M^2)+2\sqrt{2}+cM\frac{4\sqrt{2}}{\nu}\exp(\frac{c}{\nu^2}M^2)\biggr)\le c(M)\sqrt{E(y)}.
\end{aligned}
\end{equation}
 
\end{proof}

\begin{proposition}\label{differentiable}
For all $y\in\A$, the map $Y\mapsto E(y+Y)$ is a differentiable function on the Hilbert space $\A_0$ and for any $Y\in \A_0$, we have 
$$
E^\prime(y)\cdot Y=\langle v,V \rangle_{\A_0}=\int_0^T\langle v,V\rangle_{\boldsymbol{V}}+\int_0^T\langle \partial_tv,\partial_tV\rangle_{\boldsymbol{V}'}
$$
where $V\in\A_0$ is the unique solution in $\mathcal{D}^{\prime}(0,T)$ of 
\begin{equation}
\label{solution_V}
\left\{
\begin{aligned}
& \frac{d}{dt} \int_\Omega V\cdot w +\int_{\Omega} \nabla V\cdot \nabla w +\frac{d}{dt} \int_\Omega Y\cdot w +\nu\int_{\Omega} \nabla Y\cdot \nabla w +\int_{\Omega} y\cdot \nabla Y\cdot w \\
&\hspace*{7cm}+\int_{\Omega} Y\cdot \nabla y\cdot w=0, \quad \forall w\in \boldsymbol{V} \\
& V(0)=0.
\end{aligned}
\right.
\end{equation}
\end{proposition}

\begin{proof} 
Let $y\in\A$ and $Y\in \A_0$.
We have $E(y+Y)=\frac12\|\overline{V}\|^2_{\A_0}$
where $\overline{V}\in\A_0$ is the unique solution of
$$
\left\{\begin{aligned}
& \frac{d}{dt} \int_\Omega \overline{V}\cdot w +\int_{\Omega} \nabla \overline{V}\cdot \nabla w +\frac{d}{dt} \int_\Omega (y+Y)\cdot w +\nu\int_{\Omega} \nabla (y+Y)\cdot \nabla w +\int_{\Omega} (y+Y)\cdot \nabla (y+Y)\cdot w \\
&\hspace*{3cm}-\langle f,w\rangle_{\boldsymbol{V}'\times \boldsymbol{V}} =0, \quad \forall w\in \boldsymbol{V} \\
& \overline{V}(0)=0.
\end{aligned}
\right.
$$

If $v\in\A_0$ is the solution of (\ref{corrector}) associated to $y$, $v'\in\A_0$ is the unique solution of
$$
\left\{\begin{aligned}
&\frac{d}{dt} \int_\Omega v'\cdot w +\int_{\Omega} \nabla v'\cdot \nabla w +\int_{\Omega} Y\cdot \nabla Y\cdot w =0, \quad \forall w\in \boldsymbol{V}\\
&v'(0)=0
\end{aligned}\right.$$
and 
 $V\in\A_0$ is the unique solution of (\ref{solution_V}), then it is straightforward to check that $\overline{V}-v-v'-V\in\A_0$ is   solution of
 $$
\left\{\begin{aligned}
&\frac{d}{dt} \int_\Omega (\overline{V}-v-v'-V)\cdot w +\int_{\Omega} \nabla (\overline{V}-v-v'-V)\cdot \nabla w =0, \quad \forall w\in \boldsymbol{V}\\
&(\overline{V}-v-v'-V)(0)=0
\end{aligned}\right.$$
and therefore $\overline{V}-v-v'-V=0$. Thus
$$E(y+Y)=\frac12\| v+v'+V\|^2_{\A_0}=\frac12\|  v\|^2_{\A_0}+\frac12\| v'\|^2_{\A_0}+\frac12\|V\|^2_{\A_0}+\langle V,v'\rangle_{\A_0}+\langle V,v\rangle_{\A_0}+\langle v,v'\rangle_{\A_0}.$$
 
 We deduce from (\ref{solution_V}) and (\ref{estimate-z-1}) that
$$\|V\|^2_{L^2(0,T,\boldsymbol{V})}\le c(\|\partial_t Y\|^2_{L^2(0,T,\boldsymbol{V}')}+\nu^2\|B_1(Y)\|^2_{L^2(0,T,\boldsymbol{V}')}+\|B(y,Y)\|^2_{L^2(0,T,\boldsymbol{V}')}+
\|B(Y,y)\|^2_{L^2(0,T,\boldsymbol{V}')})$$
and from (\ref{Estim-B_2}),  (\ref{Estim-B}) and (\ref{estim-y-L-infini}) that
$$\|V\|^2_{L^2(0,T,\boldsymbol{V})}\le c\|Y\|^2_{\A_0}.$$

Similarly, we  deduce from  (\ref{estimate-z-2}) that
$$\|\partial_t V\|^2_{L^2(0,T,\boldsymbol{V}')}\le c\|Y\|^2_{\A_0}.$$
Thus 
$$\| V\|^2_{\A_0}\le c\|Y\|^2_{\A_0}=o(\|Y\|_{\A_0}).$$

From (\ref{estimate-z-1}), (\ref{estimate-z-2}) and (\ref{Estim-B}), we also deduce that
$$\|v'\|^2_{L^2(0,T,\boldsymbol{V})}\le \|B(Y,Y)\|^2_{L^2(0,T,\boldsymbol{V}')}\le  c\|Y\|^2_{L^\infty(0,T,\boldsymbol{H})}\|Y\|^2_{L^2(0,T,\boldsymbol{V})}\le c\|Y\|^4_{\A_0}$$
and
$$\|\partial_t v'\|^2_{L^2(0,T,\boldsymbol{V}')}\le c\|Y\|^2_{L^\infty(0,T,\boldsymbol{H})}\|Y\|^2_{L^2(0,T,\boldsymbol{V})}\le c\|Y\|^4_{\A_0},$$
thus we also have 
$$\| v'\|^2_{\A_0}\le c\|Y\|^4_{\A_0}=o(\|Y\|_{\A_0}).$$

From the previous estimates, we then obtain
$$|\langle V,v'\rangle_{\A_0}|\le \|V\|_{\A_0}\|v'\|_{\A_0}\le c\|Y\|^3_{\A_0}=o(\|Y\|_{\A_0})$$
and
$$|\langle v,v'\rangle_{\A_0}|\le \|v\|_{\A_0}\|v'\|_{\A_0}\le c\sqrt{E(y)}\|Y\|^2_{\A_0}=o(\|Y\|_{\A_0}),$$
thus
$$
E(y+Y)=E(y)+\langle v,V\rangle_{\A_0}+o(\|Y\|_{\A_0}).
$$
Eventually, the estimate
$$|\langle v,V\rangle_{\A_0}|\le \|v\|_{\A_0}\|V\|_{\A_0}\le c\sqrt{E(y)} \|Y\|_{\A_0}$$
gives the continuity of the linear map $Y\mapsto \langle v,V\rangle_{\A_0}$.
\end{proof}

We are now in position to prove the following result.
\begin{proposition}\label{convergence}
If $\{y_k\}_{k\in \mathbb{N}}$ is a  sequence of $\A$ bounded in $L^2(0,T,\boldsymbol{V})\cap H^1(0,T;\boldsymbol{V}')$ satisfying   $E^{\prime}(y_k)\to 0$ as $k\to \infty$, then $E(y_k)\to 0$ as $k\to \infty$.
\end{proposition}

\begin{proof}    For any $y\in\A$ and $Y\in \A_0$, we have 
$$
E^\prime(y)\cdot Y=\langle v,V \rangle_{\A_0}=\int_0^T\langle v,V\rangle_{\boldsymbol{V}}+\int_0^T\langle \partial_tv,\partial_tV\rangle_{\boldsymbol{V}'}
$$
where $V\in \A_0$ is the unique solution in $\mathcal{D}^{\prime}(0,T)$ of (\ref{solution_V}).
In particular, taking $Y=Y_1$ defined by (\ref{solution_Y1}), we define an element $V_1$ solution of 
\begin{equation}
\label{solution_V1}
\left\{
\begin{aligned}
&\frac{d}{dt} \int_\Omega V_1\cdot w + \int_{\Omega} \nabla V_1\cdot \nabla w +\frac{d}{dt}\int_\Omega Y_1\cdot w +\nu\int_{\Omega} \nabla Y_1\cdot \nabla w +\int_{\Omega} y\cdot \nabla Y_1\cdot w \\
&\hspace*{3cm}+\int_{\Omega} Y_1\cdot \nabla y\cdot w=0, \quad \forall w\in \boldsymbol{V} \\
& V_1(0)=0.
\end{aligned}
\right.
\end{equation}
Summing (\ref{solution_V1}) and the (\ref{solution_Y1}), we obtain that $V_1-v$ solves (\ref{sol_z}) with $F\equiv 0$ and $z_0=0$. This implies that $V_1$ and $v$ coincide, and then 
\begin{equation}\label{E_Eprime}
E^{\prime}(y)\cdot Y_1=\int_0^T\|  v\|^2_{\boldsymbol{V}}+\int_0^T\|\partial_t v\|^2_{\boldsymbol{V}'}=2E(y), \quad \forall y\in\A.
\end{equation}
Let now, for any $k\in \mathbb{N}$, $Y_{1,k}$ be the solution of (\ref{solution_Y1}) associated to $y_k$. The previous equality writes $E^{\prime}(y_k)\cdot Y_{1,k}=2E(y_k)$ and implies our statement, since from Proposition \ref{prop_boundY1}, 
$Y_{1,k}$ is uniformly bounded in $\A_0$.
\end{proof}

\subsection{Minimizing sequence for $E$ - Link with the damped Newton method}

Very interestingly,  equality (\ref{E_Eprime}) shows that $-Y_1$ given by the solution of (\ref{solution_Y1}) is a descent direction for the functional $E$. Remark also, in view of (\ref{solution_Y1}), that the corrector $V$
associated to $Y_1$, given by (\ref{solution_V}) with $Y=Y_1$, is nothing else than the corrector $v$ itself. Therefore, we can define, for any $m\geq 1$, a minimizing sequence $\{y_k\}_{(k\in \mathbb{N})}$  for $E$ as follows:
\begin{equation}
\label{algo_LS_Y}
\left\{
\begin{aligned}
&y_0 \in \A, \\
&y_{k+1}=y_k-\lambda_k Y_{1,k}, \quad k\ge 0, \\
& E(y_k-\lambda_k Y_{1,k})   =\min_{\lambda\in [0,m]} E(y_k-\lambda Y_{1,k})    
\end{aligned}
\right.
\end{equation}
with $Y_{1,k}\in \A_0$ the solution of the formulation
\begin{equation}
\label{solution_Y1k}
\left\{
\begin{aligned}
& \frac{d}{dt} \int_\Omega Y_{1,k}\cdot w +\nu\int_{\Omega} \nabla Y_{1,k}\cdot \nabla w +\int_{\Omega} y_k\cdot \nabla Y_{1,k}\cdot w \\
&\hspace*{3cm}+\int_{\Omega} Y_{1,k}\cdot \nabla y_k\cdot w=-\frac{d}{dt}\int_\Omega v_k\cdot w -  \int_\Omega \nabla v_k\cdot \nabla w, \quad \forall w\in \boldsymbol{V}\\
& Y_{1,k}(0)=0,
\end{aligned}
\right.
\end{equation}
where $v_k\in \A_0$ is the corrector (associated to $y_k$) solution of \eqref{corrector}  leading (see (\ref{E_Eprime})) to $E^{\prime}(y_k)\cdot Y_{1,k}=2E(y_k)$. For any $k>0$, the direction $Y_{1,k}$ vanishes when $E(y_k)$ vanishes. 

\begin{lemma}\label{bounded}
Let $\{y_k\}_{k\in \mathbb{N}}$ the  sequence of $\A$ defined by (\ref{algo_LS_Y}). Then $\{y_k\}_{k\in \mathbb{N}}$ is a bounded  sequence of $H^1(0,T;\boldsymbol{V}')\cap L^{2}(0,T;\boldsymbol{V})$ and $\{E(y_k)\}_{k\in \mathbb{N}}$ is a decreasing sequence.
\end{lemma}

\begin{proof}
From (\ref{algo_LS_Y}) we deduce that, for all $k\in\mathbb{N}$ : 
$$E(y_{k+1})=E(y_k-\lambda_kY_{1,k})=\min_{\lambda\in[0,m]} E(y_k-\lambda Y_{1,k})  \le E(y_k)$$
and thus  the sequence $\{E(y_k)\}_{k\in \mathbb{N}}$ decreases and, for all $k\in \mathbb{N}$: $E(y_{k})\le E(y_0)$. Moreover, from the construction of the corrector $v_k\in\A_0$ associated to $y_k\in\A$ given by (\ref{corrector}), we deduce from Proposition \ref{Existence-NS} that $y_k\in \A$ is the unique solution  of 
$$
\left\{
\begin{aligned}
&\frac{d}{dt} \int_\Omega y_k\cdot w+\nu\int_{\Omega} \nabla y_k\cdot \nabla w  +\int_{\Omega} y_k\cdot \nabla y_k\cdot w =<f,w>_{\boldsymbol{V}'\times \boldsymbol{V}}\\
&\hspace*{6cm}
- \frac{d}{dt} \int_\Omega v_k\cdot w -\int_{\Omega} \nabla v_k\cdot \nabla w , \quad \forall w\in \boldsymbol{V}\\
&y_k(0)=u_0,
\end{aligned}
\right.
$$
and, using (\ref{Estim-B_2}) and (\ref{Estim-B})
\begin{equation}\label{estimation-y_k-1}
\begin{aligned}
\|y_k\|^2_{L^{\infty}(0,T;\boldsymbol{H})}
&\le \|u_0\|^2_{\boldsymbol{H}}+\frac1\nu \|f-\partial_t v_k-B_1(v_k)\|^2_{L^{2}(0,T;\boldsymbol{V}')}\\
&\le \|u_0\|^2_{\boldsymbol{H}}+\frac2{{\nu}} \|f\|^2_{L^{2}(0,T;\boldsymbol{V}')}+ \frac2{{\nu}}\|\partial_t v_k\|^2_{L^{2}(0,T;\boldsymbol{V}')}+\frac2{{\nu}} \|v_k\|^2_{L^{2}(0,T;\boldsymbol{V})}\\
&\le \|u_0\|^2_{\boldsymbol{H}}+\frac2{{\nu}}\|f\|^2_{L^{2}(0,T;\boldsymbol{V}')}+\frac4{\nu} {E(y_k)}\\
&\le \|u_0\|^2_{\boldsymbol{H}}+\frac2{{\nu}}\|f\|^2_{L^{2}(0,T;\boldsymbol{V}')}+\frac4{{\nu}} {E(y_0)}
\end{aligned}
\end{equation}
\begin{equation}\label{estimation-y_k}
\begin{aligned}
{\nu}\|y_k\|^2_{L^{2}(0,T;\boldsymbol{V})}
&\le  \|u_0\|^2_{\boldsymbol{H}}+\frac1\nu \|f-\partial_t v_k-B_1(v_k)\|^2_{L^{2}(0,T;\boldsymbol{V}')}\\
&\le \|u_0\|^2_{\boldsymbol{H}}+\frac2{{\nu}}\|f\|^2_{L^{2}(0,T;\boldsymbol{V}')}+\frac4{{\nu}} {E(y_0)}
\end{aligned}
\end{equation}
and
$$
\begin{aligned}
\|\partial_t y_k\|_{L^{2}(0,T;\boldsymbol{V}')}
&\le \|f-\partial_t v_k-B_1(v_k)-B(y_k,y_k)-\nu B_1(y_k)\|_{L^{2}(0,T;\boldsymbol{V}')}\\
&\le \|f\|_{L^{2}(0,T;\boldsymbol{V}')}+\|\partial_t v_k\|_{L^{2}(0,T;\boldsymbol{V}')}+\| v_k\|_{L^{2}(0,T;\boldsymbol{V})}\\
&\hskip 3cm+c\|y_k\|_{L^{\infty}(0,T;\boldsymbol{H})}\|y_k\|_{L^{2}(0,T;\boldsymbol{V})} +\nu \|y_k\|_{L^{2}(0,T;\boldsymbol{V})}\\
&\le \|f\|_{L^{2}(0,T;\boldsymbol{V}')}+2\sqrt{E(y_k)}+\sqrt{\nu}\|u_0\|_{\boldsymbol{H}}+\sqrt{2} \|f\|_{L^{2}(0,T;\boldsymbol{V}')}+2\sqrt{E(y_0)}\\
&\hskip 5cm+c\biggl(\|u_0\|^2_{\boldsymbol{H}}+\frac2{{\nu}}\|f\|^2_{L^{2}(0,T;\boldsymbol{V}')}+\frac4{{\nu}} {E(y_0)}\biggr)\\
&\le 3\|f\|_{L^{2}(0,T;\boldsymbol{V}')}+4\sqrt{E(y_0)}+\sqrt{\nu}\|u_0\|_{\boldsymbol{H}}\\
&\hskip 5cm+c\biggl(\|u_0\|^2_{\boldsymbol{H}}+\frac2{{\nu}}\|f\|^2_{L^{2}(0,T;\boldsymbol{V}')}+\frac4{{\nu}} {E(y_0)}\biggr).
\end{aligned}
$$
\end{proof}

\begin{lemma}\label{*}
Let $\{y_k\}_{k\in \mathbb{N}}$ the  sequence of $\A$ defined by (\ref{algo_LS_Y}). Then  for all $\lambda\in [0,m]$, the following estimate holds
\begin{equation} \label{E_expansion}
E(y_k-\lambda Y_{1,k})  \leq  {E(y_k)} \biggl(\vert 1-\lambda\vert +\lambda^2 \frac{c}{\nu\sqrt{\nu}} \sqrt{E(y_k)}\exp(\frac{c}{\nu}\int_0^T \|y_k\|_{\boldsymbol{V}}^2)\biggl)^2.
\end{equation}
\end{lemma}

\begin{proof}
Let $V_k$ be the corrector associated to $y_k-\lambda Y_{1,k}$. It is easy to check that $V_k$ is given by $(1-\lambda)v_k+\lambda^2 \overline{\overline{v}}_k$ where $\overline{\overline{v}}_k\in \A_0$ solves 
\begin{equation}
\label{correcteur2b}
\left\{
\begin{aligned}
& \frac{d}{dt}  \int_\Omega \overline{\overline{v}}_k\cdot w +\int_{\Omega} \nabla \overline{\overline{v}}_k\cdot \nabla w +\int_{\Omega} Y_{1,k}\cdot \nabla Y_{1,k}\cdot w=0, \quad \forall w\in \boldsymbol{V}\\
& \overline{\overline{v}}_k(0)=0,
\end{aligned}
\right.
\end{equation}
 and thus
\begin{equation}\label{exact-expansion}
\begin{aligned}
2E(y_k-\lambda Y_{1,k}) &=\|V_k\|^2_{\A_0} =\|(1-\lambda)v_k+\lambda^2 \overline{\overline{v}}_k\|^2_{\A_0}\\
&\le (|1-\lambda|\|v_k\|_{\A_0}+\lambda^2\|\overline{\overline{v}}_k\|_{\A_0})^2\\
&\le (\sqrt{2}|1-\lambda| \sqrt{E(y_k)}+\lambda^2\|\overline{\overline{v}}_k\|_{\A_0})^2,
\end{aligned}
\end{equation}
which gives
\begin{equation}\label{estim-E(y_k-lambda-Y}
E(y_k-\lambda Y_{1,k})   \le \biggl(|1-\lambda| \sqrt{E(y_k)}+\frac{\lambda^2}{\sqrt{2}}\|\overline{\overline{v}}_k\|_{\A_0}\biggr)^2:= g(\lambda,y_k).
\end{equation}

From (\ref{correcteur2b}), (\ref{estimate-z-1}), (\ref{estimate-z-2}) and   (\ref{Estim-B}) we deduce that
 $$\|\overline{\overline{v}}_k\|_{L^{2}(0,T;\boldsymbol{V})}\le \|B(Y_{1,k},Y_{1,k})\|_{L^{2}(0,T;\boldsymbol{V}')}\le c\|Y_{1,k}\|_{L^{\infty}(0,T;\boldsymbol{H})}\|Y_{1,k}\|_{L^{2}(0,T;\boldsymbol{V})}$$
 and
 $$
 \begin{aligned}
 \|\partial_t \overline{\overline{v}}_k\|_{L^{2}(0,T;\boldsymbol{V}')}
 &\le \|-B_1(\overline{\overline{v}}_k)-B(Y_{1,k},Y_{1,k})\|_{L^{2}(0,T;\boldsymbol{V}')}\\
 &\le \|\overline{\overline{v}}_k\|_{L^{2}(0,T;\boldsymbol{V})}+ c\|Y_{1,k}\|_{L^{\infty}(0,T;\boldsymbol{H})}\|Y_{1,k}\|_{L^{2}(0,T;\boldsymbol{V})}\\
 &\le   c\|Y_{1,k}\|_{L^{\infty}(0,T;\boldsymbol{H})}\|Y_{1,k}\|_{L^{2}(0,T;\boldsymbol{V})}.
 \end{aligned}$$

On the other hand, from (\ref{esti1-Y_1})  we deduce that
\begin{equation}
\label{estim-Y-1k}
\|Y_{1,k}\|^2_{L^{\infty}(0,T;\boldsymbol{H})}+\nu  \|Y_{1,k}\|^2_{L^{2}(0,T;\boldsymbol{V})} \leq \frac{16}\nu E(y_k)\exp\biggl(\frac{c}\nu\int_0^T \|y_k\|_{\boldsymbol{V}}^2\biggr). 
 \end{equation}
 Thus
 $$\|\overline{\overline{v}}_k\|_{L^{2}(0,T;\boldsymbol{V})}\le \frac{c}{\nu\sqrt{\nu}} E(y_k)\exp\biggl(\frac{c}{\nu}\int_0^T \|y_k\|_{\boldsymbol{V}}^2\biggr) $$
 and
  $$ \|\partial_t \overline{\overline{v}}_k\|_{L^{2}(0,T;\boldsymbol{V}')}\le \frac{c}{\nu\sqrt{\nu}} E(y_k)\exp\biggl(\frac{c}{\nu}\int_0^T \|y_k\|_{\boldsymbol{V}}^2\biggr)$$
  which gives 
$$\|\overline{\overline{v}}_k\|_{\A_0}=\sqrt{\|\overline{\overline{v}}_k\|^2_{L^{2}(0,T;\boldsymbol{V})}+
\|\partial_t \overline{\overline{v}}_k\|^2_{L^{2}(0,T;\boldsymbol{V})}}\leq  \frac{c}{\nu\sqrt{\nu}} E(y_k)\exp\biggl(\frac{c}{\nu}\int_0^T \|y_k\|_{\boldsymbol{V}}^2\biggr).
$$
From (\ref{estim-E(y_k-lambda-Y}) we then deduce  (\ref{E_expansion}).
\end{proof}

\begin{lemma}\label{E-go-0}
Let $\{y_k\}_{k\in \mathbb{N}}$ the  sequence of $\A$ defined by (\ref{algo_LS_Y}). Then $E(y_k)\to 0$ as $k\to \infty$. Moreover, there exists a $k_0\in \mathbb{N}$ such that 
the sequence $\{E(y_k)\}_{k\geq k_0}$ decays quadratically.
\end{lemma}

\begin{proof}
We  deduce from  (\ref{E_expansion}), using  (\ref{estimation-y_k}) that, for all $\lambda\in [0,m]$ and $k\in \mathbb{N}^\star$ :
\begin{equation} \label{estimateC1}
\sqrt{E(y_{k+1}) } 
\leq  \sqrt{E(y_k)} \biggl(\vert 1-\lambda\vert +\lambda^2 C_1\sqrt{E(y_k)}\biggr)
\end{equation}
where $C_1= \frac{c}{\nu\sqrt{\nu}} \exp\biggl(\frac{c}{\nu^2}\|u_0\|^2_{\boldsymbol{H}}+ \frac{c}{\nu^3} \|f\|^2_{L^{2}(0,T;\boldsymbol{V}')}+ \frac{c}{\nu^3} E(y_0) \biggr)$ does not depend on $y_k$.

Let us denote the polynomial $p_k$ by $p_k(\lambda)=\vert 1-\lambda\vert +\lambda^2 C_1\sqrt{E(y_k)}$ for all $\lambda\in [0,m]$. 
If $C_1\sqrt{E(y_0)}< 1$ (and thus $C_1\sqrt{E(y_k)}<1$ for all $k\in\mathbb{N}$) then 
$$\min_{\lambda\in[0,m]}p_k(\lambda)\le p_k(1)=C_1\sqrt{E(y_k)}$$
and thus
\begin{equation}
\label{C1E}
C_1\sqrt{E(y_{k+1})}\le \big(C_1\sqrt{E(y_k)}\big)^2
\end{equation}
implying that $C_1\sqrt{E(y_k)}\to 0$  as $k\to \infty$ with a quadratic rate.

Suppose now that $C_1\sqrt{E(y_0)}\ge 1$ and   denote $ I=\{k\in \mathbb{N},\ C_1\sqrt{E(y_k)}\ge 1\}$. Let us prove that $I$ is a finite subset of $\mathbb{N}$. For all $k\in I$, since $C_1\sqrt{E(y_k)}\ge 1$,
$$\min_{\lambda\in[0,m]}p_k(\lambda)=\min_{\lambda\in[0,1]}p_k(\lambda)=p_k\Big(\frac{1}{2C_1\sqrt{E(y_k)}}\Big)=1-\frac{1}{4C_1\sqrt{E(y_k)}}\le 1-\frac{1}{4C_1\sqrt{E(y_0)}}<1$$
and thus, for all $k\in I$,
$$\sqrt{E(y_{k+1}) } \le \Big(1-\frac{1}{4C_1\sqrt{E(y_0)}}\Big)\sqrt{E(y_k)}\le \Big(1-\frac{1}{4C_1\sqrt{E(y_0)}}\Big)^{k+1}\sqrt{E(y_0)}.$$

Since $\Big(1-\frac{1}{4C_1\sqrt{E(y_0)}}\Big)^{k+1}\to 0$ as $k\to+\infty$, there exists $k_0\in\mathbb{N}$ such that for all $k\ge k_0$, $C_1\sqrt{E(y_{k+1}) }<1$. 
Thus  $I$ is a finite subset of $\mathbb{N}$. Arguing as in the first case, it follows that $C_1\sqrt{E(y_k)}\to 0$  as $k\to \infty$.

\end{proof}

\begin{remark}
In view of the estimate above of the constant $C_1$, the number of iterates $k_1$ necessary to achieve the quadratic regime (from which the convergence is very fast) is of the order 
$\nu^{-3/2}\exp\biggl(\frac{c}{\nu^2}\|u_0\|^2_{\boldsymbol{H}}+ \frac{c}{\nu^3} \|f\|^2_{L^{2}(0,T;\boldsymbol{V}')}+ \frac{c}{\nu^3} E(y_0)\biggr)$, and therefore increases rapidly as $\nu\to 0$. In order to reduce the effect of the term $\nu^{-3}E(y_0)$, it is thus relevant, for small values of $\nu$, to couple the algorithm with a continuation approach with respect to $\nu$ (i.e. start the sequence $\{y_k\}_{(k\geq 0)}$ with an element $y_0$ solution of \eqref{NS} associated to $\overline{\nu}>\nu$ so that $\nu^{-3}E(y_0)$ be at most of the order $\mathcal{O}(\nu^{-2})$).
\end{remark}

\begin{lemma}\label{lambda-k-go-1}
Let $\{y_k\}_{k\in \mathbb{N}}$ the  sequence of $\A$ defined by (\ref{algo_LS_Y}). Then $\lambda_k\to 1$ as $k\to \infty$.\end{lemma}

\begin{proof}
We have 
$$
\begin{aligned}
2E(y_{k+1})=2E(y_k-\lambda_k Y_{1,k}) 
&=(1-\lambda_k)^2\|v_k\|^2_{\A_0}+2\lambda_k^2(1-\lambda_k)\langle v_k, \overline{\overline{v}}_k \rangle_{\A_0}+\lambda_k^4 \| \overline{\overline{v}}_k\|^2_{\A_0}\\
& =2(1-\lambda_k)^2E(y_k)+2\lambda_k^2(1-\lambda_k)\langle v_k, \overline{\overline{v}}_k \rangle_{\A_0}+\lambda_k^4 \| \overline{\overline{v}}_k\|^2_{\A_0},
\end{aligned}$$
and thus, as long as $E(y_k)\not=0$,
$$
(1-\lambda_k)^2=\frac{E(y_{k+1})}{E(y_{k})}-\lambda_k^2(1-\lambda_k)\frac{\langle v_k, \overline{\overline{v}}_k \rangle_{\A_0}}{E(y_{k})}-\lambda_k^4 \frac{\| \overline{\overline{v}}_k\|^2_{\A_0}}{2E(y_{k})}.
$$
Since 
$$
\|\overline{\overline{v}}_k\|_{\A_0}\leq  \frac{c}{\nu\sqrt{\nu}} E(y_k)\exp(\frac{c}{\nu}\int_0^T \|y_k\|_{\boldsymbol{V}}^2),
$$
we then have
$$\Big|\frac{\langle v_k, \overline{\overline{v}}_k \rangle_{\A_0}}{E(y_{k})}\Big|\le \frac{\|v_k\|_{\A_0}\| \overline{\overline{v}}_k\|_{\A_0}}{E(y_{k})}\le  \frac{c}{\nu\sqrt{\nu}}\sqrt{E(y_k)}\exp(\frac{c}{\nu}\int_0^T \|y_k\|_{\boldsymbol{V}}^2)\to 0$$
and
$$0\le \frac{\| \overline{\overline{v}}_k\|^2_{\A_0}}{E(y_{k})}\le  \frac{c}{\nu^3}{E(y_k)}\exp(\frac{c}{\nu}\int_0^T \|y_k\|_{\boldsymbol{V}}^2)\to 0$$
as $k\to +\infty$.
Consequently, since $\lambda_k\in[0,m]$ and $\frac{E(y_{k+1})}{E(y_{k})}\to 0$, we deduce that $(1-\lambda_k)^2\to0$, that is $\lambda_k\to 1$ as $k\to \infty$.
\end{proof}

From Lemmas \ref{bounded}, \ref{E-go-0} and Proposition \ref{main-estimatee} we can deduce that :

\begin{proposition}\label{conv-y-k-y}
Let $\{y_k\}_{k\in \mathbb{N}}$ the  sequence of $\A$ defined by (\ref{algo_LS_Y}). Then $y_k\to \bar y$ in  $H^1(0,T;\boldsymbol{V}')\cap L^{2}(0,T;\boldsymbol{V})$ where $\bar y\in\A$ is the unique solution of (\ref{NS}) given in Proposition \ref{Existence-NS}.
\end{proposition}

\begin{remark}
The strong convergence of the sequence $\{y_k\}_{k>0}$ is a consequence of the coercivity inequality (\ref{coercivity}), which is itself a consequence of the uniqueness of the solution $y$ of (\ref{NS}).
Actually, we can directly prove that $\{y_k\}_{k>0}$ is a convergent sequence in the Hilbert space $\A$ as follows; 
from (\ref{algo_LS_Y}) and the previous proposition, we deduce that the serie $\sum_{k\geq 0} \lambda_kY_{1k}$ converges in $H^1(0,T;\boldsymbol{V}')\cap L^{2}(0,T;\boldsymbol{V})$ and $\bar y=y_0+\sum_{k=0}^{+\infty} \lambda_kY_{1k}$. Moreover
 $\sum \lambda_k\|Y_{1k}\|_{ L^{2}(0,T;\boldsymbol{V})}$ converges and, if we denote by $k_0$ one $k\in\mathbb{N}$ such that  $C_1\sqrt{E(y_k)}<1$ (see Lemma \ref{E-go-0}), then for all $k\ge k_0$, using (\ref{estim-Y-1k}), (\ref{estimation-y_k}) and (\ref{C1E}) (since we can choose $C_1>1$), we have 
 \begin{equation}\label{**}
 \begin{aligned}
 \|\bar y-y_k\|_{L^{2}(0,T;\boldsymbol{V})}
 &=\| \sum_{i=k+1}^{+\infty} \lambda_iY_{1i}\|_{L^{2}(0,T;\boldsymbol{V})}\le  \sum_{i=k+1}^{+\infty} \lambda_i\|Y_{1i}\|_{L^{2}(0,T;\boldsymbol{V})}\\
& \le m \sum_{i=k+1}^{+\infty} \sqrt{C_1 E(y_i)}
\le  m\sum_{i=k+1}^{+\infty} C_1 \sqrt{ E(y_i)} \\
& \le  m\sum_{i=k+1}^{+\infty} (C_1 \sqrt{ E(y_{k_0})})^{2 ^{i-k_0}}
  \le  m\sum_{i=0}^{+\infty} (C_1 \sqrt{ E(y_{k_0})})^{2 ^{i+k+1-k_0}}\\
  &\le m (C_1 \sqrt{ E(y_{k_0})})^{2 ^{k+1-k_0}} \sum_{i=0}^{+\infty} (C_1 \sqrt{ E(y_{k_0})})^{2 ^{i}}\\
 &\le cm(C_1 \sqrt{ E(y_{k_0})})^{2 ^{k+1-k_0}}.
 \end{aligned}
 \end{equation}

From (\ref{esti2-Y_1}) we deduce that 
$$
\begin{aligned}
\|\partial_t Y_{1,k}\|_{L^2(0,T,\boldsymbol{V}')}
&\le \sqrt{E(y_k)}(2\sqrt{2}\exp(\frac{c}\nu\int_0^T\|y_k\|_{\boldsymbol{V}}^2)+2\sqrt{2}+c\|y_k\|_{L^2(0,T;\boldsymbol{V})}\frac{2\sqrt{2}}{\sqrt{\nu}}\exp(\frac{c}\nu\int_0^T \|y_k\|_{\boldsymbol{V}}^2) \\
&\hskip 3cm +c\|y_k\|_{L^\infty(0,T;\boldsymbol{H})}\frac{2\sqrt{2}}{{\nu}}\exp(\frac{c}\nu\int_0^T \|y_k\|_{\boldsymbol{V}}^2) )
\end{aligned}
$$
which gives, using (\ref{estimation-y_k-1}) and (\ref{estimation-y_k})
 \begin{equation}
\label{estim-dtY-1k}
\begin{aligned}
\|\partial_t Y_{1,k}\|_{L^2(0,T,\boldsymbol{V}')}
&\le \sqrt{E(y_k)}(2\sqrt{2}\exp(\frac{c}\nu\int_0^T\|y_k\|_{\boldsymbol{V}}^2)+2\sqrt{2}+\\
&+\frac{c}{\nu\sqrt{\nu}}( \sqrt{\nu}\|u_0\|_{\boldsymbol{H}}+\sqrt{2}\|f\|_{L^{2}(0,T;\boldsymbol{V}')}+2 \sqrt{E(y_0)})\exp(\frac{c}\nu\int_0^T \|y_k\|_{\boldsymbol{V}}^2) \\
&\le C_2\sqrt{E(y_k)}
\end{aligned}
 \end{equation}
where 
$$
 C_2
=c\exp\biggl(\frac{c}{\nu^2}\|u_0\|^2_{\boldsymbol{H}}+ \frac{c}{\nu^3} \|f\|^2_{L^{2}(0,T;\boldsymbol{V}')}+ \frac{c}{\nu^3} E(y_0) \biggr)\biggl(1+
\frac{1}{\nu\sqrt{\nu}}( \sqrt{\nu}\|u_0\|_{\boldsymbol{H}}+\sqrt{2}\|f\|_{L^{2}(0,T;\boldsymbol{V}')}+2 \sqrt{E(y_0)})\biggr).$$

Arguing as in the proof of Lemma \ref{E-go-0}, there exists $k_1\in\mathbb{N}$ such that, for all $k\ge k_1$ 
$$C_2\sqrt{E(y_{k+1})}\le \big(C_2\sqrt{E(y_k)}\big)^2$$
thus
 \begin{equation}\label{***}
 \begin{aligned}
 \|\partial_t \bar y-\partial_t y_k\|_{L^{2}(0,T;\boldsymbol{V}')}
 &=\| \sum_{i=k+1}^{+\infty} \lambda_i\partial_t Y_{1i}\|_{L^{2}(0,T;\boldsymbol{V}')}\le  \sum_{i=k+1}^{+\infty} \lambda_i\|\partial_t Y_{1i}\|_{L^{2}(0,T;\boldsymbol{V}')}\\
 &\le  m\sum_{i=k+1}^{+\infty} \sqrt{C_2 E(y_i)}
 \le m(C_2 \sqrt{ E(y_{k_1})})^{2 ^{k+1-k_1}} \sum_{i=0}^{+\infty} (C_2 \sqrt{ E(y_{k_1})})^{2 ^{i}}  \\
 &\le mc(C_2 \sqrt{ E(y_{k_1})})^{2 ^{k+1-k_1}}.
 \end{aligned}
 \end{equation}
\end{remark}

\begin{remark}\label{remg}
Lemmas \ref{bounded}, \ref{*}, \ref{E-go-0} and Proposition \ref{conv-y-k-y} remain true if we replace the minimization of $\lambda$ over $[0,m]$ by the minimization over $\mathbb{R}_+$. However, the sequence $\{\lambda_k\}_{k>0}$ may not be bounded in $\mathbb{R}_+$ (the fourth order polynomial $\lambda\to E(y_k-\lambda Y_k)$ may admit a critical point far from $1$.  In that case, (\ref{**}) and (\ref{***}) may not hold for some $m>0$.

Similarly, Lemmas \ref{E-go-0}, \ref{lambda-k-go-1} and Proposition \ref{conv-y-k-y} remain true for the sequence $\{y_k\}_{k\geq 0}$ 
defined as follows 
\begin{equation}
\label{algo_LS_Ybis}
\left\{
\begin{aligned}
&y_0 \in \A, \\
&y_{k+1}=y_k-\lambda_k Y_{1,k}, \quad k\ge 0, \\
& g(\lambda_k,y_k)   =\min_{\lambda\in \mathbb{R}^+} g(\lambda,y_k)
\end{aligned}
\right.
\end{equation}
leading to $\lambda_k\in ]0,1]$ for all $k\geq 0$ and $lim_{k\to \infty}\lambda_k\to 1^-$. The fourth order polynomial $g$ is defined in \eqref{estim-E(y_k-lambda-Y}.
\end{remark}

\begin{remark}
It seems surprising that the algorithm (\ref{algo_LS_Y}) achieves a quadratic rate for $k$ large.  
Let us consider the application $F: \mathcal{A}\to L^2(0,T,\boldsymbol{V}^\prime)$ defined as $F(y)=y_t+\nu B_1(y) + B(y,y)-f$.
The sequence $\{y_k\}_{k>0}$ associated to the Newton method to find the zero of $F$ is defined as follows: 
\begin{equation}
\nonumber
\left\{
\begin{aligned}
& y_0\in\mathcal{A}, \\
& F^{\prime}(y_k)\cdot (y_{k+1}-y_{k})=-F(y_k), \quad k\geq 0.
\end{aligned}
\right.
\end{equation} 
We check that this sequence coincides with the sequence obtained from (\ref{algo_LS_Y}) if $\lambda_k$ is fixed equal to one. The algorithm (\ref{algo_LS_Y}) which consists to optimize the parameter $\lambda_k\in [0,m]$, $m\geq 1$, in order to minimize $E(y_k)$, equivalently $\Vert F(y_k)\Vert_{L^2(0,T,\boldsymbol{V}^{\prime})}$ corresponds therefore to the so-called damped Newton method for the application $F$ (see \cite{deuflhard}). As the iterates increase, the optimal parameter $\lambda_k$ converges to one (according to Lemma \ref{lambda-k-go-1}), this globally convergent method behaves like the standard Newton method (for which $\lambda_k$ is fixed equal to one): this explains the quadratic rate after a finite number of iterates. Among the few works devoted to the damped Newton method for partial differential equations, we mention \cite{saramito} for computing viscoplastic fluid flows. 
\end{remark}

\section{The three dimensional case}\label{section-3D}

Let $\Omega\subset \mathbb{R}^3$ be a bounded connected open set whose boundary $\partial\Omega$ is $\C^2$ and let $P$ be the orthogonal projector in $L^2(\Omega)^3$ onto $\boldsymbol{H}$ 
   
We recall (see for instance \cite{Temam}) that for $f\in L^2(Q_T)^3$ and $u_0\in \boldsymbol{V}$, there exists  $T^*>0$ and a unique solution $y\in L^\infty(0,T^*;\boldsymbol{V})\cap L^2(0,T^*; H^2(\Omega)^3)$, $\partial_t y\in L^2(0,T^*; \boldsymbol{H})$ of the equation
\begin{equation}\label{NS-3D}
\left\{
\begin{aligned}
 &  
\frac{d}{dt}\int_\Omega y\cdot w+\nu\int_\Omega \nabla y\cdot\nabla w+\int_\Omega y\cdot\nabla y\cdot w=\int_\Omega f\cdot w,\quad \forall w\in \boldsymbol{V} \\
& y(\cdot,0)=u_0, \quad \text{in}\quad  \Omega.
\end{aligned}
\right.
\end{equation}

Following the methodology of the previous section, we introduce a least-squares method allowing to approximate the solution of (\ref{NS-3D}).
Arguments are similar so that details are omitted. 

\subsection{Preliminary results}

In the following, we repeatedly use the following classical estimates.

The first one is (see \cite{tartar}) :
\begin{lemma} Let any $u\in H^2(\Omega)^3$. Then $u\in L^\infty(\Omega)$ and there exists a constant $c=c(\Omega)$ such that
\begin{equation}
\label{L-infini}
\|u\|_\infty \leq c \|u\|^{1/2}_{\boldsymbol{V}}\|\Delta u\|_2^{1/2} .\end{equation}
\end{lemma}

We also have (see \cite{Simon}):

\begin{lemma} Let any $u\in   H^2(\Omega)^3$. There exists a constant $c=c(\Omega)$ such that
\begin{equation}
\label{Proj}
\|\Delta u\|_2 \leq c \|P(\Delta u)\|_2 .\end{equation}
\end{lemma}

and

\begin{lemma} \label{EDO}
Let $g\in W^{1,1}(0,T)$ and $k\in L^1(0,T)$, $\|k\|_{L^1(0,T)}\le k_0$ satisfying
$$\frac{dg}{dt}\le F(g)+k\hbox{ in } [0,T],\quad g(0)\le g_0$$
where $F$ is bounded on bounded sets from $\mathbb{R}$ into $\mathbb{R}$, that is
$$\forall a>0,\ \exists A>0\hbox{ such that } |x|\le a\Longrightarrow |F(x)|\le A.$$
Then, for every $\varepsilon>k_0$, there exists $T^*_1=T_1^*(F,g_0,k_0)>0$ independent of $g$ such that 
$$g(t)\le g_0+\varepsilon\quad \forall t\le T^*_1.$$
\end{lemma}

\begin{proof} The proof is an adaptation of the proof given in  \cite[Lemma 6]{Simon}.

First, since $g\in W^{1,1}(0,T)$, $g$ is continuous.  Let $\varepsilon>k_0$,  let $m$ be the smallest real value such that $g(m)=g_0+\varepsilon$ (if $m$ does not exist, we can take $T_1^*=T$) and let $n$ be the largest real value lower than $m$ such that $g(n)=g_0$. On $[n,m]$ 
 there holds $F(g)\le A$ where $A=\sup\{ |F(x)|: \ g_0\le x\le g_0+\varepsilon\}>0.$
 Then integrating by part we have
 $$\varepsilon =g(m)-g(n)\le \int_n^mA+k(t)dt\le mA+k_0$$
 thus $m\ge \frac{\varepsilon-k_0}{A}>0$ and we can take $T_1^*=\min( \frac{\varepsilon-k_0}{A},T).$ 
\end{proof}

We also denote, for all $0<t\le T$ 
$$
\A(t)=\{y\in L^2(0,t;H^2(\Omega)^3\cap \boldsymbol{V})\cap H^1(0,t;\boldsymbol{H}),\ y(0)=u_0\}
$$
and 
$$
\A_0(t)=\{y\in L^2(0,t;H^2(\Omega)^3\cap \boldsymbol{V})\cap H^1(0,t;\boldsymbol{H}),\ y(0)=0\}.
$$
Endowed with the scalar product
$$
\langle y,z\rangle_{\A_0(t)}=\int_0^t\langle P(\Delta y),P(\Delta z)\rangle_{\boldsymbol{H}}+\langle \partial_t y,\partial_t z\rangle_{\boldsymbol{H}}
$$
and the associated norm
$$
\|y\|_{\A_0(t)}=\sqrt{\|P(\Delta y)\|^2_{ L^2(0,t;\boldsymbol{H})}+\|\partial_t y\|^2_{L^2(0,t;\boldsymbol{H})}}
$$
$\A_0(t)$ is a Hilbert space.

We now recall and introduce several technical results. The first one is well-known  (we refer to \cite{Lions69,Temam}). 
 \begin{proposition}\label{Existence-NS-3D}
There exists $T^*=T^*(\Omega,\nu,u_0,f)$, $0<T^*\le T$ and a unique $\bar y\in  \A(T^*)$  solution in $\mathcal{D}^{\prime}(0,T^*)$ of (\ref{NS}).\end{proposition}
We also introduce the following result :
\begin{proposition}\label{prop_existence_v-3D}
For all $y\in\A(T^*)$,  there exists a unique $v\in \A_0(T^*)$  solution in $\mathcal{D}^{\prime}(0,T^*)$ of
\begin{equation}\label{corrector-3D}
\left\{
\begin{aligned}
& \frac{d}{dt} \int_\Omega v\cdot w +\int_{\Omega} \nabla v\cdot \nabla w +\frac{d}{dt} \int_\Omega y\cdot w+\nu\int_{\Omega} \nabla y\cdot \nabla w  \\
&\hspace*{6cm}+\int_{\Omega} y\cdot \nabla y\cdot w =\int_\Omega f\cdot w, \quad \forall w\in \boldsymbol{V}\\
&v(0)=0.
\end{aligned}
\right.
\end{equation}
\end{proposition}
The proof of this proposition is a consequence of the following standard  result (see \cite{Temam, Lions69}).

\begin{proposition}\label{Existence-3D}
For all $z_0\in \boldsymbol{V}$, all $\nu_1>0$ and all $F\in L^2(Q_{T^*})^3$, there exists a unique $z\in L^2(0,T^*;H^2(\Omega)^3\cap \boldsymbol{V})\cap H^1(0,T^*;\boldsymbol{H})$  solution in $\mathcal{D}^\prime(0,T^*)$ of 
\begin{equation}\label{sol_z-3D}
\left\{
\begin{aligned}
&\frac{d}{dt} \int_\Omega z\cdot w +\nu_1\int_{\Omega} \nabla z\cdot \nabla w =\int_\Omega F\cdot w, \quad \forall w\in \boldsymbol{V} \\
& z(0)=z_0.
\end{aligned}
\right.
\end{equation}
Moreover, for all $t\in [0,T^*]$, 
\begin{equation}\label{estimate-z-1-3D}
\|z(t)\|_{\boldsymbol{V}}^2+\nu_1\|P(\Delta z)\|_{L^2(0,t;\boldsymbol{H})}^2  \le  \frac1{\nu_1}\|P(F)\|^2_{L^2(Q_t)^3}+\| z_0\|^2_{\boldsymbol{V}}
\end{equation}
and
\begin{equation}\label{estimate-z-2-3D}
\nu_1\|z(t)\|_{\boldsymbol{V}}^2+\|\partial_t z\|_{L^2(0,t;\boldsymbol{H})}^2  \le  \|P(F)\|^2_{L^2(Q_t)^3}+\nu_1 \| z_0\|^2_{\boldsymbol{V}}.
\end{equation}
\end{proposition}

\begin{proof} (of Proposition \ref{prop_existence_v-3D}) 
Let $y\in  \A(T^*)$. Then $y\cdot \nabla y$, $\Delta y$ and $\partial_t y$ 
belong to $L^2(Q_{T^*})^3$. 
Then (\ref{corrector-3D}) may be rewriten as
$$
\left\{
\begin{aligned}
& \frac{d}{dt} \int_\Omega v\cdot w +\int_{\Omega} \nabla v\cdot \nabla w =\int_\Omega F\cdot w, \quad \forall w\in \boldsymbol{V}\\
& v(0)=0,
\end{aligned}
\right.
$$
where $F=f-y\cdot \nabla y+\nu \Delta y-\partial_t y\in L^2(Q_{T^*})^3$,  Proposition \ref{prop_existence_v-3D} is therefore a consequence of Proposition \ref{Existence-3D} with $\nu_1=1$.

Moreover, for all $t\in [0,T^*]$, 
$$
\|v(t)\|_{\boldsymbol{V}}^2+\|P(\Delta v)\|_{L^2(0,t;\boldsymbol{H})}^2 \le  \|P(f-y\cdot\nabla  y+\nu \Delta y-\partial_t y)\|^2_{L^2(Q_t)^3}
$$
and
$$
\|v(t)\|_{\boldsymbol{V}}^2+\|\partial_t v\|_{L^2(0,t;\boldsymbol{H})}^2  \le \|P(f-y\cdot\nabla y+\nu \Delta y-\partial_t y)\|^2_{L^2(Q_t)^3}.$$
\end{proof}

\subsection{The least-squares functional}

We now introduce our least-squares functional $E: \A(T^*)\to\mathbb{R}^+$ by putting
\begin{equation}\label{foncE-3D}
E(y)=\frac{1}{2}\int_0^{T^*}\| P(\Delta v)\|^2_{\boldsymbol{H}}+\frac{1}{2}\int_0^{T^*} \| \partial_t v\|^2_{\boldsymbol{H}}=\frac12\|v\|^2_{\A_0(T^*)}
\end{equation}
where the corrector $v$ is the unique solution of (\ref{corrector-3D}). The infimum of $E$ is equal to zero and is reached by a solution of (\ref{NS-3D}).

\begin{theorem}\label{th-convergence-3D}
Let $\{y_k\}_{k\in \mathbb{N}}$ be a sequence of $\A(T^*)$ bounded in $L^2(0,T^*;H^2(\Omega)^3\cap \boldsymbol{V})\cap H^1(0,T^*;\boldsymbol{H})$. If  $E^{\prime}(y_k)\to 0$ as $k\to \infty$, then the whole sequence $\{y_k\}_{k\in\mathbb{N}}$ converges strongly as $k\to \infty$ in $L^2(0,T^*;H^2(\Omega)^3\cap \boldsymbol{V})\cap H^1(0,T^*;\boldsymbol{H})$ to the solution $\bar y$ of \eqref{NS-3D}.
\end{theorem}

As in the previous section, we first check  that leading the error functional $E$ down to zero implies strong convergence to the unique solution of \eqref{NS-3D}. 

\begin{proposition}\label{main-estimatee-3D}
Let $\bar y\in\A(T^*) $ be the solution of (\ref{NS-3D}), $M\in \mathbb{R}$ such that $\Vert \partial_t \bar y\Vert_{L^{2}(Q_{T^*})^3}\leq M$ and $\sqrt{\nu}\Vert P(\Delta  \bar y)\Vert_{L^2(Q_{T^*})^3}\leq M$ and let $y\in \A(T^*)$. If $\Vert \partial_t y\Vert_{L^{2}(Q_{T^*})^3}\leq M$ and $\sqrt{\nu}\Vert P(\Delta  y)\Vert_{L^2(Q_{T^*})^3}\leq M$, then there exists a constant $c(M)$ such that 
$$
\| y-\bar y\|_{L^\infty(0,T^*;\boldsymbol{V})}  +\sqrt{\nu}\| P(\Delta y)-P(\Delta \bar y)\|_{L^2(Q_{T^*})^3)}+\|\partial_t y-\partial_t \bar y\|_{L^2(Q_{T^*})^3)}\leq c(M) \sqrt{E(y)}.
$$
\end{proposition}
\par\noindent

This proposition establishes that as we take down the error $E$ to zero, we get closer, in the norm of $\A_0(T^*)$  to the solution $\bar y$ of the problem (\ref{NS-3D}), and so, it justifies to look for global minimizers of \eqref{foncE-3D}.

\begin{proof} Let $Y=y-\bar y$. Then $Y\cdot \nabla y$, $\bar y\cdot\nabla Y$, $\partial_t v$  and $\Delta v$  belong to $L^2(Q_{T^*})^3$, and from (\ref{NS-3D}) and  (\ref{corrector-3D})  we deduce that 
$$
\left\{
\begin{aligned}
&\frac{d}{dt} \int_\Omega Y\cdot w+\nu\int_{\Omega} \nabla Y\cdot \nabla w  =-\int_\Omega  (\partial_t v+\Delta v+ Y\cdot\nabla  y+\bar y\cdot\nabla  Y)\cdot w , \quad \forall w\in \boldsymbol{V}\\
& Y(0)=0
\end{aligned}
\right.
$$
and from (\ref{estimate-z-1-3D})  we deduce that for all $t\in[0,T^*]$
$$
\begin{aligned}
\| Y(t)\|^2_{\boldsymbol{V}}+\nu\int_{Q_t} |P(\Delta Y)|^2
&\le \frac{1}{\nu}\int_0^t\| \partial_t v+\Delta v+ Y\cdot \nabla y+\bar y\cdot \nabla Y\|_{2}^2\\
&\le \frac{4}{\nu}(\| \partial_t v\|^2_{L^2(Q_t)^3}+\|\Delta v\|^2_{L^2(Q_t)^3}+\|Y\cdot \nabla y\|^2_{L^2(Q_t)^3}+\|\bar y\cdot \nabla Y\|^2_{L^2(Q_t)^3})\\
&\le \frac{c}{\nu}\biggl(E(y)+\int_0^t(\|\nabla y\|_3^2+\|\bar y\|_\infty^2) \| Y\|_{\boldsymbol{V}}^2\biggr)\\
&\le \frac{c}{\nu}\biggl(E(y)+\int_0^t(\| P(\Delta y)\|_2^2+\|P(\Delta \bar y)\|_2^2) \| Y\|_{\boldsymbol{V}}^2\biggr).
\end{aligned}
$$
Gronwall's lemma then implies that for all $t\in[0,T^*]$
$$
\| Y(t)\|^2_{\boldsymbol{V}}+\nu\int_{Q_t} |P(\Delta Y)|^2\leq \frac{c}\nu E(y)\exp(\frac{c}\nu\int_0^t \| P(\Delta y)\|_2^2+\|P(\Delta \bar y)\|_2^2)
\leq \frac{c}\nu E(y)\exp(\frac{c}{\nu^2}M^2)$$ 
which gives
$$\| Y\|^2_{L^\infty(0,T^*;\boldsymbol{V})}+\nu\int_{Q_{T^*}} |P(\Delta Y)|^2\le \frac{c}{\sqrt{\nu}}\sqrt{ E(y)}\exp(\frac{c}{\nu^2}M^2)\le C(M)\sqrt{E(y)}.$$

Similarly, using (\ref{estimate-z-2-3D}) instead of (\ref{estimate-z-1-3D}), we also obtain
$$\nu \| Y\|^2_{L^\infty(0,T^*;\boldsymbol{V})}+\int_{Q_{T^*}} |\partial_t Y|^2\le \frac{c}{\sqrt{\nu}}\sqrt{ E(y)}\exp(\frac{c}{\nu^2}M^2)\le C(M)\sqrt{E(y)}.$$
\end{proof}

For any  $y\in  \A(T^*)$, we now look for an element $Y_1\in \A_0(T^*)$ solution of the following formulation 
\begin{equation}
\label{solution_Y1-3D}
\left\{
\begin{aligned}
& \frac{d}{dt}\int_\Omega Y_1\cdot w +\nu\int_{\Omega} \nabla Y_1\cdot \nabla w +\int_{\Omega} y\cdot \nabla Y_1\cdot w \\
&\hspace*{3cm}+\int_{\Omega} Y_1\cdot \nabla y\cdot w=-\frac{d}{dt}\int_\Omega v\cdot w -  \int_\Omega \nabla v\cdot \nabla w, \quad \forall w\in \boldsymbol{V} \\
& Y_1(0)=0,
\end{aligned}
\right.
\end{equation}
where $v\in \A_0(T^*)$ is the corrector (associated to $y$) solution of \eqref{corrector-3D}. $Y_1$ enjoys the following property.

\begin{proposition}\label{prop_boundY1-3D}
For all $y\in\A(T^*)$, there exists an unique  $Y_1\in\A_0(T^*)$ solution of (\ref{solution_Y1-3D}). Moreover if for some $M\in\mathbb{R}$,  $\Vert \partial_t y\Vert_{L^{2}(Q_{T^*})^3}\leq M$, $\Vert  y\Vert_{L^{\infty}(0,T^*;{\boldsymbol{V}})}\leq M$ and $\sqrt{\nu}\Vert P(\Delta  y)\Vert_{L^2(Q_{T^*})^3}\leq M$, then this solution satisfies 
$$\Vert \partial_t Y_1\Vert_{L^{2}(Q_{T^*})^3}+ \Vert  Y_1\Vert_{L^{\infty}(0,T^*;{\boldsymbol{V}})}+\sqrt{\nu}\Vert P(\Delta  Y_1)\Vert_{L^2(Q_{T^*})^3}\leq c(M)\sqrt{E(y)}.$$\end{proposition}

\begin{proof} For all $t\in[0,T^*]$, $X(t)= \mathcal{C}([0,t]; \boldsymbol{V})\cap L^2(0,t;H^2(\Omega)^3)$ endowed with the norm $ \|z\|_{X(t)}=\sqrt{\| z\|_{L^\infty(0,t,\boldsymbol{V})}^2+\nu\|P(\Delta (z))\|^2_{L^2(0,t;\boldsymbol{H})}}$ is a Banach space.
Let $y_1\in X(T^*)$. Then $y_1\cdot \nabla y\in L^2(Q_{T^*})^3$.
As in Proposition \ref{Existence-3D}, there exists a unique $z_1\in \A_0(T^*)\subset X(T^*)$ solution of 
\begin{equation}
\label{solution_Z_1-3D}
\left\{
\begin{aligned}
&\frac{d}{dt} \int_\Omega z_1\cdot w +\nu\int_{\Omega} \nabla z_1\cdot \nabla w +\int_{\Omega} y\cdot \nabla y_1\cdot w +\int_{\Omega} y_1\cdot \nabla y\cdot w=-
\int_\Omega (\partial_t v+\Delta v)\cdot w, \quad \forall w\in \boldsymbol{V} \\
& z_1(0)=0.
\end{aligned}
\right.
\end{equation}
Let $\mathcal{T} :y_1\mapsto z_1$. Then if $z_2=\mathcal{T}(y_2)$, $z_1-z_2$ is solution of 
$$
\left\{
\begin{aligned}
&\frac{d}{dt} \int_\Omega (z_1-z_2)\cdot w +\nu\int_{\Omega} \nabla (z_1-z_2)\cdot \nabla w +\int_{\Omega} y\cdot \nabla (y_1-y_2)\cdot w +\int_{\Omega} (y_1-y_2)\cdot \nabla y\cdot w=0, \quad \forall w\in \boldsymbol{V} \\
&(z_1-z_2)(0)=0,
\end{aligned}
\right.
$$
and thus, for all $t\in[0,T^*]$
$$
\frac d{dt}\|z_1-z_2\|^2_{\boldsymbol{V}}+\nu\|P(\Delta (z_1-z_2))\|^2_{\boldsymbol{H}}
\le \frac 2\nu\|y\cdot \nabla (y_1-y_2)+ (y_1-y_2)\cdot \nabla y\|^2_2.
$$
But
$$
\|y\cdot \nabla (y_1-y_2)+ (y_1-y_2)\cdot \nabla y\|^2_2
\le c(\|y\|_{\boldsymbol{V}}\|P(\Delta y)\|_{\boldsymbol{H}}\|y_1-y_2\|^2_{\boldsymbol{V}}+\|y\|^2_{\boldsymbol{V}}\|y_1-y_2\|_{\boldsymbol{V}}\|P(\Delta (y_1-y_2))\|_{\boldsymbol{H}})
$$
so that
$$
\frac d{dt}\|z_1-z_2\|^2_{\boldsymbol{V}}+\nu\|P(\Delta (z_1-z_2))\|^2_{\boldsymbol{H}}\le \frac c\nu (\|y\|_{\boldsymbol{V}}\|P(\Delta y)\|_{\boldsymbol{H}}\|y_1-y_2\|^2_{\boldsymbol{V}}+\|y\|^2_{\boldsymbol{V}}\|y_1-y_2\|_{\boldsymbol{V}}\|P(\Delta (y_1-y_2))\|_{\boldsymbol{H}})
$$
and for all $t\in[0,T^*]$ 
$$
\begin{aligned}
&\| z_1-z_2\|_{L^\infty(0,t,\boldsymbol{V})}^2+\nu\|P(\Delta (z_1-z_2))\|^2_{L^2(0,t;\boldsymbol{H})} \\
&\le \frac c\nu(\|y\|_{L^\infty(0,t;\boldsymbol{V})}\|y_1-y_2\|^2_{L^\infty(0,t:\boldsymbol{V})}\int_0^t\|P(\Delta y)\|_{\boldsymbol{H}}
+\|y\|^2_{L^\infty(0,t;\boldsymbol{V})}\|y_1-y_2\|_{L^\infty(0,t;\boldsymbol{V})}\int_0^t\|P(\Delta (y_1-y_2))\|_{\boldsymbol{H}}\\
&\le  \frac c\nu t^{\frac12}\|y\|_{L^\infty(0,t;\boldsymbol{V})}(\|P(\Delta y)\|_{L^2(0,t;\boldsymbol{H})}\|y_1-y_2\|^2_{L^\infty(0,t:\boldsymbol{V})}+\|y_1-y_2\|_{L^\infty(0,t:\boldsymbol{V})}\|P(\Delta (y_1-y_2))\|_{L^2(0,t;\boldsymbol{H})})\\
&\le  \frac c\nu t^{\frac12}\|y\|_{L^\infty(0,T^*;\boldsymbol{V})}(1+\|P(\Delta y)\|_{L^2(0,T^*;\boldsymbol{H})})(\|y_1-y_2\|^2_{L^\infty(0,t:\boldsymbol{V})}+\|P(\Delta (y_1-y_2))\|^2_{L^2(0,t;\boldsymbol{H})}).
\end{aligned}
$$

Let now $t'\in]0,T^*]$ such that $\frac c\nu (t')^{\frac12}\|y\|_{L^\infty(0,T^*;\boldsymbol{V})}(1+\|P(\Delta y)\|_{L^2(0,T^*;\boldsymbol{H})})\le \min(\frac12,\frac\nu2)$. We then have 
$$
\| z_1-z_2\|_{L^\infty(0,t',\boldsymbol{V})}^2+\nu\|P(\Delta (z_1-z_2))\|^2_{L^2(0,t';\boldsymbol{H})}
 \le\frac12( \| y_1-y_2\|_{L^\infty(0,t',\boldsymbol{V})}^2+\nu\|P(\Delta (y_1-y_2))\|^2_{L^2(0,t';\boldsymbol{H})})
 $$
and the map $\mathcal{T}$ is a contraction mapping on $X(t')$. We deduce that $\mathcal{T}$ admits a unique fixed point $Y_1\in X(t')$.
Moreover, from (\ref{solution_Z_1-3D}) we deduce that $\partial_tY_1\in L^2(0,t', \boldsymbol{H})$ and thus $Y_1\in\A_0(t')$.
Since the map $t\mapsto t^{\frac12}$ is a uniformly continuous function on $[0,T^*]$, we can take $t'=T^*$.

For this solution we have, for all $t\in[0,T]$, 
$$
\begin{aligned}
\frac d{dt}\|Y_1\|^2_{\boldsymbol{V}}+\nu\|P(\Delta Y_1)\|^2_{\boldsymbol{H}}
&\le \frac 2\nu\|y\cdot \nabla Y_1+ Y_1\cdot \nabla y-\partial_t v+P(\Delta v)\|^2_2\\
&\hspace*{-2cm}\le \frac c\nu(\|y\|_{\boldsymbol{V}}\|P(\Delta y)\|_{\boldsymbol{H}}\|Y_1\|^2_{\boldsymbol{V}}+\|Y_1\|_{\boldsymbol{V}}\|P(\Delta Y_1)\|_{\boldsymbol{H}}\|y\|^2_{\boldsymbol{V}}+\|\partial_t v\|^2_{\boldsymbol{H}}+\|P(\Delta v)\|^2_{\boldsymbol{H}})\\
&\hspace*{-2cm}\le \frac \nu 2\|P(\Delta Y_1)\|^2_{\boldsymbol{H}}+\frac c\nu(\|y\|_{\boldsymbol{V}}\|P(\Delta y)\|_{\boldsymbol{H}}\|Y_1\|^2_{\boldsymbol{V}}+\frac 1{\nu^2}\|Y_1\|^2_{\boldsymbol{V}}\|y\|^4_{\boldsymbol{V}}+\|\partial_t v\|^2_{\boldsymbol{H}}+\|P(\Delta v)\|^2_{\boldsymbol{H}})
\end{aligned}$$
thus
$$\frac d{dt}\|Y_1\|^2_{\boldsymbol{V}}+{\nu}\|P(\Delta Y_1)\|^2_{\boldsymbol{H}}\le \frac c\nu(\|y\|_{\boldsymbol{V}}\|P(\Delta y)\|_{\boldsymbol{H}}\|Y_1\|^2_{\boldsymbol{V}}+\frac 1{\nu^2}\|Y_1\|^2_{\boldsymbol{V}}\|y\|^4_{\boldsymbol{V}}+\|\partial_t v\|^2_{\boldsymbol{H}}+\|P(\Delta v)\|^2_{\boldsymbol{H}}).$$

Gronwall's lemma then implies that for all $t\in[0,T^*]$ :
\begin{equation}
\label{esti1-Y_1-3D}
\begin{aligned}
\|Y_1\|^2_{\boldsymbol{V}}(t)+{\nu}\int_0^t\|P(\Delta Y_1)\|^2_{\boldsymbol{H}}
&\le \frac c\nu E(y)
\exp\Big( \frac c\nu \int_0^t (\|y\|_{\boldsymbol{V}}\|P(\Delta y)\|_{\boldsymbol{H}}+\frac 1{\nu^2}\|y\|^4_{\boldsymbol{V}}\Big)\\
&\le  c(M)E(y)
\end{aligned}
\end{equation}

Similar arguments give
\begin{equation}
\label{esti2-Y_1-3D}
{\nu}\|Y_1\|^2_{\boldsymbol{V}}(t)+\int_0^t\|\partial_t Y_1\|^2_{\boldsymbol{H}}\le  c(M)E(y).
\end{equation}
\end{proof}

\begin{proposition}\label{differentiable-3D}
For all $y\in\A(T^*)$, the map $Y\mapsto E(y+Y)$ is a differentiable function on the Hilbert space $\A_0(T^*)$ and for any $Y\in \A_0(T^*)$, we have 
$$
E^\prime(y)\cdot Y=\langle v,V \rangle_{\A_0(T^*)}=\int_0^{T^*} \langle  P(\Delta v), P(\Delta V) \rangle_{\boldsymbol{H}}+\int_0^{T^*} \langle  \partial_tv, \partial_tV \rangle_{\boldsymbol{H}}
$$
where $V\in\A_0(T^*)$ is the unique solution in $\mathcal{D}^{\prime}(0,T^*)$ of 
\begin{equation}
\label{solution_V-3D}
\left\{
\begin{aligned}
& \frac{d}{dt} \int_\Omega V\cdot w +\int_{\Omega} \nabla V\cdot \nabla w +\frac{d}{dt} \int_\Omega Y\cdot w +\nu\int_{\Omega} \nabla Y\cdot \nabla w +\int_{\Omega} y\cdot \nabla Y\cdot w \\
&\hspace*{7cm}+\int_{\Omega} Y\cdot \nabla y\cdot w=0, \quad \forall w\in \boldsymbol{V} \\
& V(0)=0.
\end{aligned}
\right.
\end{equation}
\end{proposition}
\begin{proof} 
Let $y\in\A(T^*)$ and $Y\in \A_0(T^*)$.
We have $E(y+Y)=\frac12\|\bar V\|^2_{\A_0(T^*)}$
where $\bar V\in\A_0(T^*)$ is the unique solution of
$$
\left\{
\begin{aligned}
& \frac{d}{dt} \int_\Omega \bar V\cdot w +\int_{\Omega} \nabla \bar V\cdot \nabla w +\frac{d}{dt} \int_\Omega (y+Y)\cdot w +\nu\int_{\Omega} \nabla (y+Y)\cdot \nabla w +\int_{\Omega} (y+Y)\cdot \nabla (y+Y)\cdot w \\
&\hspace*{3cm}-\int_{\Omega} f\cdot w =0, \quad \forall w\in \boldsymbol{V} \\
& \bar V(0)=0.
\end{aligned}
\right.
$$
If $v\in\A_0(T^*)$ is the solution of (\ref{corrector-3D}) associated to $y$, $v'\in\A_0(T^*)$ is the unique solution of
$$
\left\{\begin{aligned}
&\frac{d}{dt} \int_\Omega v'\cdot w +\int_{\Omega} \nabla v'\cdot \nabla w +\int_{\Omega} Y\cdot \nabla Y\cdot w =0, \quad \forall w\in \boldsymbol{V}\\
&v'(0)=0
\end{aligned}\right.$$
and 
 $V\in\A_0(T^*)$ is the unique solution of (\ref{solution_V-3D}), it is easy to check that $\bar V-v-v'-V\in\A_0(T^*)$ is solution of
 $$
\left\{\begin{aligned}
&\frac{d}{dt} \int_\Omega (\bar V-v-v'-V)\cdot w +\int_{\Omega} \nabla (\bar V-v-v'-V)\cdot \nabla w =0, \quad \forall w\in \boldsymbol{V}\\
&(\bar V-v-v'-V)(0)=0
\end{aligned}\right.$$
and therefore $\bar V-v-v'-V=0$. Thus
$$\begin{aligned}
E(y+Y)
&=\frac12\| v+v'+V\|^2_{\A_0(T^*)}\\
&=\frac12\|  v\|^2_{\A_0(T^*)}+\frac12\| v'\|^2_{\A_0(T^*)}+\frac12\|V\|^2_{\A_0(T^*)}+\langle V,v'\rangle_{\A_0(T^*)}+\langle V,v\rangle_{\A_0(T^*)}+\langle v,v'\rangle_{\A_0(T^*)}.
\end{aligned}
$$
 We deduce from (\ref{solution_V-3D}) and (\ref{estimate-z-1-3D}) that
$$
\begin{aligned}
\|V\|^2_{L^\infty(0,T^*;\boldsymbol{V})}+ \|P(\Delta V)\|^2_{L^2(0,T^*;\boldsymbol{H})}\le 
&c\big(\|\partial_t Y\|^2_{L^2(0,T^*;\boldsymbol{H})}+\nu^2\|P(\Delta Y)\|^2_{L^2(0,T^*;\boldsymbol{H})}\\
&+\|y\cdot \nabla Y\|^2_{L^2(Q_{T^*})^3}+ \|Y\cdot\nabla y\|^2_{L^2(Q_{T^*})^3)}\big)
\end{aligned}
$$
and, since
$$\begin{aligned}
\|y\cdot \nabla Y\|^2_{L^2(Q_{T^*})^3}
&\le \int_0^{T^*}\|\nabla Y\|_6^2\|y\|_3^2\le c\|y\|^2_{L^\infty(0,T^*;L^3(\Omega)^3)}\|P(\Delta Y)\|^2_{L^2(0,T^*;\boldsymbol{H})}\\
&\le c\|y\|^2_{\A_0(T^*)}\|Y\|^2_{\A_0(T^*)}
\end{aligned}$$
and
$$
 \|Y\cdot\nabla y\|^2_{L^2(Q_{T^*})^3)}\le c\|y\|^2_{\A_0(T^*)} \|Y\|^2_{\A_0(T^*)}$$
we deduce that
$$
\|V\|^2_{L^\infty(0,T^*;\boldsymbol{V})}+ \|P(\Delta V)\|^2_{L^2(0,T^*;\boldsymbol{H})} \le c\|Y\|^2_{\A_0(T^*)}.
$$
In the same way, we  deduce from  (\ref{estimate-z-2-3D}) that
$$
\|\partial_t V\|^2_{L^2(0,T^*;\boldsymbol{H})}\le c\|Y\|^2_{\A_0(T^*)}.
$$
Thus 
$$
\| V\|^2_{\A_0(T^*)}\le c\|Y\|^2_{\A_0(T^*)}=o(\|Y\|_{\A_0(T^*)}).
$$
From (\ref{estimate-z-1-3D}) and (\ref{estimate-z-2-3D}), we also deduce that
$$\|v'\|^2_{L^\infty(0,T^*;\boldsymbol{V})}+ \|P(\Delta v')\|^2_{L^2(0,T^*;\boldsymbol{H})}  \le \|Y\cdot\nabla Y\|^2_{L^2(Q_{T^*})^3}\le    c\|Y\|^4_{\A_0(T^*)}$$
and
$$
\|\partial_t v'\|^2_{L^2(0,T^*;\boldsymbol{H})}\le   c\|Y\|^4_{\A_0(T^*)},
$$
thus we also have 
$$
\| v'\|^2_{\A_0(T^*)}\le c\|Y\|^4_{\A_0(T^*)}=o(\|Y\|_{\A_0(T^*)}).
$$
From the previous estimates, we then obtain
$$
|\langle V,v'\rangle_{\A_0(T^*)}|\le \|V\|_{\A_0(T^*)}\|v'\|_{\A_0(T^*)}\le c\|Y\|^3_{\A_0(T^*)}=o(\|Y\|_{\A_0(T^*)})
$$
and
$$
|\langle v,v'\rangle_{\A_0(T^*)}|\le \|v\|_{\A_0(T^*)}\|v'\|_{\A_0(T^*)}\le c\sqrt{E(y)}\|Y\|^2_{\A_0(T^*)}=o(\|Y\|_{\A_0(T^*)}),
$$
thus
$$
E(y+Y)=E(y)+\langle v,V\rangle_{\A_0(T^*)}+o(\|Y\|_{\A_0(T^*)}).
$$
Eventually, the estimate
$$
|\langle v,V\rangle_{\A_0(T^*)}|\le \|v\|_{\A_0(T^*)}\|V\|_{\A_0(T^*)}\le cE(y) \|Y\|_{\A_0(T^*)}
$$
gives the continuity of the linear map $Y\mapsto \langle v,V\rangle_{\A_0(T^*)}$.
\end{proof}

We are now in position to prove the following result :
\begin{proposition}\label{convergence3D}
If $\{y_k\}_{k\in \mathbb{N}}$ is a  sequence of $\A(T^*)$ bounded in $L^2(0,T^*;H^2(\Omega)^3\cap \boldsymbol{V})\cap H^1(0,T^*;\boldsymbol{H})$ satisfying   $E^{\prime}(y_k)\to 0$ as $k\to \infty$, then $E(y_k)\to 0$ as $k\to \infty$.
\end{proposition}
\begin{proof}    For any $y\in\A({T^*})$ and $Y\in \A_0({T^*})$, we have 
$$
E^\prime(y)\cdot Y=\langle  v,  V \rangle_{\A_0({T^*})}=\int_0^{T^*}\langle P(\Delta v),P(\Delta V)\rangle_{\boldsymbol{H}}+\int_0^{T^*}\langle \partial_tv,\partial_tV\rangle_{\boldsymbol{H}}
$$
where $V\in \A_0({T^*})$ is the unique solution in $\mathcal{D}^{\prime}(0,{T^*})$ of (\ref{solution_V-3D}).
In particular, taking $Y=Y_1$ defined by (\ref{solution_Y1-3D}), we define an element $V_1$ solution of 
\begin{equation}
\label{solution_V1-3D}
\left\{
\begin{aligned}
&\frac{d}{dt} \int_\Omega V_1\cdot w + \int_{\Omega} \nabla V_1\cdot \nabla w +\frac{d}{dt}\int_\Omega Y_1\cdot w +\nu\int_{\Omega} \nabla Y_1\cdot \nabla w +\int_{\Omega} y\cdot \nabla Y_1\cdot w \\
&\hspace*{3cm}+\int_{\Omega} Y_1\cdot \nabla y\cdot w=0, \quad \forall w\in \boldsymbol{V} \\
& V_1(0)=0.
\end{aligned}
\right.
\end{equation}
Summing (\ref{solution_V1-3D}) and the (\ref{solution_Y1-3D}), we obtain that $V_1-v$ solves (\ref{sol_z-3D}) with $F\equiv 0$ and $z_0=0$. This implies that $V_1$ and $v$ coincide, and then 
\begin{equation}\label{E_Eprime-3D}
E^{\prime}(y)\cdot Y_1=\int_0^{T^*}\|  P(\Delta v)\|^2_{\boldsymbol{H}}+\int_0^{T^*}\|\partial_t v\|^2_{\boldsymbol{H}}=2E(y), \quad \forall y\in\A({T^*}).
\end{equation}
Let now, for any $k\in \mathbb{N}$, $Y_{1,k}$ be the solution of (\ref{solution_Y1-3D}) associated to $y_k$. The previous equality writes $E^{\prime}(y_k)\cdot Y_{1,k}=2E(y_k)$ and implies our statement, since from Proposition \ref{prop_boundY1-3D}, 
$Y_{1,k}$ is uniformly bounded in $\A_0({T^*})$.
\end{proof}

\subsection{Minimizing sequence for $E$}

Equality (\ref{E_Eprime-3D}) shows that $-Y_1$ given by the solution of (\ref{solution_Y1-3D}) is a descent direction for the functional $E$. Remark also, in view of (\ref{solution_Y1-3D}), that the corrector $V$
associated to $Y_1$, given by (\ref{solution_V-3D}) with $Y=Y_1$, is nothing else than the corrector $v$ itself. Therefore, we can define, for any $m\geq 1$, a minimizing sequence $y_k$ as follows: 
\begin{equation}
\label{algo_LS_Y-3D}
\left\{
\begin{aligned}
&y_0 \in \A(T^*), \\
&y_{k+1}=y_k-\lambda_k Y_{1,k}, \quad k\ge 0, \\
& E(y_k-\lambda_k Y_{1,k})   =\min_{\lambda\in [0,m]} E(y_k-\lambda Y_{1,k})    
\end{aligned}
\right.
\end{equation}
where $Y_{1,k}$ in $\A_0(T^*)$ solves the formulation
\begin{equation}
\label{solution_Y1k-3D}
\left\{
\begin{aligned}
& \frac{d}{dt} \int_\Omega Y_{1,k}\cdot w +\nu\int_{\Omega} \nabla Y_{1,k}\cdot \nabla w +\int_{\Omega} y_k\cdot \nabla Y_{1,k}\cdot w \\
&\hspace*{3cm}+\int_{\Omega} Y_{1,k}\cdot \nabla y_k\cdot w=-\frac{d}{dt}\int_\Omega v_k\cdot w -  \int_\Omega \nabla v_k\cdot \nabla w, \quad \forall w\in \boldsymbol{V}\\
& Y_{1,k}(0)=0,
\end{aligned}
\right.
\end{equation}
and $v_k$ in $\A_0(T^*)$ is the corrector (associated to $y_k$) solution of \eqref{corrector-3D}  leading (see (\ref{E_Eprime-3D})) to $E^{\prime}(y_k)\cdot Y_{1,k}=2E(y_k)$. 

It is easy to check that the corrector $V_k$ associated to $y_k-\lambda Y_{1,k}$ is given by $(1-\lambda)v_k+\lambda^2 \overline{\overline{v}}_k$ where $\overline{\overline{v}}_k\in \A_0(T^*)$ solves 
\begin{equation}
\label{correcteur2b-3D}
\left\{
\begin{aligned}
& \frac{d}{dt}  \int_\Omega \overline{\overline{v}}_k\cdot w +\int_{\Omega} \nabla \overline{\overline{v}}_k\cdot \nabla w +\int_{\Omega} Y_{1,k}\cdot \nabla Y_{1,k}\cdot w=0, \quad \forall w\in \boldsymbol{V}\\
& \overline{\overline{v}}_k(0)=0,
\end{aligned}
\right.
\end{equation}
 and thus
$$
\begin{aligned}
2E(y_k-\lambda Y_{1,k}) &=\|V_k\|^2_{\A_0(T^*)} =\|(1-\lambda)v_k+\lambda^2 \overline{\overline{v}}_k\|^2_{\A_0(T^*)}\\
&=(1-\lambda)^2\|v_k\|^2_{\A_0(T^*)}+2\lambda^2(1-\lambda)\langle v_k, \overline{\overline{v}}_k \rangle_{\A_0(T^*)}+\lambda^4 \| \overline{\overline{v}}_k\|^2_{\A_0(T^*)}.
\end{aligned}
$$
It is then easy to see that if $v_k\not=0$, $E(y_k-\lambda Y_{1,k})\to+\infty$ as $\lambda\to +\infty$ and thus there exists $\lambda_k\in\mathbb{R}_+$ such that $E(y_k-\lambda_k Y_{1,k})   =\min_{\lambda\in \mathbb{R}^+} E(y_k-\lambda Y_{1,k})$.

\begin{lemma}\label{bounded-3D}
Let $\{y_k\}_{k\in \mathbb{N}}$ the  sequence of $\A$ defined by (\ref{algo_LS_Y-3D}). Then $\{y_k\}_{k\in \mathbb{N}}$ is a bounded  sequence of $H^1(0,T^*;\boldsymbol{H})\cap L^{2}(0,T^*;H^2(\Omega)\cap \boldsymbol{V})$ and $\{E(y_k)\}_{k\in \mathbb{N}}$ is a decreasing sequence.
\end{lemma}

\begin{proof}
From (\ref{algo_LS_Y-3D}) we deduce that, for all $k\in\mathbb{N}$ : 
$$E(y_{k+1})=E(y_k-\lambda_kY_{1,k})=\min_{\lambda\in \mathbb{R}^+} E(y_k-\lambda Y_{1,k})  \le E(y_k)$$
and thus  the sequence $\{E(y_k)\}_{k\in \mathbb{N}}$ decreases and, for all $k\in \mathbb{N}$, $E(y_{k})\le E(y_0).$

From the construction of the corrector $v_k\in\A_0(T^*)$ associated to $y_k\in\A(T^*)$ given by (\ref{corrector-3D}), we deduce from Proposition \ref{Existence-NS-3D} that $y_k\in \A(T^*)$ is the unique solution  of 
$$
\left\{
\begin{aligned}
&\frac{d}{dt} \int_\Omega y_k\cdot w+\nu\int_{\Omega} \nabla y_k\cdot \nabla w  +\int_{\Omega} y_k\cdot \nabla y_k\cdot w =\int_\Omega f\cdot w\\
&\hspace*{6cm}
- \frac{d}{dt} \int_\Omega v_k\cdot w -\int_{\Omega} \nabla v_k\cdot \nabla w , \quad \forall w\in \boldsymbol{V}\\
&y_k(0)=u_0.
\end{aligned}
\right.
$$
For this solution we have the classical estimates
$$\begin{aligned}
\frac{d}{dt}\int_\Omega |\nabla y_k|^2+\nu \int_\Omega |P(\Delta y_k)|^2 &\le \frac {c}{\nu}\|\nabla y_k\|_2^6+\frac{2}{\nu}\|P(f-\partial_t v_k+\Delta v_k)\|_2^2\\
& \le \frac {c}{\nu}\|\nabla y_k\|_2^6+\frac{4}{\nu}\|P(f)\|_2^2+\frac{4}{\nu}(\|\partial_t v_k\|_2^2+\|P(\Delta v_k)\|_2^2)
\end{aligned}$$
Let us remark that 
$$\begin{aligned}
\int_0^{T ^*}\frac{4}{\nu}\|P(f)\|_2^2+\frac{4}{\nu}(\|\partial_t v_k\|_2^2+\|P(\Delta v_k)\|_2^2)
&\le\frac{4}{\nu}\|f\|_{L^2(Q_{T^*})^3}^2+\frac{8}{\nu}E(y_k)\\
&\le\frac{4}{\nu}\|f\|_{L^2(Q_{T^*})^3}^2+\frac{8}{\nu}E(y_0)=k_0
\end{aligned}$$
thus we deduce from Lemma \ref{EDO} that there exists $T_1^*=T_1^*(\Omega, \nu, u_0, E(y_0), f)$, $0<T_1^*\le T^*$ such that, for all $t\in[0,T_1^*]$ 
$$\int_\Omega |\nabla y_k|^2(t)+\nu \int_0^t\!\!\!\int_\Omega |P(\Delta y_k)|^2\le \|u_0\|_{\boldsymbol{V}}^2+\frac{8}{\nu}\|f\|_{L^2(Q_{T^*})^3}^2+\frac{16}{\nu}E(y_0).$$
Then it suffices to take $T^*=T_1^*$ in Proposition \ref{Existence-NS-3D}.

We then have 
\begin{equation}\label{estimation-y_k-1-3D}
\|y_k\|^2_{L^{\infty}(0,T^*;\boldsymbol{V})}
\le \|u_0\|^2_{\boldsymbol{V}}+\frac{8}{\nu} \|f\|^2_{L^{2}(Q_{T^*})^3}+ \frac{16}{{\nu}}{E(y_0)},
\end{equation}
\begin{equation}\label{estimation-y_k-3D}
\nu \|P(\Delta y_k)\|^2_{L^{2}(0,T^*;\boldsymbol{H})}
\le \|u_0\|^2_{\boldsymbol{V}}+\frac{8}{\nu} \|f\|^2_{L^{2}(Q_{T^*})^3}+ \frac{16}{{\nu}}{E(y_0)}
\end{equation}
and
$$
\begin{aligned}
\|\partial_t y_k\|_{L^{2}(0,T^*;\boldsymbol{H})}
&\le \|P(f-\partial_t v_k+\Delta v_k-y_k\cdot\nabla y_k+\nu \Delta y_k)\|_{L^{2}(0,T^*;\boldsymbol{H})}\\
&\le \|f\|_{L^{2}(Q_{T^*})^3}+\|\partial_t v_k\|_{L^{2}(0,T^*;\boldsymbol{H})}+\| P(\Delta v_k)\|_{L^{2}(0,T^*;\boldsymbol{H})}\\
&\hskip 3cm+c\|y_k\|^{\frac32}_{L^{\infty}(0,T^*;\boldsymbol{V})}\|P(\Delta y_k)\|^{\frac12}_{L^{2}(0,T^*;\boldsymbol{H})} +\nu \|P(\Delta y_k)\|_{L^{2}(0,T^*;\boldsymbol{H})}\\
&\le  \|f\|_{L^{2}(Q_{T^*})^3}+2\sqrt{E(y_k)}+ \frac{c}{\nu^{1/4}}\|u_0\|^2_{\boldsymbol{V}}+\frac{c}{\nu^{5/4}} \|f\|^2_{L^{2}(Q_{T^*})^3}+ \frac{c}{{\nu^{5/4}}}{E(y_0)}\\
&\hskip 5cm+ \frac{1}{\nu^{1/2}}\|u_0\|_{\boldsymbol{V}}+\frac{2\sqrt{2}}{\nu} \|f\|_{L^{2}(Q_{T^*})^3}+ \frac 4{{\nu}}\sqrt{E(y_0)}\\
&\le  \|f\|_{L^{2}(Q_{T^*})^3}+(2+\frac 4{\nu} )\sqrt{E(y_0)}+ \frac{c}{\nu^{1/4}}\|u_0\|^2_{\boldsymbol{V}}+\frac{c}{\nu^{5/4}} \|f\|^2_{L^{2}(Q_{T^*})^3}+ \frac{c}{{\nu^{5/4}}}{E(y_0)}\\
&\hskip 5cm+ \frac{1}{\nu^{1/2}}\|u_0\|_{\boldsymbol{V}}+\frac{2\sqrt{2}}{\nu} \|f\|_{L^{2}(Q_{T^*})^3}.
\end{aligned}
$$
\end{proof}
\begin{lemma}
Let $\{y_k\}_{k\in \mathbb{N}}$ be the  sequence of $\A(T^*)$ defined by (\ref{algo_LS_Y-3D}). Then  for all $\lambda>0$, the following estimate holds
\begin{equation} \label{E_expansion-3D}
E(y_k-\lambda Y_{1,k})  \leq  {E(y_k)} \biggl(\vert 1-\lambda\vert +\lambda^2 \frac c{\nu^{5/4}} \sqrt{E(y_k)}
\exp\Big( \frac c\nu \int_0^{T^*} \|y_k\|_{\boldsymbol{V}}\|P(\Delta y_k)\|_{\boldsymbol{H}}+\frac 1{\nu^2}\|y_k\|^4_{\boldsymbol{V}}\Big)\biggl)^2.
\end{equation}
\end{lemma}
\begin{proof}
Since
$$
2E(y_k-\lambda Y_{1,k}) =\|V_k\|^2_{\A_0(T^*)} =(1-\lambda)^2\|v_k\|^2_{\A_0(T^*)}+2\lambda^2(1-\lambda)\langle v_k, \overline{\overline{v}}_k \rangle_{\A_0(T^*)}+\lambda^4 \| \overline{\overline{v}}_k\|^2_{\A_0(T^*)},
$$
it follows that
$$
\begin{aligned}
2E(y_k-\lambda Y_{1,k}) 
&\le (1-\lambda)^2\|v_k\|^2_{\A_0(T^*)}+2\lambda^2|1-\lambda|\|v_k\|_{\A_0(T^*)}\|\overline{\overline{v}}_k\|_{\A_0(T^*)}+\lambda^4 \| \overline{\overline{v}}_k\|^2_{\A_0(T^*)}\\
&\le (|1-\lambda|\|v_k\|_{\A_0(T^*)}+\lambda^2\|\overline{\overline{v}}_k\|_{\A_0(T^*)})^2\\
&\le (\sqrt{2}|1-\lambda| \sqrt{E(y_k)}+\lambda^2\|\overline{\overline{v}}_k\|_{\A_0(T^*)})^2,
\end{aligned}
$$
which gives
\begin{equation}\label{estim-Ey_k-lambda-Y-3D}
E(y_k-\lambda Y_{1,k})   \le \big(|1-\lambda| \sqrt{E(y_k)}+\frac{\lambda^2}{\sqrt{2}}\|\overline{\overline{v}}_k\|_{\A_0(T^*)}\big)^2.
\end{equation}

From (\ref{correcteur2b-3D}), (\ref{estimate-z-1-3D}) and (\ref{estimate-z-2-3D}) we deduce that
 $$
 \begin{aligned}
 \|\overline{\overline{v}}_k\|^2_{L^{\infty}(0,T^*;\boldsymbol{V})}+\|P(\Delta \overline{\overline{v}}_k)\|^2_{L^{2}(0,T^*;\boldsymbol{H})} &\le \|P(Y_{1,k}\cdot\nabla Y_{1,k})\|^2_{L^{2}(0,T^*;\boldsymbol{H})}\\ 
 & \le c\|Y_{1,k}\|^{3}_{L^{\infty}(0,T^*;\boldsymbol{V})}\|P(\Delta Y_{1,k})\|_{L^{2}(0,T^*;\boldsymbol{H})}
 \end{aligned}
  $$
 and
 $$
 \begin{aligned}
 \|\partial_t \overline{\overline{v}}_k\|_{L^{2}(0,T^*;\boldsymbol{H})}
 &\le \|P(\Delta \overline{\overline{v}}_k-Y_{1,k}\cdot\nabla Y_{1,k})\|_{L^{2}(0,T^*;\boldsymbol{H})}\\
 &\le \|P(\Delta \overline{\overline{v}}_k)\|_{L^{2}(0,T^*;\boldsymbol{H})}+ c\|Y_{1,k}\|^{\frac32}_{L^{\infty}(0,T^*;\boldsymbol{V})}\|P(\Delta Y_{1,k})\|^{\frac12}_{L^{2}(0,T^*;\boldsymbol{H})} \\
 &\le    c\|Y_{1,k}\|^{\frac32}_{L^{\infty}(0,T^*;\boldsymbol{V})}\|P(\Delta Y_{1,k})\|^{\frac12}_{L^{2}(0,T^*;\boldsymbol{H})}.
 \end{aligned}$$

On the other hand, we deduce from (\ref{esti1-Y_1-3D}) that
\begin{equation}
\label{estim-Y-1k-3D}
\|Y_{1,k}\|^2_{L^{\infty}(0,T^*;\boldsymbol{V})}+\nu  \|P(\Delta Y_{1,k)}\|^2_{L^{2}(0,T^*;\boldsymbol{H})} \leq  \frac c\nu E(y_k)
\exp\Big( \frac c\nu \int_0^{T^*} \|y_k\|_{\boldsymbol{V}}\|P(\Delta y_k)\|_{\boldsymbol{H}}+\frac 1{\nu^2}\|y_k\|^4_{\boldsymbol{V}}\Big).
 \end{equation}
 Thus
 $$\|P(\Delta \overline{\overline{v}}_k)\|_{L^{2}(0,T^*;\boldsymbol{H})}\le \frac c{\nu^{5/4}} E(y_k)
\exp\Big( \frac c{\nu} \int_0^{T^*} \|y_k\|_{\boldsymbol{V}}\|P(\Delta y_k)\|_{\boldsymbol{H}}+\frac 1{\nu^2}\|y_k\|^4_{\boldsymbol{V}}\Big)
$$
 and
  $$ \|\partial_t \overline{\overline{v}}_k\|_{L^{2}(0,T^*;\boldsymbol{H})}\le  \frac c{\nu^{5/4}} E(y_k)
\exp\Big( \frac c{\nu} \int_0^{T^*} \|y_k\|_{\boldsymbol{V}}\|P(\Delta y_k)\|_{\boldsymbol{H}}+\frac 1{\nu^2}\|y_k\|^4_{\boldsymbol{V}}\Big)
$$
  which gives 
$$
\begin{aligned}
\|\overline{\overline{v}}_k\|_{\A_0(T^*)}=&\sqrt{\|\overline{\overline{v}}_k\|^2_{L^{2}(0,T^*;\boldsymbol{V})}+
\|\partial_t \overline{\overline{v}}_k\|^2_{L^{2}(0,T^*;\boldsymbol{V})}}\\
&\leq  \frac c{\nu^{5/4}} E(y)
\exp\Big( \frac c\nu \int_0^{T^*} \|y_k\|_{\boldsymbol{V}}\|P(\Delta y_k)\|_{\boldsymbol{H}}+\frac 1{\nu^2}\|y_k\|^4_{\boldsymbol{V}}\Big).
\end{aligned}
$$

From (\ref{estim-Ey_k-lambda-Y-3D}) we then deduce  (\ref{E_expansion-3D}).

\end{proof}

\begin{lemma}\label{E-go-0-3D}
Let $\{y_k\}_{k\in \mathbb{N}}$ the  sequence of $\A(T^*)$ defined by (\ref{algo_LS_Y-3D}). Then $E(y_k)\to 0$ as $k\to \infty$.\end{lemma}

\begin{proof}
Denoting $C_2= \|u_0\|^2_{\boldsymbol{V}}+\frac{8}{\nu} \|f\|^2_{L^{2}(Q_{T^*})^3}+ \frac{16}{{\nu}}{E(y_0)}$, we  deduce from  (\ref{E_expansion-3D}), using  (\ref{estimation-y_k-1-3D}) and (\ref{estimation-y_k-3D}) that, for all $\lambda\in\mathbb{R}_+$ :
$$
\sqrt{E(y_{k+1}) } 
\leq  \sqrt{E(y_k)} \biggl(\vert 1-\lambda\vert +\lambda^2 C_1\sqrt{E(y_k)}\biggr)
$$
where $C_1=  \frac c{\nu^{5/4}}\exp\biggl(c (\frac{C_2}{\nu^2}+(\frac{C_2}{\nu^2})^2)\biggr)$ does not depend on $y_k$, $k\in\mathbb{N}^*$.

Let us denote, for all $\lambda\in\mathbb{R}_+$,  $p_k(\lambda)=\vert 1-\lambda\vert +\lambda^2 C_1\sqrt{E(y_k)}$. 
If $C_1\sqrt{E(y_0)}< 1$ (and thus $C_1\sqrt{E(y_k)}<1$ for all $k\in\mathbb{N}$) then 
$$
\min_{\lambda\in\mathbb{R}_+}p_k(\lambda)\le p_k(1)=C_1\sqrt{E(y_k)}
$$
and thus $\sqrt{E(y_{k+1})}\le C_1E(y_k)$. This gives
\begin{equation}
C_1\sqrt{E(y_{k+1})}\le \big(C_1\sqrt{E(y_k)}\big)^2
\end{equation}
and then $C_1\sqrt{E(y_k)}\to 0$  as $k\to \infty$.

Suppose now that  $C_1\sqrt{E(y_0)}\ge 1$ and   denote $ I=\{k\in \mathbb{N},\ C_1\sqrt{E(y_k)}\ge 1\}$. Let us prove that $I$ is a finite subset of $\mathbb{N}$. 

For all $k\in I$, since $C_1\sqrt{E(y_k)}\ge 1$,
$$\min_{\lambda\in\mathbb{R}_+}p_k(\lambda)=\min_{\lambda\in[0,1]}p_k(\lambda)=p_k\Big(\frac{1}{2C_1\sqrt{E(y_k)}}\Big)=1-\frac{1}{4C_1\sqrt{E(y_k)}}\le 1-\frac{1}{4C_1\sqrt{E(y_0)}}<1$$
and thus, for all $k\in I$
$$\sqrt{E(y_{k+1}) } \le \Big(1-\frac{1}{4C_1\sqrt{E(y_0)}}\Big)\sqrt{E(y_k)}\le \Big(1-\frac{1}{4C_1\sqrt{E(y_0)}}\Big)^{k+1}\sqrt{E(y_0)}.$$

Since $\Big(1-\frac{1}{4C_1\sqrt{E(y_0)}}\Big)^{k+1}\to 0$ as $k\to+\infty$, there exists $k_0\in\mathbb{N}$ such that for all $k\ge k_0$, $C_1\sqrt{E(y_{k+1}) }<1$. 
Thus  $I$ is a finite subset of $\mathbb{N}$.  Arguing as in the first case, it follows that $C_1\sqrt{E(y_k)}\to 0$  as $k\to \infty$.
\end{proof}
 From Lemmas \ref{bounded-3D}, \ref{E-go-0-3D} and Proposition \ref{main-estimatee-3D} we can deduce that :

\begin{proposition}\label{prop3D_convergence}
Let $\{y_k\}_{k\in \mathbb{N}}$ the  sequence of $\A(T^*)$ defined by (\ref{algo_LS_Y-3D}). Then $y_k\to \bar y$ in  $H^1(0,T^*;\boldsymbol{H})\cap L^{2}(0,T^*;H^2(\Omega)^3\cap \boldsymbol{V})$ where $\bar y\in\A(T^*)$ is the unique solution of (\ref{NS}) given in Proposition \ref{Existence-NS-3D}.
\end{proposition}

From (\ref{algo_LS_Y-3D}) and Proposition \ref{prop3D_convergence}, we deduce that the serie $\sum \lambda_kY_{1k}$ converges in $H^1(0,T^*;\boldsymbol{H})\cap L^{2}(0,T^*;H^2(\Omega)^3\cap \boldsymbol{V})$ and $\bar y=y_0+\sum_{k=1}^{+\infty} \lambda_kY_{1k}$. Moreover
 $\sum \lambda_k\|Y_{1k}\|_{ \A_0(T^*)}$ converges and, if we denote $k_0$ one $k\in\mathbb{N}$ such that  $C_1\sqrt{E(y_k)}<1$ (see Lemma \ref{E-go-0-3D}), then for all $k\ge k_0$, using (\ref{estim-Y-1k-3D}), (\ref{estimation-y_k-3D}) and (\ref{C1E}) (since we can choose $C_1>1$) 
 $$\begin{aligned}
 \|\bar y-y_k\|_{L^{\infty}(0,T^*;\boldsymbol{V})}
 &=\| \sum_{i=k+1}^{+\infty} \lambda_iY_{1i}\|_{L^{\infty}(0,T^*;\boldsymbol{V})}\le  \sum_{i=k+1}^{+\infty} \lambda_i\|Y_{1i}\|_{L^{\infty}(0,T^*;\boldsymbol{V})}
 \le  m\sum_{i=k+1}^{+\infty} \sqrt{C_1 E(y_i)}\\
 & \le  m\sum_{i=k+1}^{+\infty} C_1 \sqrt{ E(y_i)}  \le  m\sum_{i=k+1}^{+\infty} (C_1 \sqrt{ E(y_{k_0})})^{2 ^{i-k_0}}
  \le  m\sum_{i=0}^{+\infty} (C_1 \sqrt{ E(y_{k_0})})^{2 ^{i+k+1-k_0}}\\
  &\le m(C_1 \sqrt{ E(y_{k_0})})^{2 ^{k+1-k_0}} \sum_{i=0}^{+\infty} (C_1 \sqrt{ E(y_{k_0})})^{2 ^{i}}
  =mc(C_1 \sqrt{ E(y_{k_0})})^{2 ^{k+1-k_0}}
 \end{aligned}$$
 and
  $$\begin{aligned}
 \|P(\Delta(\bar y-y_k))\|_{L^{2}(0,T^*;\boldsymbol{H})}
 &=\frac1\nu\| \sum_{i=k+1}^{+\infty} \lambda_iP(\Delta Y_{1i})\|_{L^{2}(0,T^*;\boldsymbol{H})}\le  \frac1\nu\sum_{i=k+1}^{+\infty} \lambda_i\|P(\Delta(Y_{1i)}\|_{L^{2}(0,T^*;\boldsymbol{H})}\\
 &\le  \frac1\nu\sum_{i=k+1}^{+\infty} \sqrt{C_1 E(y_i)}
  \le\frac c\nu(C_1 \sqrt{ E(y_{k_0})})^{2 ^{k+1-k_0}}.
 \end{aligned}$$

A similar estimate can be obtained for $\|\partial_t \bar y-\partial_t y_k\|_{L^{2}(0,T^\star;\boldsymbol{H})}$ (we refer to \eqref{***} for the 2D case).
 
 \begin{remark}
 All the results of this section remain true in the 2D case with $T^*=T$; in other words, results of Section \ref{section-2D} are valid to approximate regular solution associated to $u_0\in \boldsymbol{V}$ and $f\in L^2(Q_T)^3$ and $E$ defined by \eqref{foncE-3D}.
 \end{remark}

 \begin{remark}
 In a different functional framework, a similar approach is considered in \cite{PedregalNS}; more precisely, the author introduces the functional $E:\mathcal{V}\to \mathbb{R}$ defined $E(y)=\frac{1}{2}\Vert \nabla v\Vert^2_{L^2(Q_T)}$ 
with $\mathcal{V}:=y_0+\mathcal{V}_0$, $y_0\in H^1(Q_T)$ and $\mathcal{V}_0:=\{u\in H^1(Q_T; \mathbb{R}^d), u(t,\cdot)\in \boldsymbol{V}\,\forall t\in (0,T), u(0,\cdot)=0\}$
where $v(t,\cdot)$ solves for all $t\in (0,T)$, the steady Navier-Stokes equation with source term equal to $y_t(t,\cdot)-\nu \Delta y(t,\cdot)+(y(t,\cdot)\cdot \nabla)y(t,\cdot)-f(t,\cdot)$. Strong solutions are therefore considered assuming $u_0\in \boldsymbol{V}$ and $f\in (L^2(Q_T))^d$. Bound of $E(y)$ implies bound of $y$ in $L^2(0,T;\boldsymbol{V})$ but not in $H^1(0,T, L^2(\Omega)^d)$. This prevents to get the convergence of minimizing sequences in $\mathcal{V}$. 
\end{remark}

\section{Numerical illustrations} \label{section-numeric}

\subsection{Algorithm - Approximation}\label{section-algo}

We detail the main steps of the iterative algorithm \eqref{algo_LS_Y}. First, we define the initial term $y_0$ of the sequence $\{y_k\}_{(k\geq 0)}$ as the solution of the Stokes problem, solved by the backward Euler scheme: 
\begin{equation}
\label{initialization_sys}
\left\{
\begin{aligned}
 &  
\int_\Omega \frac{y_0^{n+1}-y_0^n}{\delta t}\cdot w+\overline{\nu}\int_\Omega \nabla y_0^{n+1}\cdot\nabla w=\langle f^n,w\rangle_{\boldsymbol{V}'\times \boldsymbol{V}}, \,\, \forall w\in \boldsymbol{V}, \,\, \forall n\geq 0, \\
& y_0^0(\cdot,0)=u_0, \quad \text{in}\quad  \Omega.
\end{aligned}
\right.
\end{equation}
for some $\overline{\nu}>0$. The incompressibility constraint is taken into account through a lagrange multiplier $\lambda\in L^2(\Omega)$ leading to the mixed formulation 
\begin{equation}
\left\{
\begin{aligned}
 &  \int_\Omega \frac{y_0^{n+1}-y_0^n}{\delta t}\cdot w+\nu\int_\Omega \nabla y_0^{n+1}\cdot\nabla w + \int_\Omega \lambda^{n+1}\,\nabla\cdot w   =\langle f^n,w\rangle_{\boldsymbol{V}'\times \boldsymbol{V}},\, \, \forall w\in (H^1_0(Q_T))^2, \,\, \forall n\geq 0,\\
 &  \int_\Omega \mu\,\nabla\cdot y_0^{n+1}   =0,\,\, \forall \mu\in L^2(\Omega), \,\,\forall n\geq 0,\\
 & y_0^0=u_0, \quad \text{in}\quad  \Omega.
\end{aligned}
\right.
\end{equation}
A conformal approximation in space is used for $(H_0^1(\Omega))^2\times L^2(\Omega)$ based on the inf-sup stable $\mathbb{P}_2/\mathbb{P}_1$ Taylor-Hood finite element. 
Then, assuming that (an approximation $\{y^n_{h,k}\}_{\{n,h\}}$ of) $y_k$ has been obtained for some $k\geq0$, $y_{k+1}$ is obtained as follows.

$\boldsymbol{(i)}$ From $y_k$, computation of (an approximation of) the corrector $v_k$ through the backward Euler scheme 
\begin{equation}
\label{pb1}
\left\{
\begin{aligned}
& \int_\Omega \frac{v_k^{n+1}-v_k^n}{\delta t}\cdot w +\int_{\Omega} \nabla v_k^{n+1}\cdot \nabla w +\int_\Omega \frac{y_k^{n+1}-y_k^n}{\delta t}\cdot w+\nu\int_{\Omega} \nabla y_k^{n+1}\cdot \nabla w  \\
&\hspace*{5cm}+\int_{\Omega} y_k^{n+1}\cdot \nabla y_k^{n+1}\cdot w =<f^n,w>_{\boldsymbol{V}'\times \boldsymbol{V}}, \,\, \forall w\in \boldsymbol{V},\,\, \forall n\geq 0,\\
&v_k^0=0.
\end{aligned}
\right.
\end{equation}

$\boldsymbol{(ii)}$ Then, in order to compute the term $\Vert v_{k,t}\Vert_{L^2(0,T;\boldsymbol{V}^\prime)}$ of $E(y_k)$, introduction of the function $w_k\in L^2(\boldsymbol{V})$ solution of 
\begin{equation}
\int_{0}^T\!\!\!\int_\Omega \nabla w_k \cdot \nabla w + v_{k,t} \cdot w = 0, \,\,\forall  w\in L^2(\boldsymbol{V})
\end{equation}
so that $\Vert v_{k,t}\Vert_{L^2(\boldsymbol{V}^\prime)}=\Vert \nabla w_k\Vert_{L^2(Q_T)}$. An approximation of $w_k$ is obtained through the scheme
\begin{equation}
\label{pb2}
\int_\Omega \nabla w_k^n \cdot \nabla w + \frac{v_k^{n+1}-v_k^n}{\delta t} \cdot w = 0, \forall  w\in \boldsymbol{V}, \,\forall n\in \mathbb{N}.
\end{equation}

$\boldsymbol{(iii)}$ Computation of an approximation of $Y_{1,k}$ solution of (\ref{solution_Y1k}) through the scheme 
\begin{equation}
\label{solution_Y1k_sch}
\left\{
\begin{aligned}
& \int_\Omega \frac{Y^{n+1}_{1,k}-Y^{n}_{1,k}}{\delta t}\cdot w +\nu\int_{\Omega} \nabla Y^{n+1}_{1,k}\cdot \nabla w +\int_{\Omega} y_k^{n+1}\cdot \nabla Y^{n+1}_{1,k}\cdot w+\int_{\Omega} Y_{1,k}^{n+1}\cdot \nabla y^{n+1}_k\cdot w \\
&\hspace*{3cm}=-\int_\Omega \frac{v^{n+1}_k-v^{n}_k}{\delta t}\cdot w -  \int_\Omega \nabla v_k^{n+1}\cdot \nabla w, \,\, \forall w\in \boldsymbol{V},\,\, \forall n\geq 0.\\
& Y^0_{1,k}=0.
\end{aligned}
\right.
\end{equation}

$\boldsymbol{(iv)}$ Computation of the corrector function $\overline{\overline{v}}_k$ solution of (\ref{correcteur2b}) through the scheme 
 \begin{equation}
\label{correcteur2b-dis}
\left\{
\begin{aligned}
& \int_\Omega \frac{\overline{\overline{v}}_k^{n+1}-\overline{\overline{v}}_k^n}{\delta t}\cdot w +\int_{\Omega} \nabla \overline{\overline{v}}_k^{n+1}\cdot \nabla w +\int_{\Omega} Y_{1,k}^{n+1}\cdot \nabla Y_{1,k}^{n+1}\cdot w=0, \,\,\forall w\in \boldsymbol{V}, \,\, n\geq 0,\\
& \overline{\overline{v}}_k(0)=0.
\end{aligned}
\right.
\end{equation}

$\boldsymbol{(v)}$ Computation of $\|v_k\|^2_{\A_0}$, $\langle v_k, \overline{\overline{v}}_k \rangle_{\A_0}$ and $\| \overline{\overline{v}}_k\|^2_{\A_0}$ appearing in $E(y_k-\lambda Y_{1,k})$ (see \eqref{exact-expansion}). The computation of $\|\overline{\overline{v}}_k\|_{\A_0}$ requires the computation of 
$\|\overline{\overline{v}}_k\|_{L^2(\boldsymbol{V}^\prime)}$, i.e. the introduction of  $\overline{\overline{w}}_k$ solution of
\begin{equation}
\nonumber
\int_{0}^T\!\!\!\int_\Omega \nabla \overline{\overline{w}}_k \cdot \nabla w + v_{k,t} \cdot w = 0, \,\,\forall  w\in L^2(\boldsymbol{V})
\end{equation}
so that $\Vert  \overline{\overline{v}}_{k,t}\Vert_{L^2(\boldsymbol{V}^\prime)}=\Vert \nabla  \overline{\overline{w}}_k\Vert_{L^2(Q_T)}$ through the scheme
\begin{equation}
\label{pb4}
\int_\Omega \nabla \overline{\overline{w}}_k^n \cdot \nabla w + \frac{\overline{\overline{v}}_k^{n+1}-\overline{\overline{v}}_k^n}{\delta t} \cdot w = 0, \,\forall  w\in \boldsymbol{V}, \,\forall n\in \mathbb{N}.
\end{equation}

$\boldsymbol{(vi)}$ Determination of the minimum $\lambda_k\in (0,m]$ of 
$$
\lambda\to E(y_k-\lambda Y_{1,k})=(1-\lambda)^2\|v_k\|^2_{\A_0}+2\lambda^2(1-\lambda)\langle v_k, \overline{\overline{v}}_k \rangle_{\A_0}+\lambda^4 \| \overline{\overline{v}}_k\|^2_{\A_0}
$$
through a Newton-Raphson method starting from $0$ and finally update of the sequence $y_{k+1}=y_k -\lambda_k Y_{1,k}$.

As a summary, the determination of $y_{k+1}$ from $y_k$ involves the resolution of four Stokes types formulations, namely \eqref{pb1},\eqref{pb2},\eqref{correcteur2b-dis} and \eqref{pb4} plus the resolution of the linearized Navier-Stokes formulation \eqref{solution_Y1k_sch}. This latter concentrates most of the computational times ressources since the operator (to be inverted) varies with the indexe $n$. 

Instead of minimizing exactly the fourth order polynomial $\lambda\to E(y_k-\lambda Y_{1,k})$ in step $\boldsymbol{(vi)}$, we may simpler minimize w.r.t. $\lambda\in (0,1]$ the right hand side of the estimate 
$$
E(y_k-\lambda Y_{1,k})   \le \biggl(|1-\lambda| \sqrt{E(y_k)}+\frac{\lambda^2}{\sqrt{2}}\|\overline{\overline{v}}_k\|_{\A_0}\biggr)^2
$$
(appearing in the proof of Lemma \ref{*}) leading to $\widehat{\lambda}_k=\min\biggl(1,\frac{\sqrt{E(y_k)}}{\sqrt{2}\Vert \overline{\overline{v}}_k\Vert_{\A_0}} \biggr)$.  (see remark \ref{remg}). This avoids the computation of the scalar product $\langle v_k, \overline{\overline{v}}_k \rangle_{\A_0}$ and one resolution of Stokes type formulations.

\begin{remark}
Similarly, we may also consider the equivalent functional $\widetilde{E}$ defined in \eqref{Etilde}. This avoids the introduction of the auxillary corrector function $v$ and reduces to three  (instead of four) the number of Stokes type formulations to be solved. Precisely, using the initialization defined in \eqref{initialization_sys}, the algorithm is as follows : 

$\boldsymbol{(i)}$ Computation of $\widetilde{E}(y_k)=\Vert h_k\Vert_{L^2(\boldsymbol{V})}=\Vert \nabla h_k\Vert_{L^2(Q_T)}$ where $h_k$ solves 
\begin{equation}
\nonumber
\int_{0}^T\!\!\!\int_\Omega \nabla h_k \cdot \nabla w + (y_{k,t}-\nu\Delta y_k+y_k\cdot \nabla y_k-f) \cdot w = 0, \,\,\forall  w\in L^2(\boldsymbol{V})
\end{equation}
through the scheme
\begin{equation}
\label{pb4bis}
\int_\Omega \nabla h_k^n \cdot \nabla w + \frac{y_k^{n+1}-y_k^n}{\delta t} \cdot w+\nu \nabla y_k^{n+1}\cdot \nabla w+ y^{n+1}_k\cdot \nabla y^{n+1}_k= <f^n,w>_{\boldsymbol{V^\prime},\boldsymbol{V}}, \,\forall  w\in \boldsymbol{V}, \,\forall n\in \mathbb{N}.
\end{equation}

$\boldsymbol{(ii)}$ Computation of an approximation of $Y_{1,k}$ from $y_k$ through the scheme
\begin{equation}
\label{solution_Y1k_sch_y1k}
\left\{
\begin{aligned}
&\int_\Omega \frac{Y^{n+1}_{1,k}-Y^{n}_{1,k}}{\delta t}\cdot w +\nu\int_{\Omega} \nabla Y^{n+1}_{1,k}\cdot \nabla w +\int_{\Omega} y_k^{n+1}\cdot \nabla Y^{n+1}_{1,k}\cdot w
 +\int_{\Omega} Y_k^{n+1}\cdot \nabla y^{n+1}_{1,k}\cdot w \\
  &\hspace*{4cm}=\int_\Omega \frac{y_k^{n+1}-y_k^n}{\delta t}\cdot w+\nu\int_{\Omega} \nabla y_k^{n+1}\cdot \nabla w  \\
&\hspace*{5cm}+\int_{\Omega} y_k^{n+1}\cdot \nabla y_k^{n+1}\cdot w -<f^n,w>_{\boldsymbol{V}'\times \boldsymbol{V}}, \,\, \forall w\in \boldsymbol{V},\,\, \forall n\geq 0,\\
&Y_{1,k}^0=0.
\end{aligned}
\right.
\end{equation}

$\boldsymbol{(iii)}$ Computation of $\Vert B(Y_{1,k},Y_{1,k})\Vert_{L^2(0,T;\boldsymbol{V}^\prime)}=\Vert h_k\Vert_{L^2(\boldsymbol{V})}=\Vert \nabla h_k\Vert_{L^2(Q_T)}$ where $h_k$ solves 
\begin{equation}
\nonumber
\int_{0}^T\!\!\!\int_\Omega \nabla h_k \cdot \nabla w + Y_{1,k}\cdot \nabla Y_{1,k}\cdot w = 0, \,\,\forall  w\in L^2(\boldsymbol{V})
\end{equation}
and similarly of the term $\langle y_{k,t}+\nu B_1(y_k)+B(y_k,y_k) , B(Y_{1,k},Y_{1,k}) \rangle_{L^2(0,T;\boldsymbol{V}^\prime)}$.

$\boldsymbol{(iv)}$ Determination of the minimum $\lambda_k\in (0,m]$ of 
$$
\begin{aligned}
\lambda\to \widetilde{E}(y_k-\lambda Y_{1,k})=& (1-\lambda)^2\widetilde{E}(y_k)+\lambda^2(1-\lambda)\langle y_{k,t}+\nu B_1(y_k)+B(y_k,y_k)-f, B(Y_{1,k},Y_{1,k}) \rangle_{L^2(0,T;\boldsymbol{V}^\prime)}\\
&+\frac{\lambda^4}{2} \Vert B(Y_{1,k},Y_{1,k})\Vert^2_{L^2(0,T;\boldsymbol{V}^\prime)}
\end{aligned}
$$
through a Newton-Raphson method starting from $0$ and finally update of the sequence $y_{k+1}=y_k -\lambda_k Y_{1,k}$ until $\widetilde{E}(y_k)$ is small enough.

We emphasize one more time that the case $\lambda_k$ coincides with the standard Newton algorithm to find zeros of the functional $F:\mathcal{A}\to L^2(0,T;\boldsymbol{V^\prime})$ defined by $F(y)=y_t+\nu B_1(y)+B(y,y)-f$. In term of computational time ressources, the determination of the optimal descent step $\lambda_k$ is negligible with respect to the resolution in the step $\boldsymbol{(ii)}$.
\end{remark}

\subsection{2D semi-circular driven cavity}
We illustrate our theoreticals results for the 2D semi-circular cavity discussed in \cite{Glowinski_disk}. The geometry is a semi-disk $\Omega=\{(x_1,x_2)\in \mathbb{R}^2, x_1^2+x_2^2< 1/4, x_2\leq 0\}$
depicted on Figure \ref{disk}. The velocity is imposed to $y=(g,0)$ on $\Gamma_0=\{(x_1,0)\in \mathbb{R}^2, \vert x_1\vert<1/2\}$ with $g$ vanishing at $x_1=\pm 1/2$
and close to one elsewhere: we take $g(x_1)=(1-e^{100(x_1-1/2)})(1-e^{-100(x_1+1/2)})$. On the complementary $\Gamma_1=\{(x_1,x_2)\in \mathbb{R}^2, x_2<0, x_1^2+x_2^2=1/4\}$ of the boundary the velocity is fixed to zero. 

\begin{figure}[http]
\begin{center}
\includegraphics[scale=0.7]{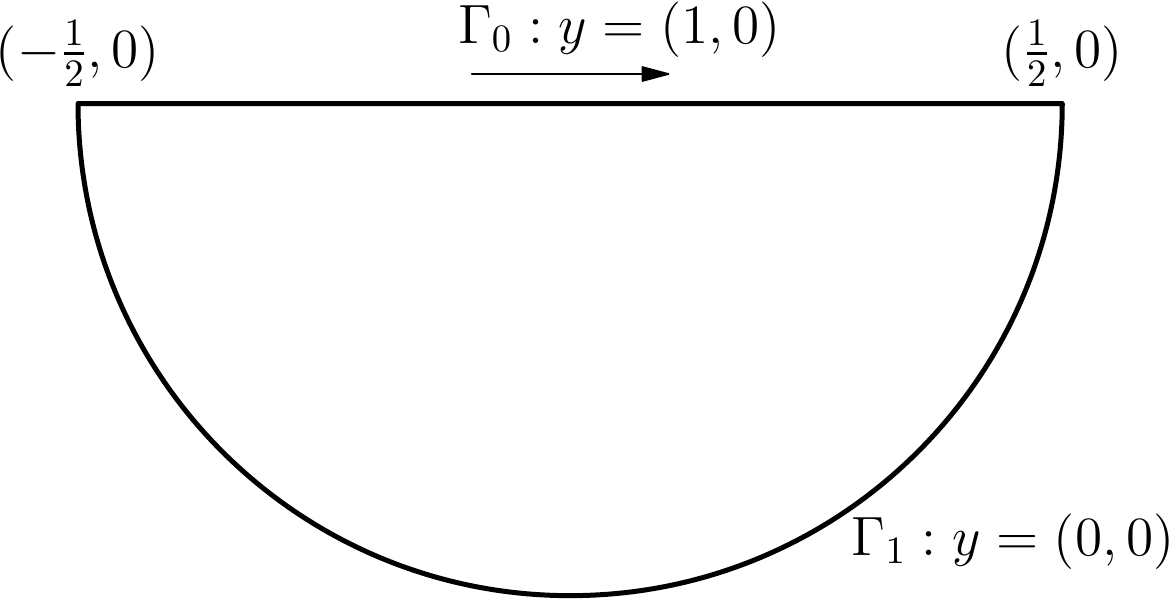}
\caption{Semi-disk geometry.}\label{disk}
\end{center}
\end{figure}

This example has been used in \cite{lemoinemunch} to solve the corresponding steady problem (for which the weak solution is not unique), using again an iterative least-squares strategy. There, the method proved to be robust enough for small values of $\nu$ of the order $10^{-4}$, while standard Newton method failed. Figures \ref{fig:streamlines_disk} depicts the streamlines of steady state solutions corresponding to $\nu^{-1}=500$ and to $\nu^{-1}=i\times 10^{3}$ for $i=1,\cdots, 7$. The figures are in very good agreements with those depicted in \cite{Glowinski_disk}. When the Reynolds number (here equal to $\nu^{-1}$) is small, the final steady state consists of one vortex. As the Reynolds number increases, first a secondary vortex then a tertiary vortex arises, whose size depends on the Reynolds number too. Moreover, according to \cite{Glowinski_disk}, when the Reynolds number exceeds approximatively $6\ 650$, an Hopf bifurcation phenomenon occurs in the sense that the unsteady solution  does not reach a steady state anymore (at time evolves) but shows an oscillatory behavior. We mention that the Navier-Stokes system is solved in  \cite{Glowinski_disk} using an operator-splitting/finite elements based methodology. In particular, concerning the time discretization, the forward Euler scheme is employed.

\begin{figure}[http]
\begin{center}
\includegraphics[scale=0.18]{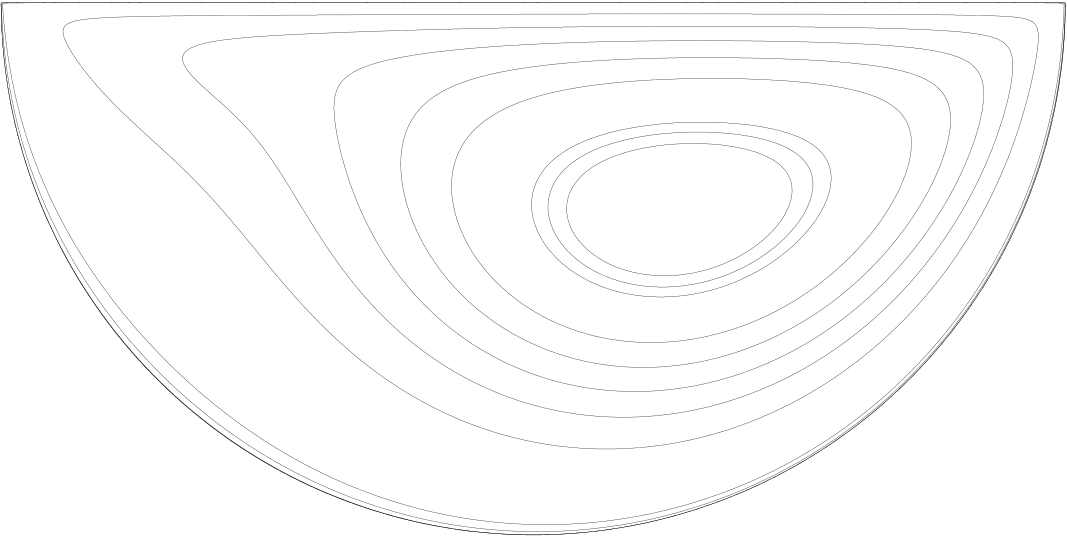} \includegraphics[scale=0.18]{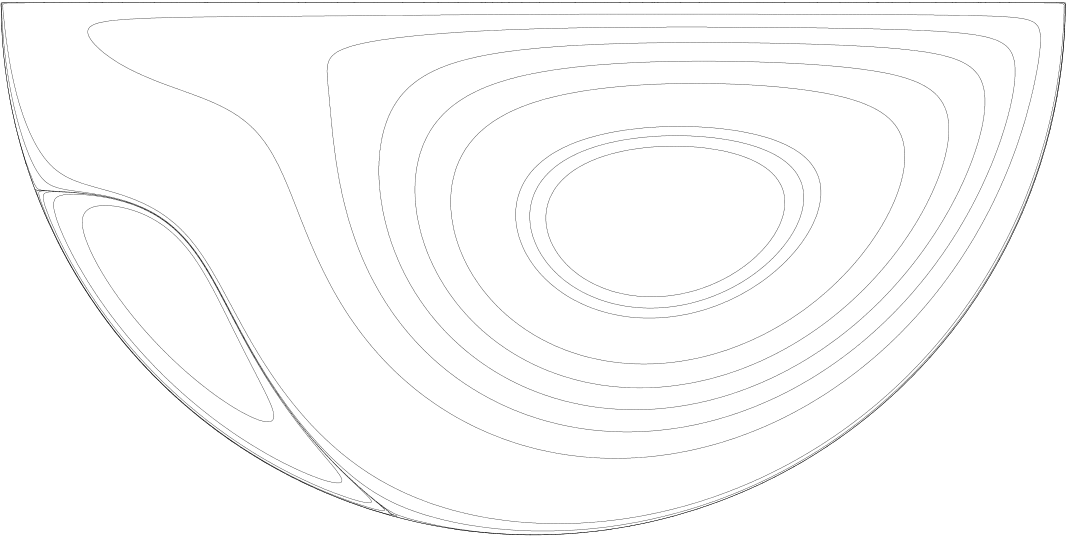}\\
\vspace*{0.3cm}
\includegraphics[scale=0.18]{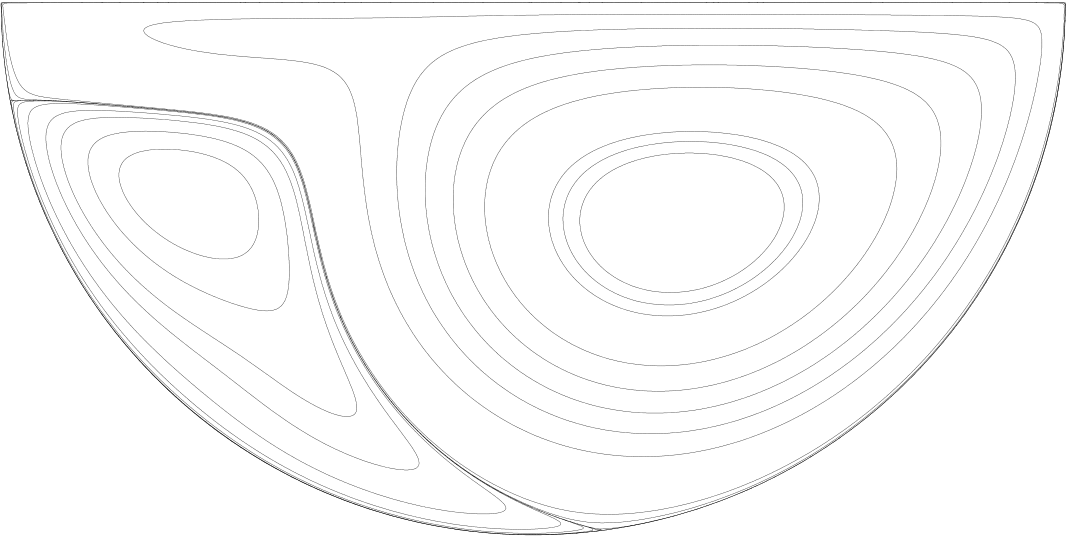} \includegraphics[scale=0.18]{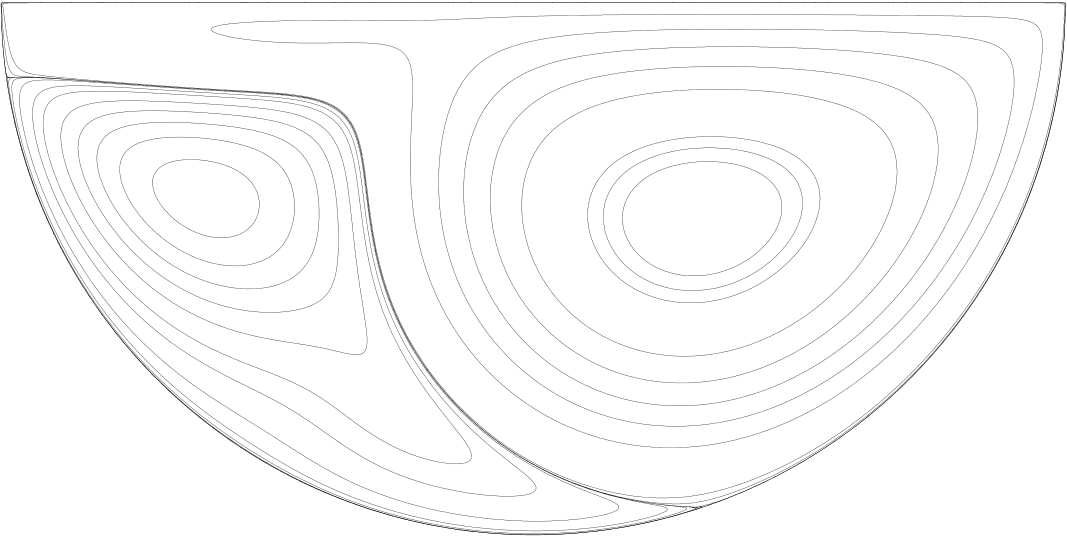}\\
\vspace*{0.3cm}
\includegraphics[scale=0.18]{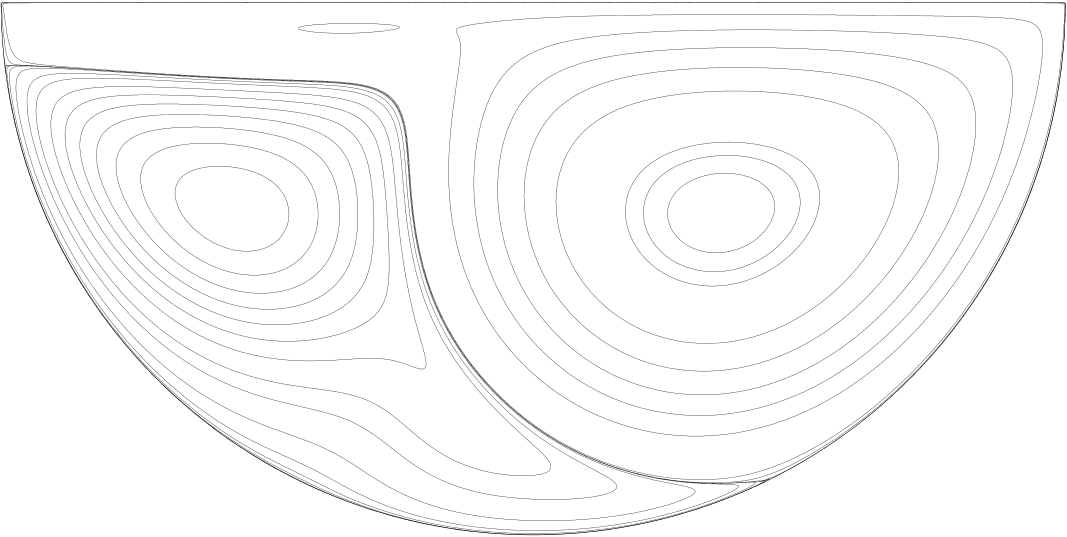} \includegraphics[scale=0.18]{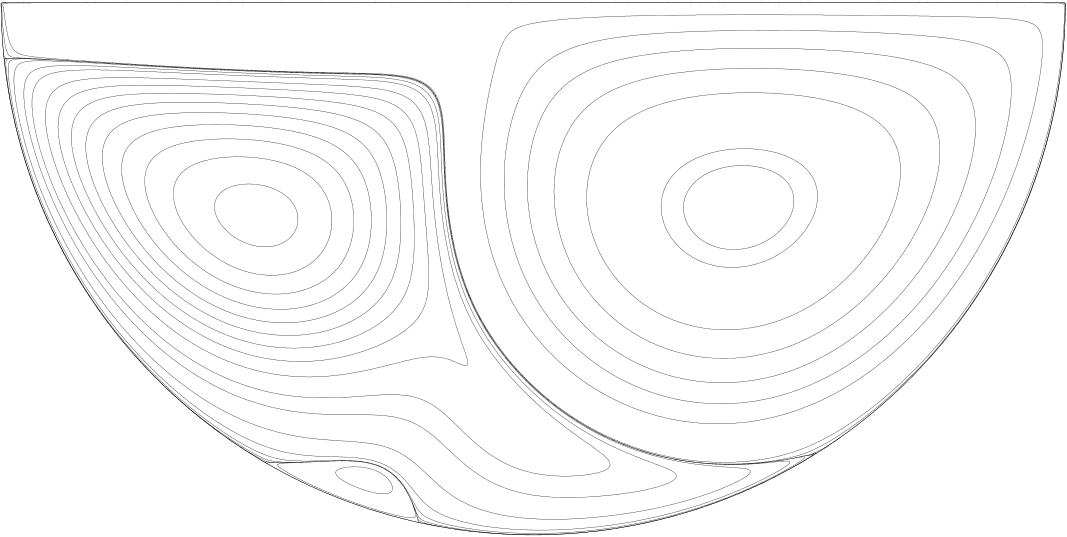}\\
\vspace*{0.3cm}
\includegraphics[scale=0.18]{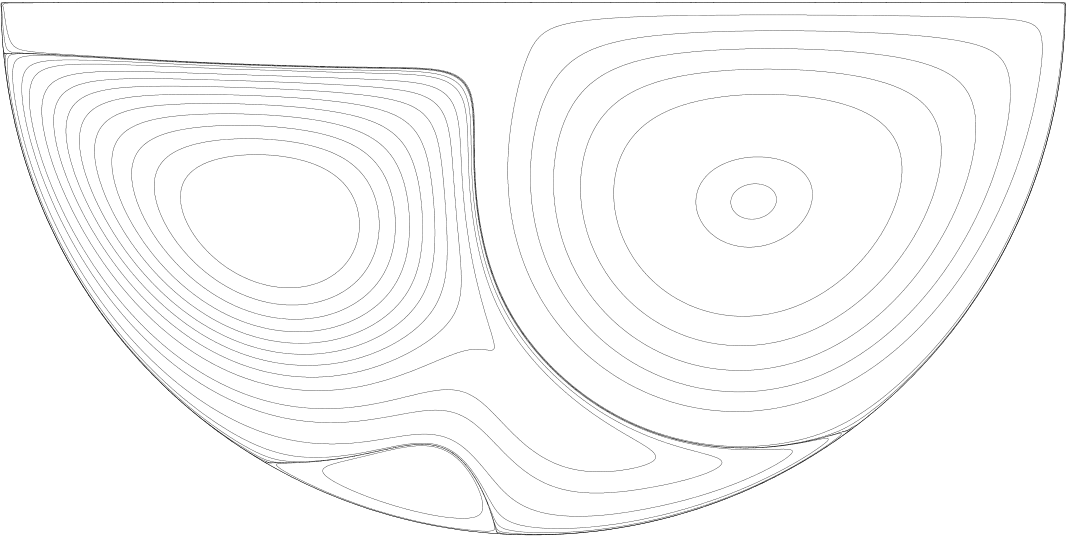} \includegraphics[scale=0.18]{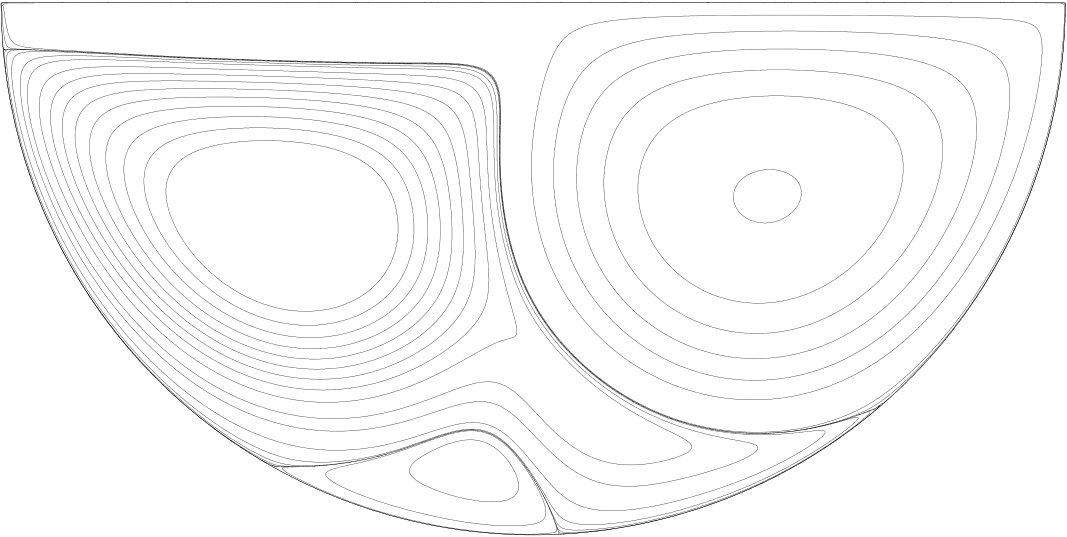}\\
\caption{Streamlines of the steady state solution for $\nu^{-1}=500, 1000,2000,3000,4000,5000,6000$ and $\nu^{-1}=7000$.}\label{fig:streamlines_disk}
\end{center}
\end{figure}

\subsection{Experiments}

We report some numerical results performed with the FreeFem++ package developed at Sorbonne university (see \cite{hecht}).  
Regular triangular meshes are used together with the $\mathbb{P}_2/\mathbb{P}_1$ Taylor-Hood finite element, satisfying the Ladyzenskaia-Babushka-Brezzi condition of stability.  An example of mesh composed of $9\ 063$ triangles is displayed in Figure \ref{diskh}.

\begin{figure}[http]
\begin{center}
\includegraphics[scale=0.6]{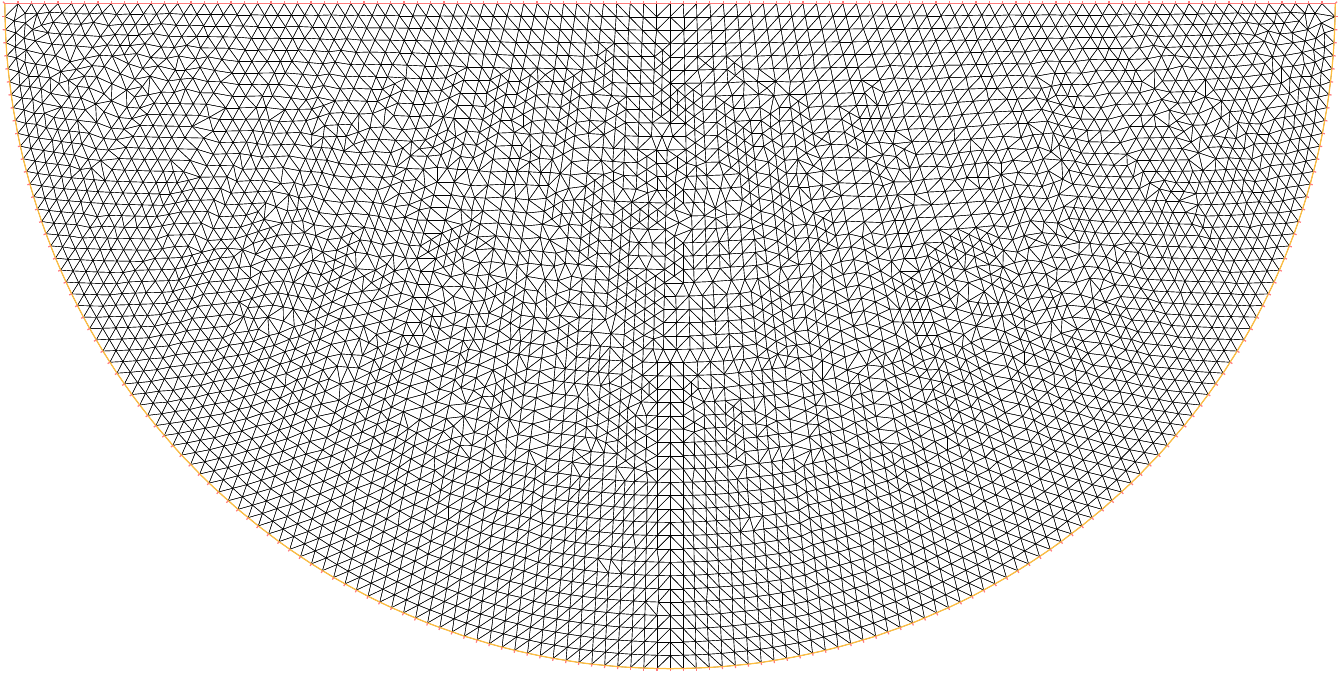}
\caption{A regular triangulation of the semi-disk geometry; $\sharp$triangles = 9\ 064; $\sharp$vertices = 4\ 663; size $h\approx 1.62\times 10^{-2}$.}\label{diskh}
\end{center}
\end{figure}

In order to deeply emphasize the influence of the value of $\nu$ on the behavior of the algorithm described in Section \ref{section-algo}, we consider an initial guess $y_0$ of the sequence $\{y_k\}_{(k>0)}$ independent of $\nu$.  Precisely, we define $y_0$ as the solution of the unsteady Stokes system with viscosity equal to one (i.e. $\overline{\nu}=1$ in \eqref{initialization_sys}) and source term $f\equiv 0$. The initial condition $u_0\in \boldsymbol{H}$ is defined as the solution of $ -\Delta u_0+\nabla p=0, \nabla\cdot u_0=0$ in $\Omega$ and boundary conditions $u_0=g$ on $\Gamma_0$ and $u_0=0$ on $\Gamma_1$. $u_0$ belongs actually to $\boldsymbol{V}$. 

Table \ref{tab:1sur500} and \ref{tab:1sur1000} report numerical values of the sequences $\{\sqrt{2E(y_k)}\}_{(k>0)}$, $\{\lambda_k\}_{(k>0)}$ and $\{\Vert y_k-y_{k-1}\Vert_{L^2(\boldsymbol{V})}/\Vert y_k\Vert_{L^2(\boldsymbol{V})}\}_{(k>0)}$ associated to $\nu=1/500$ and $\nu=1/1000$ respectively and $T=10.$, $f=0$. The tables also display (on the right part) 
the values obtained when the parameter $\lambda_k$ is fixed constant equal to one, corresponding to the standard Newton method. The algorithms are stopped when $\sqrt{2E(y_k)}\leq  10^{-8}$. The triangular mesh of Figure \ref{diskh} for which the discretization parameter $h$ is equal to $1.62\times 10^{-2}$ is employed. The number of degrees of freedom is $23\ 315$. Moreover, the time discretization parameter in $\delta t$ is taken equal to $10^{-2}$.  

For $\nu=1/500$, the optimal $\lambda_k$ are close to one ($\max_k \vert 1-\lambda_k\vert \leq 1/5$), so that the two algorithms produce very similar behaviors. The convergence is observed after 6 iterations. For $\nu=1/1000$, we observe that the optimal $\lambda_k$ are far from one during the first iterates. The optimization of the parameter allows a faster convergence (after 9 iterates) than the usual Newton method. For instance, after 8 iterates, $\sqrt{2E(y_k)}\approx 9.931\times 10^{-11}$ in the first case and  $\sqrt{2E(y_k)}\approx 5.669\times 10^{-5}$ in the second one. In agreement with the theoretical results, we also check that $\lambda_k$ goes to one. Moreover, the decrease of $\sqrt{2E(y_k)}$ is first linear, then (when $\lambda_k$ becomes close to one) quadratic. 

\begin{table}[http]
	\centering
		\begin{tabular}{|c|c|c|c||c|c|}
			\hline
			$\sharp$iterate $k$  			   & $\frac{\Vert y_{k}-y_{k-1}\Vert_{L^2(\boldsymbol{V})}}{\Vert y_{k-1}\Vert_{L^2(\boldsymbol{V})}}$ & $\sqrt{2 E(y_k)}$  & $\lambda_k$  & $\frac{\Vert y_{k}-y_{k-1}\Vert_{L^2(\boldsymbol{V})}}{\Vert y_{k-1}\Vert_{L^2(\boldsymbol{V})}}$ ($\lambda_k=1$) & $\sqrt{2 E(y_k)}$ ($\lambda_k=1$)\tabularnewline
			\hline
			$0$ 	        	& $-$	& $2.690\times10^{-2}$ & $0.8112$ & $-$ & $2.690\times10^{-2}$ \tabularnewline
			$1$              & $4.540\times 10^{-1}$ & $1.077\times10^{-2}$    & $0.7758$ & $5.597\times 10^{-1}$ & $1.254\times10^{-2}$\tabularnewline
			$2$ 	        &  $1.836\times 10^{-1}$ & $3.653\times10^{-3}$ 	  &  $0.8749$ & $2.236\times 10^{-1}$ & $5.174\times10^{-3}$	\tabularnewline
			$3$ 	        & $7.503\times 10^{-2}$ & $7.794\times10^{-4}$ 		& $0.9919$  & $7.830\times 10^{-2}$	& $6.133\times10^{-4}$\tabularnewline
			$4$ 	        & $1.437\times 10^{-2}$ & $2.564\times10^{-5}$ 		 & $1.0006$  & $9.403\times 10^{-3}$ & $1.253\times10^{-5}$	\tabularnewline
			$5$ 	        & $4.296\times 10^{-4}$ & $3.180\times10^{-8}$ 		 & $1.$  & $1.681\times 10^{-4}$ &$4.424\times10^{-9}$	\tabularnewline
			$6$ 	        & $5.630\times 10^{-7}$ & $6.384\times10^{-11}$ 		 & $-$  & $-$ & $-$	\tabularnewline
						\hline
		\end{tabular}
	\caption{$\nu=1/500$; Results for the algorithm (\ref{algo_LS_Y}).}
	\label{tab:1sur500}
	\end{table}
	
\begin{table}[http]
	\centering
		\begin{tabular}{|c|c|c|c||c|c|}
			\hline
			$\sharp$iterate $k$  			   & $\frac{\Vert y_{k}-y_{k-1}\Vert_{L^2(\boldsymbol{V})}}{\Vert y_{k-1}\Vert_{L^2(\boldsymbol{V})}}$ & $\sqrt{2 E(y_k)}$  & $\lambda_k$  & $\frac{\Vert y_{k}-y_{k-1}\Vert_{L^2(\boldsymbol{V})}}{\Vert y_{k-1}\Vert_{L^2(\boldsymbol{V})}}$ ($\lambda_k=1$) &  $\sqrt{2 E(y_k)}$ ($\lambda_k=1$)\tabularnewline
			\hline
			$0$ 	        	& $-$	                      & $2.690\times10^{-2}$               & $0.6344$ &  $-$ & $2.690\times10^{-2}$ \tabularnewline
			$1$          & $5.138\times 10^{-1}$ & $1.493\times10^{-2}$              & $0.5803$ &  $8.101\times 10^{-1}$ &  $2.234\times10^{-2}$\tabularnewline
			$2$ 	        & $2.534\times 10^{-1}$ & $7.608\times10^{-3}$ 	      &  $0.3496$ & $4.451\times 10^{-1}$ & $2.918\times10^{-2}$	\tabularnewline
			$3$ 	        & $1.345\times 10^{-1}$ & $5.477\times10^{-3}$ 		& $0.4025$  & $5.717\times 10^{-1}$ & $5.684\times10^{-2}$\tabularnewline
			$4$ 	        & $1.105\times 10^{-1}$ & $3.814\times10^{-3}$ 		 & $0.5614$  & $3.683\times 10^{-1}$ & $2.625\times10^{-2}$	\tabularnewline
			$5$ 	        & $8.951\times 10^{-2}$ & $2.295\times10^{-3}$ 		 & $0.8680$  & $2.864\times 10^{-1}$ & $1.828\times10^{-2}$	\tabularnewline
			$6$ 	        & $6.394\times 10^{-2}$ & $8.679\times10^{-4}$ 		 & $1.0366$  & $1.423\times 10^{-1}$ & $4.307\times10^{-3}$	\tabularnewline
			$7$ 	        & $1.788\times 10^{-2}$ & $4.153\times10^{-5}$ 		 & $0.9994$  & $6.059\times 10^{-2}$ & $9.600\times10^{-4}$	\tabularnewline
			$8$ 	        & $7.982\times 10^{-4}$ & $9.931\times10^{-8}$ 		 & $0.9999$  & $1.484\times 10^{-2}$ & $5.669\times10^{-5}$	\tabularnewline
			$9$ 	        & $2.256\times 10^{-6}$ & $4.000\times10^{-11}$ 		 & $-$  & $9.741\times 10^{-4}$ & $3.020\times10^{-7}$	\tabularnewline
			$10$ 	        & $-$ & $-$ 		 & $-$  & $4.267\times 10^{-6}$ & $3.846\times10^{-11}$	\tabularnewline
						\hline
		\end{tabular}
	\caption{$\nu=1/1000$; Results for the algorithm (\ref{algo_LS_Y}).}
	\label{tab:1sur1000}
	\end{table}
	
At time $T=10$, the unsteady state solution is close to the solution of the steady Navier-Stokes equation: the last element $y_{k=9}$ of the converged sequence satisfies $\Vert y_{k=9}(T,\cdot)-y_{k=9}(T-\delta t,\cdot)\Vert_{L^2(\Omega)}/\Vert y_{k=9}(T,\cdot)\Vert_{L^2(\Omega)}\approx 1.19\times 10^{-5}$. Figures \ref{fig:streamlines_disk_time} display the streamlines of the unsteady state solution corresponding to $\nu=1/1000$ at time $0,1,2,3,4,5,6$ and $7$ seconds to be compared with the streamlines of the steady solution depicted in Figure \ref{fig:streamlines_disk}.
	
\begin{figure}[http]
\begin{center}
\includegraphics[scale=0.5]{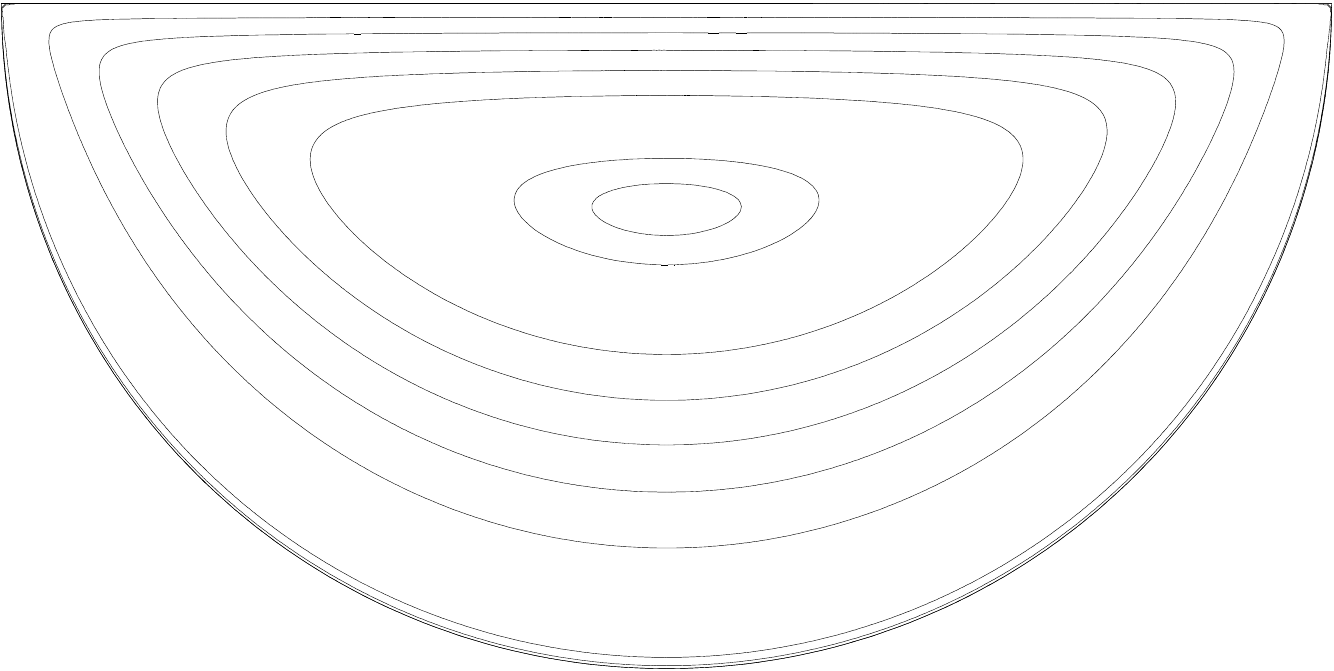} \includegraphics[scale=0.5]{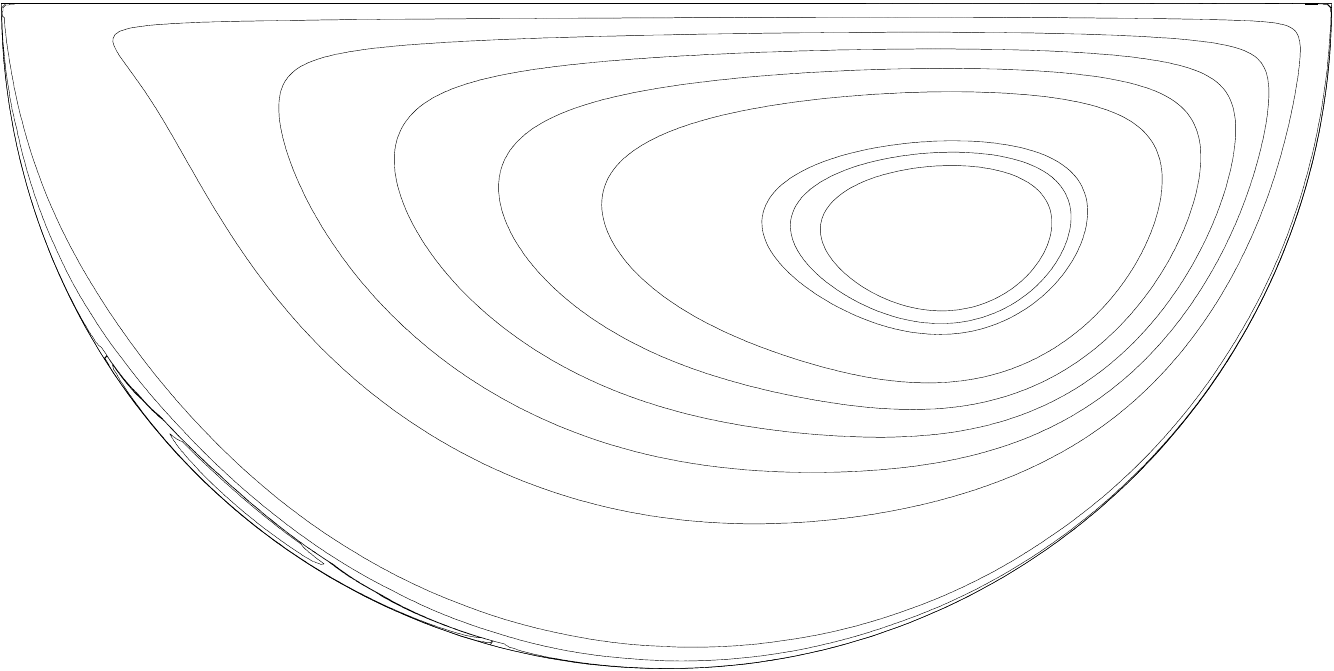}\\
\vspace*{0.3cm}
\includegraphics[scale=0.5]{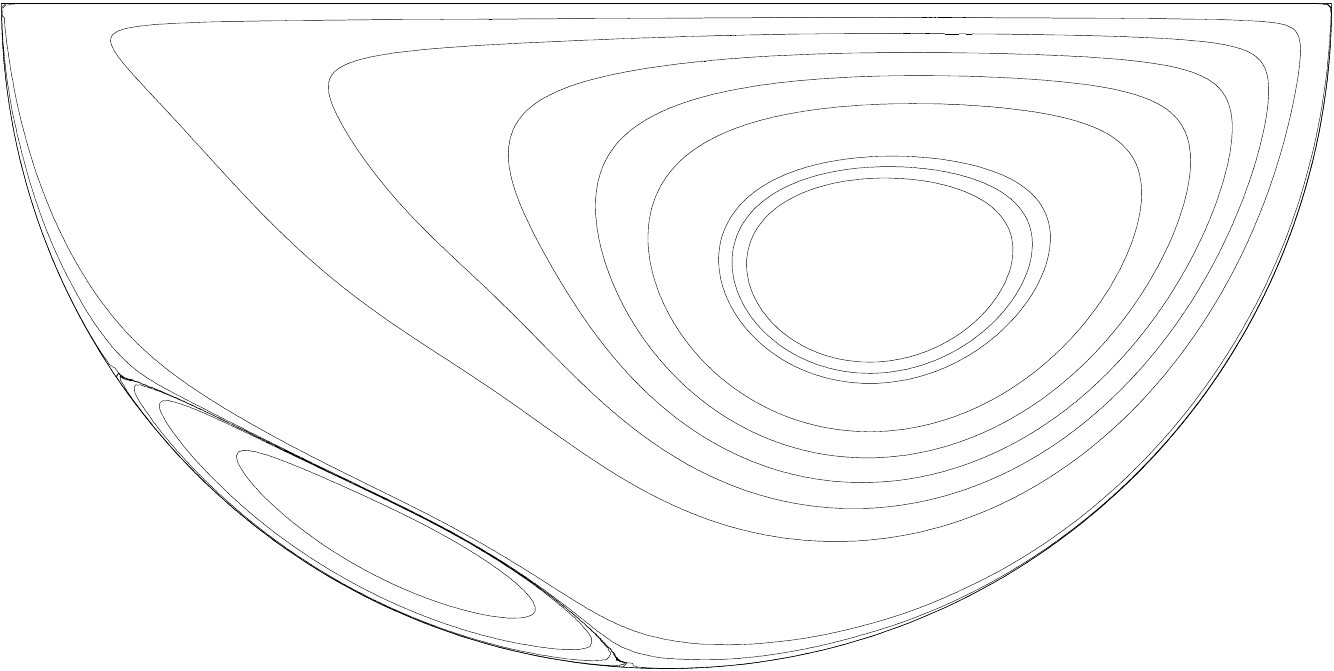} \includegraphics[scale=0.5]{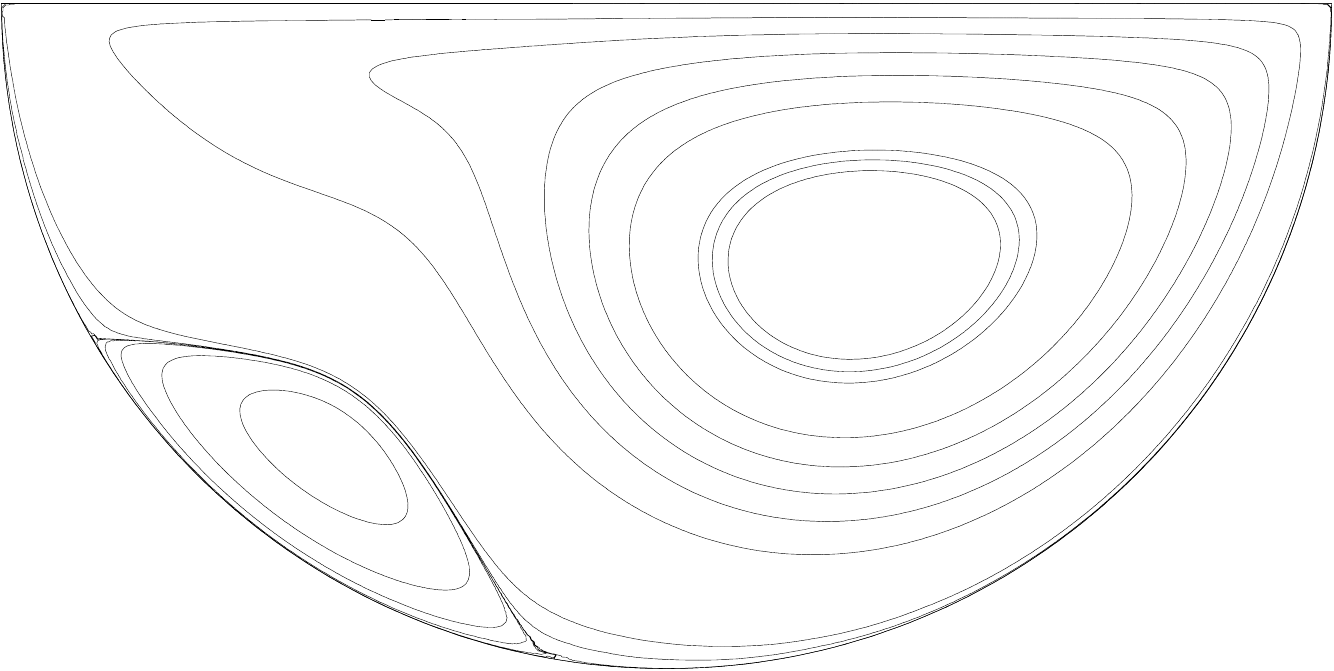}\\
\vspace*{0.3cm}
\includegraphics[scale=0.5]{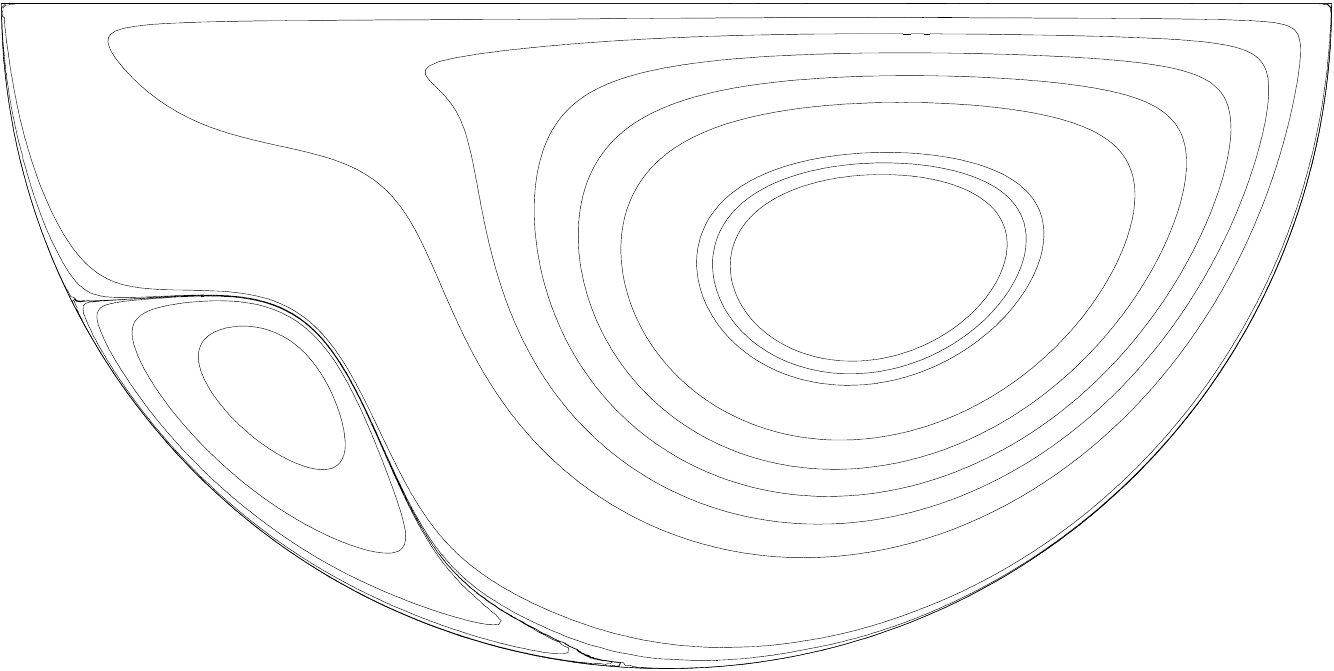} \includegraphics[scale=0.5]{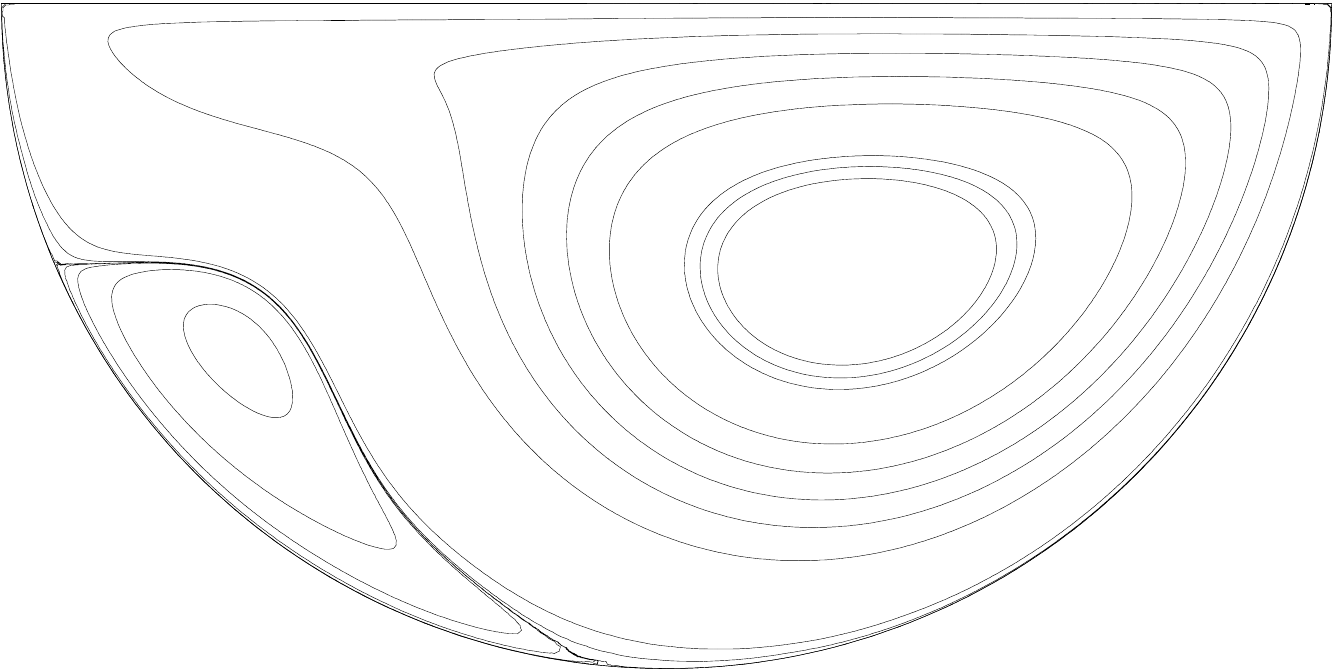}\\
\vspace*{0.3cm}
\includegraphics[scale=0.5]{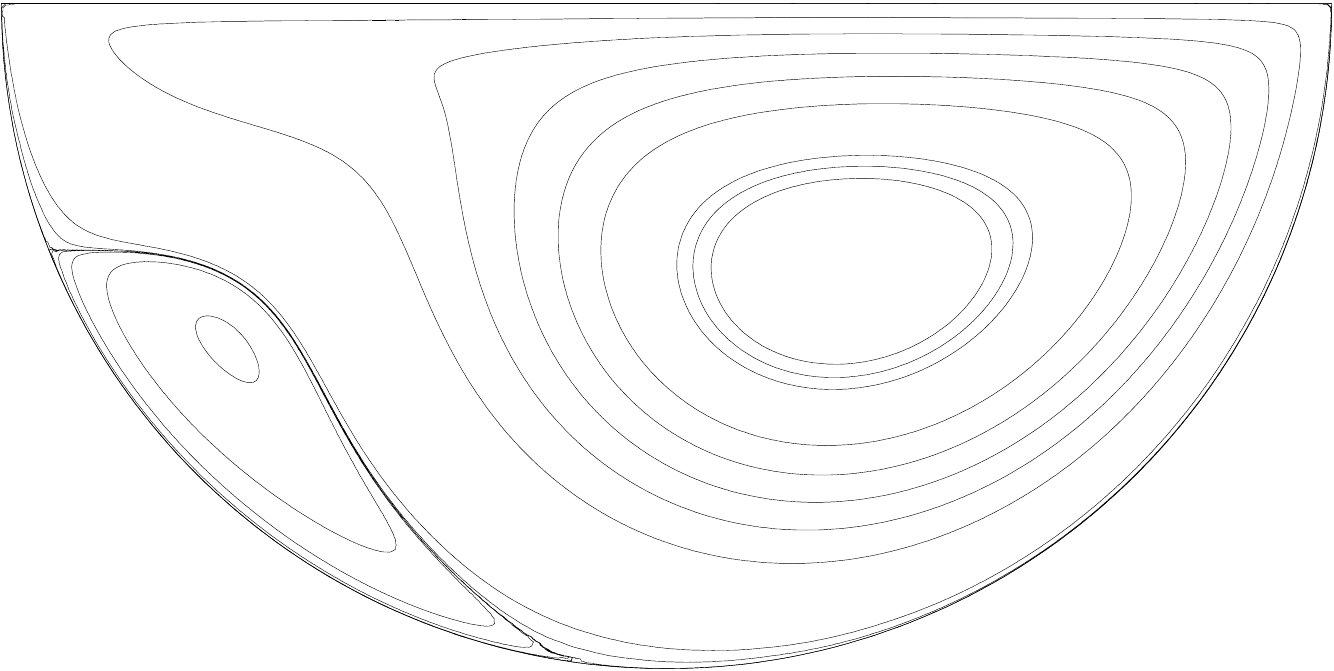} \includegraphics[scale=0.5]{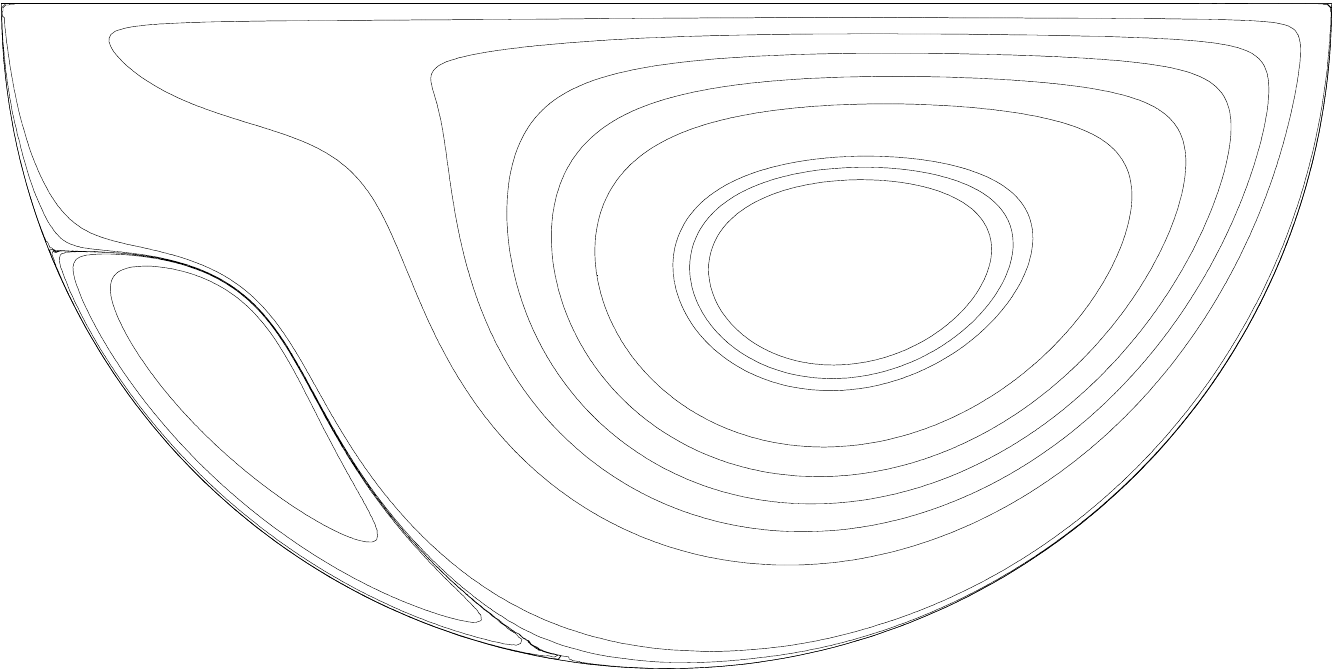}\\
\caption{Streamlines of the unsteady state solution for $\nu^{-1}=1000$ at time $t=i$, $i=0,\cdots,7$s.}\label{fig:streamlines_disk_time}
\end{center}
\end{figure}	
	
For lower values of the viscosity constant, precisely $\nu\leq 1/1100$ approximatively, the initial guess $y_0$ is too far from the zero of $E$ so that we observe the divergence after few iterates of the Newton method (case $\lambda_k=1$ for all $k>0$) but still the convergence of the algorithm described in section \ref{section-algo} (see Table \ref{tab:1sur1100}). The divergence in the case $\lambda_k=1$ is still observed with a refined discretization both in time and space, corresponding to $\delta t=0.5\times 10^{-3}$ and $h\approx 0.0110$ ($19\ 810$ triangles and $10\, 101$ vertices). The divergence of the Newton method suggests that the functional $E$ is not convex far away from the zero of $E$ and that the derivative $E^{\prime}(y)$ takes small values there.  We recall that, in view of the theoretical part, the functional $E$ is coercive and its derivative vanishes only at the zero of $E$. However, the equality $E^{\prime}(y_k)\cdot Y_{1,k}=2E(y_k)$
shows that $E^{\prime}(y_k)$ can be ``small" for ``large" $Y_{1,k}$, i.e. ``large" $y_k$. On the other hand, we observe the convergence (after $3$ iterates) of the Newton method, when initialized with the approximation corresponding to $\nu=1/1000$.
\begin{table}[http]
	\centering
		\begin{tabular}{|c|c|c|c||c|c|}
			\hline
			$\sharp$iterate $k$  			   & $\frac{\Vert y_{k}-y_{k-1}\Vert_{L^2(\boldsymbol{V})}}{\Vert y_{k-1}\Vert_{L^2(\boldsymbol{V})}}$ & $\sqrt{2 E(y_k)}$  & $\lambda_k$  & $\frac{\Vert y_{k}-y_{k-1}\Vert_{L^2(\boldsymbol{V})}}{\Vert y_{k-1}\Vert_{L^2(\boldsymbol{V})}}$ ($\lambda_k=1$) &  $\sqrt{2 E(y_k)}$ ($\lambda_k=1$)\tabularnewline
			\hline
			$0$ 	        	& $-$	                      & $2.691\times10^{-2}$               & $0.6145$ &  $-$ & $2.691\times10^{-2}$ \tabularnewline
			$1$          & $5.241\times 10^{-1}$ & $1.530\times10^{-2}$              & $0.5666$ &  $8.528\times 10^{-1}$ &  $2.385\times10^{-2}$\tabularnewline
			$2$ 	        & $2.644\times 10^{-1}$ & $8.025\times10^{-3}$ 	      &  $0.3233$ & $4.893\times 10^{-1}$ & $3.555\times10^{-2}$	\tabularnewline
			$3$ 	        & $1.380\times 10^{-1}$ & $5.982\times10^{-3}$ 		& $0.3302$  & $7.171\times 10^{-1}$ & $8.706\times10^{-2}$\tabularnewline
			$4$ 	        & $1.115\times 10^{-1}$ & $4.543\times10^{-3}$ 		 & $0.4204$  & $4.849\times 10^{-1}$ & $3.531\times10^{-2}$	\tabularnewline
			$5$ 	        & $9.429\times 10^{-2}$ & $3.221\times10^{-3}$ 		 & $0.5875$  & $1.125\times 10^{0}$ & $3.905\times10^{-1}$	\tabularnewline
			$6$ 	        & $7.664\times 10^{-2}$ & $1.944\times10^{-3}$ 		 & $0.9720$  & $-$ & $1.337\times10^{4}$	\tabularnewline
			$7$ 	        & $5.688\times 10^{-2}$ & $5.937\times10^{-4}$ 		 & $1.022$  & $-$ & $8.091\times 10^{27}$	\tabularnewline
			$8$ 	        & $1.009\times 10^{-2}$ & $1.081\times10^{-5}$ 		 & $0.9998$  & $-$ & $-$	\tabularnewline
			$9$ 	        & $2.830\times 10^{-4}$ & $1.332\times10^{-8}$ 		 & $1.$  & $-$ & $-$	\tabularnewline
			$10$ 	        & $2.893\times 10^{-7}$ & $4.611\times 10^{-11}$ 		 & $-$  & $-$ & $-$	\tabularnewline
						\hline
		\end{tabular}
	\caption{$\nu=1/1100$; Results for the algorithm (\ref{algo_LS_Y}).}
	\label{tab:1sur1100}
	\end{table}

Table \ref{tab:1sur2000} gives numerical values associated to $\nu=1/2000$ and $T=10$. We used a refined discretization: precisely, $\delta t =1/150$ and a mesh composed of  $15\ 190$ triangles, $7\, 765$ vertices ($h\approx 1.343\times 10^{-2}$). The convergence of the algorithm of section \ref{section-algo} is observed after $19$ iterates. In agreement with the theoretical results, the sequence $\{\lambda_k\}_{(k>0)}$ goes to one. Moreover, the variation of the error functional $E(y_k)$ is first quite slow, then increases to be very fast after $15$ iterates. This behavior is illustrated on Figure \ref{nu2000}.
For lower values of $\nu$, we still observed the convergence (provided a fine enough discretization so as to capture the third vortex) with an increasing number of iterates. For instance, $28$ iterates are necessary to achieve $\sqrt{2E(y_k)}\leq 10^{-8}$ for $\nu=1/3000$ and $49$ iterates for $\nu=1/4000$. This illustrates the global convergence of the algorithm. In view of the estimate \eqref{estimateC1}, a quadratic rate is achieved as soon as $\sqrt{E(y_k)}\leq C_1^{-1}$ with here (since $f\equiv 0$)
$$
C_1= \frac{c}{\nu\sqrt{\nu}} \exp\biggl(\frac{c}{\nu^2}\|u_0\|^2_{\boldsymbol{H}}+ \frac{c}{\nu^3} E(y_0) \biggr)
$$ so that $C_1^{-1}\to 0$ as $\nu\to 0$. Consequently, for small $\nu$, it is very likely more efficient (in term of computational ressources) to couple the algorithm with a continuation method w.r.t. $\nu$, in order to reach faster the quadratic regime. This aspect is not addressed in this work and we refer to \cite{lemoinemunch} where this is illustrated in the steady case. 
		
\begin{table}[http]
	\centering
		\begin{tabular}{|c|c|c|c|}
			\hline
			$\sharp$iterate $k$  			   & $\frac{\Vert y_{k}-y_{k-1}\Vert_{L^2(\boldsymbol{V})}}{\Vert y_{k-1}\Vert_{L^2(\boldsymbol{V})}}$ & $\sqrt{2 E(y_k)}$  & $\lambda_k$  \tabularnewline
			\hline
			$0$ 	        	& $-$	                      & $2.691\times10^{-2}$               & $0.5215$   \tabularnewline
			$1$          & $6.003\times 10^{-1}$ & $1.666\times10^{-2}$              & $0.4919$   \tabularnewline
			$2$ 	        & $3.292\times 10^{-1}$ & $9.800\times10^{-3}$ 	      &  $0.1566$  \tabularnewline
			$3$ 	        & $1.375\times 10^{-1}$ & $8.753\times10^{-3}$ 		& $0.1467$   \tabularnewline
			$4$ 	        & $1.346\times 10^{-1}$ & $7.851\times10^{-3}$ 		 & $0.0337$   	\tabularnewline
			$5$ 	        & $5.851\times 10^{-2}$ & $7.688\times10^{-3}$ 		 & $0.0591$   	\tabularnewline
			$6$ 	        & $7.006\times 10^{-2}$ & $7.417\times10^{-3}$ 		 & $0.1196$   	\tabularnewline
			$7$ 	        & $9.691\times 10^{-2}$ & $6.864\times10^{-3}$ 		 & $0.0977$   	\tabularnewline
			$8$ 	        & $8.093\times 10^{-2}$ & $6.465\times10^{-3}$ 		 & $0.0759$   	\tabularnewline
			$9$ 	        & $6.400\times 10^{-2}$ & $6.182\times10^{-3}$ 		 & $0.0968$   \tabularnewline
			$10$ 	        & $6.723\times 10^{-2}$ & $5.805\times10^{-3}$ 		 & $0.1184$   	\tabularnewline
			$11$ 	        & $6.919\times 10^{-2}$ & $5.371\times10^{-3}$ 		 & $0.1630$   	\tabularnewline
			$12$ 	        & $7.414\times 10^{-2}$ & $4.825\times10^{-3}$ 		 & $0.2479$   	\tabularnewline
			$13$ 	        & $8.228\times 10^{-2}$ & $4.083\times10^{-3}$ 		 & $0.3517$   	\tabularnewline
			$14$ 	        & $8.146\times 10^{-2}$ & $3.164\times10^{-3}$ 		 & $0.4746$   	\tabularnewline
			$15$ 	        & $7.349\times 10^{-2}$ & $2.207\times10^{-3}$ 		 & $0.7294$   	\tabularnewline
			$16$ 	        & $6.683\times 10^{-2}$ & $1.174\times10^{-3}$ 		 & $1.0674$   	\tabularnewline
			$17$ 	        & $3.846\times 10^{-2}$ & $2.191\times10^{-4}$ 		 & $1.0039$  	\tabularnewline
			$18$ 	        & $5.850\times 10^{-3}$ & $4.674\times10^{-5}$ 		 & $0.9998$   	\tabularnewline
			$19$ 	        & $1.573\times 10^{-4}$ & $5.843\times10^{-9}$ 		 & $-$  	\tabularnewline
									\hline
		\end{tabular}
	\caption{$\nu=1/2000$; Results for the algorithm (\ref{algo_LS_Y}).}
	\label{tab:1sur2000}
\end{table}

\begin{figure}[http]
\begin{center}
\includegraphics[scale=0.55]{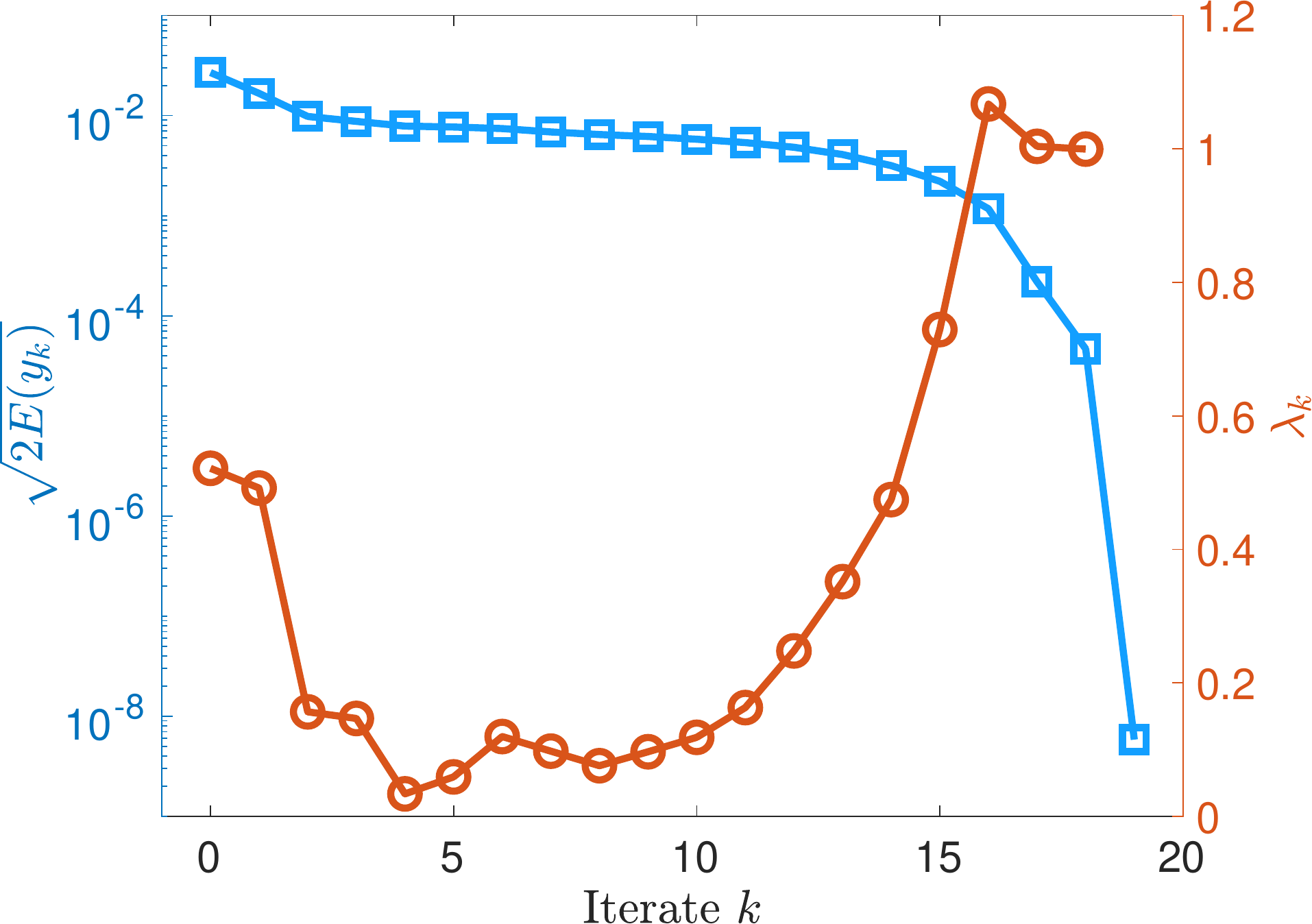}
\caption{Evolution of $\sqrt{2E(y_k)}$ and $\lambda_k$ w.r.t. $k$; $\nu=1/2000$ (see Table \ref{tab:1sur2000}).}\label{nu2000}
\end{center}
\end{figure}


\section{Conclusions and perspectives} \label{section-conclusion}

In order to get an approximation of the solutions of the unsteady Navier-Stokes equation, we have introduced and analyzed 
a least-squares method based on a minimization of an appropriate norm of the equation. In the two dimensional case, considering the weak solution associated to an initial condition in $\boldsymbol{H}\subset L^2(\Omega)^2$ and a source $f\in L^2(0,T,\boldsymbol{V^\prime})$, the least-square functional is based on the $L^2(0,T,\boldsymbol{V^{\prime}})$-norm of the state equation. In the three dimensional case, assuming $T$ small enough, the initial data in $\boldsymbol{V}\subset H^1(\Omega)^3$ and $f\in L^2(0,T;L^2(\Omega)^3)$, the functional is based on the $L^2(0,T;(L^2(\Omega))^3)$-norm of the equation. This leads to a regular solution. 
In both cases, using a particular descent direction, we construct explicitly a minimizing sequence for the functional converging strongly, for any initial guess, to the solution of the Navier-Stokes. Moreover, except for the first iterates, the convergence is quadratic. Actually, it turns out that this minimizing sequence coincides with the sequence obtained from the damped Newton method when used to solves the weak formulation associated to the Navier-Stokes equation. The numerical experiments performed in the two dimensional case illustrate the global convergence of the method and its robustness including for small values of the viscosity constant.

Moreover, the strong convergence of the whole minimizing sequence has been proved using a coercivity type property of the functional, consequence of the uniqueness of the solution. Actually, it is interesting to remark that this property is not necessary, since such minimizing sequence (which is completely determined by the initial term) is a Cauchy sequence. The approach can therefore be adapted to partial differential equations with multiple solutions or to optimization problem involving various solutions. We mention notably the approximation of null controls for (controllable) nonlinear partial differential equation: the source term $f$, possibly distributed over a non-empty set of $\Omega$ is now, together with the corresponding state, an argument of the least-squares functional. The controllability constraint is incorporated in the set $\mathcal{A}$ of admissible pair $(y,f)$. In spite of the non uniqueness of the minimizers, the approach introduced in this work still produces a strongly convergent approximation. We refer to \cite{jerome_irene_arnaud} for the analysis of this approach for surlinear (null controllable) heat equation.

\bibliographystyle{amsplain}

\providecommand{\MRhref}[2]{%
  \href{http://www.ams.org/mathscinet-getitem?mr=#1}{#2}
}
\providecommand{\href}[2]{#2}


\end{document}